\documentclass{amsart}

\usepackage{amsmath,amssymb, xypic, color}
\usepackage{graphicx}
\usepackage{hyperref}

\usepackage{tikz-cd}

\newcommand{\arr}{\longrightarrow}
\newcommand{\R}{\mathbb{R}}
\newcommand{\Z}{\mathbb{Z}}
\newcommand{\C}{\mathbb{C}}

\newcommand{\X}{\mathcal{X}}

\newcommand{\M}{\mathcal{M}}
\newcommand{\alb}{\mathsf{X}}
\newcommand{\xs}{\alb^*}
\newcommand{\xo}{\alb^\omega}
\newcommand{\xmo}{\alb^{-\omega}}
\newcommand{\nuke}{\mathcal{N}}

\newcommand{\img}{\mathop{\mathrm{IMG}}}

\newcommand{\G}{\Gamma}

\newcommand{\lims}[1][G]{\mathcal{J}_{#1}}
\newcommand{\limg}[1][G]{\mathcal{X}_{#1}}
\newcommand{\bim}{\mathfrak{B}}
\newcommand{\til}{\mathcal{T}}
\newcommand{\si}{\mathsf{s}}

\newcommand{\cdim}{\mathop{\mathrm{c.dim}}}
\newcommand{\udim}{\mathop{\mathrm{u.dim}}}
\newcommand{\ldim}{\mathop{\mathrm{l.dim}}}
\newcommand{\diam}{\mathop{\mathrm{Diam}}}

\newcommand{\mapdown}[1]%
{\Big\downarrow\rlap{$\vcenter{\hbox{$\scriptstyle#1$}}$}}

\newtheorem{theorem}{Theorem}[section]
\newtheorem{lemma}[theorem]{Lemma}
\newtheorem{proposition}[theorem]{Proposition}
\newtheorem{corollary}[theorem]{Corollary}

\theoremstyle{definition}
\newtheorem{defi}[theorem]{Definition}
\newtheorem{example}[theorem]{Example}

\title{Conformal dimension and self-similar groups}
\author{Volodymyr Nekrashevych}

\begin{document}
\maketitle

\begin{abstract}
We study the Ahlfors-regular conformal dimension of spaces with locally expanding self-covering maps and, more generally, limit spaces of contracting self-similar groups. We show that the conformal dimension is characterized by combinatorial moduli associated with the Schreier graphs of these groups. We compare these moduli with contraction coefficients of self-similar group actions and with algebras of functions on the limit spaces satisfying asymptotic $p$-energy conditions. We also give conditions on the group when the conformal dimension is equal to the infimum of visual metrics on the self-similarity complex.
\end{abstract}

\tableofcontents

\section{Introduction}

\subsection{Main results}

Conformal dimension was introduced by Pierre Pansu (see, for example,~\cite{pansu:dimensionconforme}) in the study of
boundaries of hyperbolic spaces and groups. The boundary of a hyperbolic group does not carry a distinguished metric; rather, it carries a natural \emph{quasisymmetry} class of \emph{visual metrics}. The conformal dimension of the boundary is defined as the infimum of the Hausdorff dimensions of the metrics in this class. It is a natural invariant with important applications, including
to Cannon's conjecture (see, for example,~\cite{bonkkleiner:cannon}). For an overview of the main properties and applications of conformal dimension, we refer the reader to~\cite{mackaytyson}.

Another setting in which conformal dimension arises naturally is that of locally expanding self-coverings of compact metric spaces, and of their generalizations. Such dynamical systems share many features with the actions of hyperbolic groups on their boundaries, and provide examples of \emph{hyperbolic groupoids}, see~\cite{nek:hyperbolic}. The study of conformal dimension for spaces equipped with a locally expanding dynamical system was initiated by P.~Haissinsky and K.~Pilgrim, see~\cite{haisinskypilgrim}.

\begin{defi}
A covering map of compact metric spaces $f\colon\M\arr\M$ is
\emph{locally expanding} if there exist constants $\delta>0$ and $L>1$ such
that
\[
d(f(x), f(y))\ge Ld(x,y)
\]
whenever $x,y\in\M$ satisfy $d(x,y)<\delta$. 
\end{defi}

For example, the self-map of the circle
$f\colon\R/\Z\arr\R/\Z$ given by $f(x)=2x$ is locally expanding, with
$\delta=1/4$ and $L=2$. Another classical example is the one-sided shift
$f\colon\xo\arr\xo$ on the space $\xo$ of sequences over a finite alphabet
$\alb$, defined by $f(x_1x_2\ldots)=x_2x_3\ldots$.

The definition of a locally expanding covering map can also be formulated in
purely topological terms, without explicit reference to a metric. Namely, the
preceding definition is equivalent to the following one.

\begin{defi}
\label{def:locallyexpanding}
Let $f\colon\M\arr\M$ be a covering map of a compact space $\M$. We say that
$f$ is \emph{locally expanding} if it generates an expansive action of the
monoid generated by $f$: that is, if there exists a neighborhood of the
diagonal, or \emph{entourage}, $U\subset\M\times\M$ such that
\[
(f^n(x),f^n(y))\in U \quad\text{for all } n\ge 0
\]
implies $x=y$.
\end{defi}

Let $U\subset\M\times\M$ be an entourage satisfying the conditions of Definition~\ref{def:locallyexpanding}. Denote by $\ell_U(x, y)$, for $x, y\in\M$ the largest $n$ such that $(f^k(x), f^k(y))\in U$ for all $k=0, 1, \ldots, n$. The function $\ell_U$ measures ``closeness'' of points: the closer $x$ and $y$ are, the larger is $\ell_U$. According to Definition~\ref{def:locallyexpanding}, we have $\ell_U(x, y)=\infty$ if and only if $x=y$.

The function $\ell_U$ is an analogue of the \emph{Gromov product} on the
boundary of a hyperbolic space. As in the construction of visual metrics on
Gromov boundaries, it can be used to construct a \emph{visual metric} on
$\M$. More precisely, for every sufficiently small $\alpha>0$ there exists a
metric $d_\alpha$ on $\M$ such that
\begin{equation}
\label{eq:visual}
C^{-1}e^{-\alpha\ell_U(x,y)}
 \le d_\alpha(x,y)
 \le C e^{-\alpha\ell_U(x,y)}
\end{equation}
for some constant $C>1$ and all $x,y\in\M$. The proof is the same as in the
case of visual metrics on Gromov boundaries. In this context it was
essentially proved by Fried~\cite{fried:naturalmetric}.

The Lipschitz class of $d_\alpha$ is uniquely determined by $\alpha$, and is  independent of the choice of $U$. Thus a locally
expanding covering map $f\colon\M\arr\M$ naturally determines a class of
metrics on $\M$. Any two metrics in this class are quasisymmetric to one
another (indeed, they are \emph{snowflake equivalent}, which is a stronger
equivalence relation).

It is therefore natural to consider the class of all metrics
that are quasisymmetric to a visual metric $d_\alpha$. This class admits a
direct intrinsic description, given in~\cite{haisinskypilgrim}.

An important source of locally expanding covering maps is given by the
restrictions of \emph{hyperbolic} complex rational functions to their Julia
sets. The restriction of the spherical metric to the Julia set is not, in
general, itself a visual metric $d_\alpha$; nevertheless, it is
quasisymmetric to such metrics, and hence belongs to the same framework.

An expanding self-covering $f\colon\M\arr\M$ is determined, up to a topological conjugacy, by its \emph{iterated monodromy group} $\img(f)$. This group is finitely generated and comes equipped with a \emph{self-similar action} on the rooted tree $\xs$ of finite words over a finite alphabet $\alb$. It describes the monodromy action of the fundamental group $\pi_1(\M, t)$ on the preimages of the basepoint $t$ under iterations of $f$, while the self-similarity of the action encodes the dynamics of the map $f$.

A faithful action of a group $G$ on the tree $\xs$ is said to be \emph{self-similar} if for every $g\in G$ and every $v\in\xs$ there exists an element $g|_v\in G$ such that $g(vw)=g(v)g|_v(w)$ for all $w\in\xs$.
The dynamical system $f\colon\M\arr\M$ can then be recovered from the iterated monodromy group as the \emph{limit dynamical system} of this self-similar group. 
We refer the reader to Section~\ref{s:ssgroups} for the relevant definitions.

The limit dynamical system is defined for any \emph{contracting self-similar group}. A group $G$ with a self-similar action on $\xs$ is said to be contracting if there exists a finite set $\nuke\subset G$ (called the \emph{nucleus}) such that for every $g\in G$ there is an integer $n$ such that $g|_v\in\nuke$ for all words $v$ of length at least $n$. If $G$ is finitely generated, then this condition is equivalent to the inequality $\eta_\infty<1$, where 
\begin{equation}
    \label{eq:etainfty}
\eta_\infty=\lim_{n\to\infty}\sqrt[n]{\limsup_{l(g)\to\infty}\frac{\max_{v\in\alb^n}l(g|_v)}{l(g)}},
\end{equation}
where $l(g)$ denotes the length of $g$ with respect to a fixed generating set of $G$.

The iterated monodromy group of an expanding self-covering is contracting. Conversely, if $G$ is a \emph{regular self-similar group}, then its limit dynamical system is an expanding self-covering, and $G$ is the iterated monodromy group of this system. Here regularity means that for every infinite word $x_1x_2\ldots\in\xo$ there exists $n$ such that either $g(x_1x_2\ldots x_n)\ne x_1x_2\ldots x_n$, or $g|_{x_1x_2\ldots x_n}=1$.

If $G$ is not regular, then the associated limit dynamical system is no longer a covering map. It may nevertheless be viewed as a branched covering and realized as a \emph{virtual endomorphism} of an orbispace. This yields a broader class of dynamical systems. For example, \emph{sub-hyperbolic rational functions} fit naturally into this framework. We refer the reader to~\cite{haispilgr:subhyperbolic}, where this larger class is defined and studied from the point of view of conformal dimension.

The main goal of our paper is to show how the conformal gauge (the set of metrics quasi-symmetric to the visual metrics) and the conformal dimension can be studied using techniques from the theory of self-similar groups. In particular, we relate basic aspects of conformal dimension to algebraic properties of self-similar group actions.

Our first result is a version of a theorem of M.~Carrasco~Piaggio~\cite{carrascopiaggio}. The \emph{combinatorial modulus},
introduced in the general setting of compact metric spaces
in~\cite{carrascopiaggio}, admits a particularly natural and simple
formulation in the language of self-similar groups.

Let $G$ be a contracting self-similar group acting on $\xs$, and let $\nuke$ be its nucleus. For each $n\ge 0$, consider the \emph{Schreier graph} $\G_n$ of the action of $G$ on $\alb^n$. Its vertex set is $\alb^n$, and for every $v\in\alb^n$ and $s\in\nuke$ there is an arrow from $v$ to $s(v)$, labeled by $s|_v$. 

A directed path in $\G_n$ is called \emph{essential} if the product of the labels along the path does not belong to $\nuke$. A function
\(\phi\colon E(\G_n)\to [0,\infty)\) is called \emph{admissible} if the sum of its values along every essential path is at least \(1\).
The \emph{combinatorial \(p\)-modulus} of \(\G_n\), for \(p\ge 1\), is defined by
\[M_p(n)=\inf_{\phi}\sum_{e\in E(\G_n)} \phi(e)^p,\]
where the infimum is taken over all admissible functions \(\phi\colon E(\G_n)\to [0,\infty)\). The corresponding
\emph{asymptotic \(p\)-modulus} is
\[\theta_p=\lim_{n\to\infty} \sqrt[np]{M_p(n)}.\]

We show that for any $1\le p<q\le\infty$ we have
\begin{equation}
    \label{eq:etapetaq}
\theta_p\ge\theta_q,\qquad \frac{\theta_p}{|\alb|^p}\le \frac{\theta_q}{|\alb|^q},
\end{equation}
and prove the following version of the main result of~\cite{carrascopiaggio}, see Theorem~\ref{th:modulus}.

\begin{theorem}
\label{theorem1}
The Ahlfors-regular conformal dimension of the limit space of $G$ is equal to $\inf\{p\colon\theta_p<1\}$.
\end{theorem}

We cannot apply the main theorem of~\cite{carrascopiaggio} directly, because its hypotheses need not hold in our setting (we do not assume that the limit space is locally connected). On the other hand, limit spaces of self-similar groups have additional structure, which allows us to use techniques from the theory of self-similar groups and leads to a somewhat simpler proof. In particular, we show that, if \(\theta_p<1\), then one can construct a metric of dimension at most \(p\) by a direct generalization of the visual metric, using a notion of a \emph{quasi-homomorphism} from a free monoid to \([0,\infty)\).

Our second main result is a version of a theorem of M.~Bourdon and B.~Kleiner~\cite{bourdonkleiner}. We formulate it here in the setting of expanding covering maps (we prove it in a more general setting of limit spaces of contracting groups).

Let $F\colon\M\arr\M$ be a locally expanding covering map. By definition, the iterated monodromy group $\img(F)$ acts by 
permutations on the backward orbit
\[T_t=\bigsqcup_{n\geq 0} F^{-n}(t)\]
of a basepoint $t\in\M$. Denote by $\mathcal{D}_p$ the set of all
continuous functions $u\colon \M \arr \C$ such that
\[
        \sum_{z\in T_t} |u(z)-u(g(z))|^p <\infty
\]
for every $g\in \img(F)$. Similarly, let $\mathcal{D}_{p,\lambda}$ be the
set of all continuous functions $u\colon \M \arr \C$ such that, for every
$g\in \img(F)$, there exists a constant $C>0$ with
\[
        \sum_{z\in F^{-n}(t)} |u(z)-u(g(z))|^p \leq C\lambda^n
\]
for all $n\geq 0$. Both $\mathcal{D}_p$ and $\mathcal{D}_{p,\lambda}$ are
$*$-subalgebras of $C(\M)$.

We prove the following theorem.

\begin{theorem}
\label{th:gradientalgebras}
The Ahlfors-regular conformal dimension of $\M$ is equal to the infimum of
all $p$ for which $\mathcal{D}_p$ is dense in $C(\M)$. Moreover, if $p$ is
larger than the Ahlfors-regular conformal dimension, then the infimum of all
$\lambda$ for which $\mathcal{D}_{p,\lambda}$ is dense in $C(\M)$ is equal
to $\theta_p$.
\end{theorem}

Theorem~\ref{th:gradientalgebras} admits a reformulation in the spirit of
A.~Connes' ``quantized calculus''~\cite{connes:noncomg} and a theorem of A.~Connes and D.~Sullivan on the Hausdorff dimension
of quasicircles~\cite[Proposition~14, p.~323]{connes:noncomg}.

Consider the Hilbert space
\[
        H=\ell^2\left(\bigsqcup_{n\geq 0} F^{-n}(t)\right).
\]
The action of $\img(F)$ on $T_t$ defines a unitary representation $\pi$ of
$\img(F)$ on $H$. The algebra $C(\M)$ has a natural representation $\rho$ on
the same Hilbert space, given by
\[
        \rho(u)(\delta_z)=u(z)\delta_z .
\]
With this notation, $\mathcal{D}_p$ is precisely the set of all
$u\in C(\M)$ such that, for every $g\in \img(F)$, the commutator
\([\rho(u),\pi(g)]\)
belongs to the Schatten ideal $\mathcal{L}^p(H)$.

We also relate the asymptotic combinatorial moduli
\(\theta_p\) to the \emph{contraction coefficients} \(\eta_p\). The
latter are defined by
\[\eta_p=\lim_{n\to\infty}\sqrt[n]{\limsup_{l(g)\to\infty}\frac{\left(\sum_{v\in\alb^n}l(g|_v)^p\right)^{1/p}}{l(g)}},\]
with the usual modification when \(p=\infty\), as in~\eqref{eq:etainfty} above.

The contraction coefficients satisfy the analogous monotonicity
inequalities
\[\eta_p\ge\eta_q,\qquad\frac{\eta_p}{|\alb|^p}\le\frac{\eta_q}{|\alb|^q}\]
for all $0<p<q\le\infty$.

We prove that
\begin{equation}
    \label{eq:etatheta}
    \eta_p\le\theta_p
\end{equation}
for all $1\le p\le\infty$. Moreover, we have $\eta_\infty=\theta_\infty$. 

In fact, for many self-similar groups one has \(\eta_p=\theta_p\) for all \(p\). On the other hand, the inequality
\(\eta_1<1\) is the crucial ingredient in the proof that the Grigorchuk group has intermediate growth, whereas in this case
\(\theta_1=1\). It is therefore an interesting problem to determine for which groups the inequality~\eqref{eq:etatheta} is in fact an
equality.

\subsection{Structure of the paper}

Section~``Self-similar groups'' recalls the main definitions and basic results concerning contracting self-similar groups and their limit spaces.

The following section, ``Metrics on $\lims$,'' introduces the class $\mathcal{AR}$ of Ahlfors-regular metrics on the limit space $\lims$ that are quasisymmetric to visual metrics. It also discusses visual metrics and measures on $\lims$ associated with \emph{$G$-invariant quasi-morphisms} $\xs \arr [0,\infty)$.

Section~4 is devoted to the proof of Theorem~\ref{theorem1}. We introduce several asymptotically equivalent versions of the combinatorial modulus $M_p(n)$. As a consequence of the proof, we obtain that the Ahlfors-regular conformal dimension $\cdim(G)$ is equal to the infimum of the Hausdorff dimensions of the visual metrics associated with $G$-invariant quasi-morphisms. Thus, the study of conformal dimension is, in a sense, equivalent to the study of $G$-invariant quasi-morphisms and their asymptotic properties.

Section~5 studies the contraction coefficients $\eta_p$. The infimum of the set
\(\{p : \eta_p < 1\}\)
is denoted by $\ldim(G)$. Since $\eta_p \leq \theta_p$, the conformal dimension $\cdim(G)$ is greater than or equal to $\ldim(G)$. We show that $\ldim(G)$ gives an upper bound $\exp(n^{\ldim(G)+\epsilon})$ on the growth of \emph{portraits} of elements of the group. In particular, the growth rate of $G$ is bounded above by $n^{\ldim(G)+\epsilon}$ for every $\epsilon>0$. This estimate is vacuous when $\ldim(G)\geq 1$, but there are examples of groups, such as the Grigorchuk group, for which it yields a subexponential growth bound. Moreover, by results of A.~Erschler and T.~Zheng~\cite{erschlerzheng:grigorchukgroup}, this bound is optimal for the Grigorchuk group.

In Section~\ref{ss:thurstonpmap}, we also discuss a result of~\cite{mattebonnekzheng:liouvills} that provides a method for obtaining lower bounds on $\cdim(G)$ using the contraction coefficients $\eta_p$. This method is modeled on Thurston's characterization of post-critically finite rational functions.

Section~6 is devoted to the proof of Theorem~\ref{th:gradientalgebras}.

In the final section, we study the infimum of the Hausdorff dimensions of visual metrics on the limit space $\lims$. We call this quantity the \emph{uniform dimension} and denote it by $\udim(G)$. We prove that
\[
\udim(G)=\frac{\log|\alb|}{-\log\theta_\infty},
\]
and show that $\eta_\infty=\theta_\infty$. We also give conditions on contracting self-similar groups that imply
\[
\udim(G)=\ldim(G)=\cdim(G).
\]
In particular, we show that this equality holds for all abelian contracting self-similar groups.

\subsection{Acknowledgments}  The author is grateful to Kevin Pilgrim, Dylan Thurston, and Misha Hlushchanka for discussions on the topics of the paper and useful remarks. The author was supported by NSF grants DMS-2204379 and DMS-2604376.

\subsection{AI declaration} AI was used only for proofreading.

\section{Self-similar groups}
\label{s:ssgroups}

\subsection{Main definitions}

We study metrics on spaces with locally expanding covering maps using symbolic dynamics, and in particular \emph{iterated monodromy groups}. We begin with a brief review of the relevant constructions and definitions. For more details and for proofs, see~\cite{nek:book,nek:dyngroups}.

\begin{defi}
    A \emph{self-similar group} is a pair $(G, \bim)$, where $G$ is a discrete group, and $\bim$ is a set together with commuting left and right actions of $G$ on it (such sets are called $G$-\emph{bisets}), such that the right action is free and has finitely many orbits. Here an action $(x, g)\mapsto x\cdot g$ for $x\in\bim$ and $g\in G$ is said to be free if $x\cdot g=x$ implies $g=1$.

    Two self-similar groups $(G_i, \bim_i)$, for $i=1, 2$, are called \emph{equivalent} if there exists an isomorphism $\phi\colon G_1\arr G_2$ of groups and a bijection $F\colon\bim_1\arr\bim_2$ such that $F(g_1\cdot x\cdot g_2)=\phi(g_1)\cdot F(x)\cdot\phi(g_2)$ for all $g_1, g_2\in G_1$ and $x\in\bim_1$.
\end{defi}

Suppose that $\bim_1$ and $\bim_2$ are $G$-bisets. Consider the direct product $\bim_1\times\bim_2$ and take the set, denoted $\bim_1\otimes\bim_2$, of the $G$-orbits for the action $(x_1, x_2)\mapsto (x_1\cdot g, g^{-1}\cdot x_2)$ of $G$ on it. We denote by $x_1\otimes x_2$ the orbit of $(x_1, x_2)$. We then have $x_1\cdot g\otimes x_2=x_1\otimes g\cdot x_2$. The set $\bim_1\otimes\bim_2$ is naturally a $G$-biset with respect to the actions $g_1\cdot(x_1\otimes x_2)\cdot g_2=(g_1\cdot x_1)\otimes (x_2\cdot g_2)$. It is also easy to see that $(\bim_1\otimes\bim_2)\otimes\bim_3$ is naturally isomorphic to $\bim_1\otimes(\bim_2\otimes\bim_3)$.

In particular, for any $G$-biset $\bim$, we have the biset $\bim^*$ equal to the disjoint union of the bisets $\bim^{\otimes n}$ for $n\ge 0$. Here $\bim^{\otimes 0}$ is defined as $G$ together with the left and right actions on itself by multiplication. 

The biset $\bim^*$ is also a semigroup with respect to the multiplication $v_1\otimes v_2\in\bim^{\otimes (n_1+n_2)}$ for $v_1\in\bim^{\otimes n_1}$ and $v_2\in\bim^{\otimes n_2}$. (If $n_1$ or $n_2$ are equal to $0$, then the product is the result of the corresponding action of $G$.)

\begin{defi}
Let $(G, \bim)$ be a self-similar group. A \emph{basis} of $\bim$ is a set $\alb\subset\bim$ intersecting every orbit of the right action of $G$ on $\bim$ exactly once.
\end{defi}

If $\alb$ is a basis, then every element of $\bim$ is uniquely written in the form $x\cdot g$ for $x\in\alb$ and $g\in G$. The set $\alb^{\otimes n}$ is a basis of $\bim^{\otimes n}$. We denote $x_1\otimes x_2\otimes\cdots\otimes x_n$ for $x_i\in\alb$ just by $x_1x_2\ldots x_n$, and denote $\alb^{\otimes n}$ by $\alb^n$. We also denote $\xs=\bigcup_{n\ge 0}\alb^n$, where $\alb^0$ consists of the \emph{empty word} (the identity element of $G=\bim^{\otimes 0}$).

The set of right orbits $\bim^*/G$ has a natural structure of a rooted tree, where the orbit of $v\in\bim^{\otimes n}$ is connected to the orbits of the elements of the form $v\otimes x$ for $x\in\bim$. In other words, the tree is induced on $\bim^*/G$ by the Hasse diagram of the left divisibility relation on $\bim^*$. The left $G$-action on $\bim^*$ induces an action of $G$ on the tree $\bim^*/G$ by automorphisms. We say that the self-similar group is \emph{faithful} if this action is faithful. Otherwise, if $K$ is the kernel of the action of $G$ on $\bim^*/G$, then the group $G/K$, more precisely, the self-similar group $(G/K, \bim/K)$ is called the \emph{faithful quotient} of $(G, \bim)$.  Here $\bim/K$ is the set of orbits of the right action of $K\le G$ on $\bim$. It is naturally a $G/K$-biset.

If we identify the tree $\bim^*/G$ with the set $\xs$ by the map sending an orbit $vG$ to its representative in the basis $\xs$, then we get an action of $G$ on $\xs$. We will denote it by $g\colon v\mapsto g(v)$, distinguishing it from the left action $g\colon v\mapsto g\cdot v$ of $G$ on $\bim^*$.

For any $g\in G$ and $v\in\xs$, we have $g\cdot v=g(v)\cdot h$ for a unique $h\in G$. We denote $h=g|_v$ and call it the \emph{section} of $g$ in $v$. We have
\begin{equation}
\label{eq:sectionaction}
g(v_1v_2)=g(v_1)g|_{v_1}(v_2)
\end{equation}
for all $g\in G$ and $v_1, v_2\in\xs$, and
\begin{equation}
\label{eq:sections}
g|_{v_1v_2}=g|_{v_1}|_{v_2},\qquad (g_1g_2)|_v=g_1|_{g_2(v)}g_2|_v.
\end{equation}
for all $v_1, v_2, v\in\xs$ and $g_1, g_2, g\in G$.

Let $S$ be a generating set of $G$. If we know, for every $s\in S$, the action $x\mapsto s(x)$ and the sections $s|_x$ for all $x\in\alb$, then we will know the action of $G$ on $\xs$, using~\eqref{eq:sectionaction} and~\eqref{eq:sections}.
We will often write this information compactly as equalities (rather, assignments)
\[s=\sigma(s|_x)_{x\in\alb}\]
where $\sigma$ is the permutation of $\alb$ defined by $s$. For example, if $\alb=\{1, 2\}$, and we write
\[a=(12)(1, a^3),\]
then it means that 
\[a(1w)=2w,\qquad a(2w)=1a^3(w)\]
for all $w\in\xs$. Note that this rule uniquely determines the action of $a$ on $\xs$. We call such descriptions of the generators of $G$ the \emph{wreath recursion}.

\begin{defi}
    A \emph{topological virtual endomorphism} is a pair of continuous maps $f, \iota\colon\M_1\arr\M$, where $\M_1, \M$ are path-connected topological spaces, and $f$ is a finite degree covering map.
\end{defi}

In particular, if $f\colon\M\arr\M$ is a finite degree covering map, then we consider it as a topological virtual endomorphism for $\iota\colon\M\arr\M$ equal to the identity map. Similarly, if $f\colon\M_1\arr\M$ is a covering map, where $\M_1$ is a subset of $\M$, then we take $\iota\colon\M_1\arr\M$ to be the identity embedding.

Let $f, \iota\colon\M_1\arr\M$ be a topological virtual endomorphism such that $\M$ is path connected. Choose a point $t\in\M$. Define the $\pi_1(\M, t)$-biset $\bim_{f, \iota}$ as the set of pairs $(z, [\alpha])$, where $z\in \M_1$ is such that $f(z)=t$, and $[\alpha]$ is the homotopy class of a path in $\M$ from $t$ to $\iota(z)$. The right $\pi_1(\M, t)$-action is given, for a loop $\gamma$ at $t$, by
\[(z, [\alpha])\cdot [\gamma]=(z, [\alpha\gamma]),\]
where we concatenate paths as functions: in the path $\alpha\gamma$, the path $\gamma$ is passed before $\alpha$.
The left action is given by
\[[\gamma]\cdot(z, [\alpha])=(y, [\gamma_z\alpha]),\]
where $\gamma_z$ is the lift of $\gamma$ by $f$ to a path starting in $z$, and $y$ is the end of $\gamma_z$.

One can show that $(\pi_1(\M, t), \bim_{f, \iota})$ is a self-similar group. Its faithful quotient is called the \emph{iterated monodromy group} of $f, \iota\colon\M_1\arr\M$. We denote it $(\img(f, \iota), \bim_{f, \iota})$, slightly abusing the notation for the biset.

If $\iota\colon\M_1\arr\M$ is the identity embedding of a subset $\M_1\subseteq\M$ into $\M$, then we denote $\img(f)=\img(f, \iota)$, and call it the \emph{iterated monodromy group} of the partial self-covering map $f\colon\M_1\arr\M$.

\begin{example}
\label{ex:angledoubling}
Consider the angle doubling map $f\colon\R/\Z\arr\R/\Z$ given by $f(x)=2x$, seen as a topological virtual endomorphism with $\iota$ equal to the identity map. Choose $t=0$ as the basepoint. We have $f^{-1}(0)=\{0, 1/2\}$. The biset $\bim_f$ consists of pairs $(z, [\alpha])$, where $z\in\{0, 1/2\}$, and $\alpha$ is a homotopy class of a path from $0$ to $z$. Since $z$ is uniquely determined by $\alpha$, we omit it. The homotopy class $[\alpha]$ is determined by the winding number $m\in\frac 12\Z$. Namely, a winding number $m$ corresponds to the homotopy class of the image in $\R/\Z$ of a path from $0$ to $m$ in $\R$. Similarly, the fundamental group of $\R/\Z$ is naturally identified with $\Z$. We get the following actions of $\Z=\pi_1(\R/\Z)$ on $\bim_f$:
\[[m]\cdot [n]=[m+n],\qquad [n]\cdot [m]=\left[\frac n2+m\right],\]
where $[m]\in\bim_f$ and $[n]\in\pi_1(\R/\Z)$ are elements associated with the corresponding numbers $m\in\frac 12\Z$ and $n\in\Z$.

We usually denote the elements $[0], [1/2]$ of $\bim_f$ by $0$ and $1$, respectively. Then, for the generator $a=[1]$ of $\pi_1(\R/\Z)$, we get the recurrent formulae
\[a\cdot 0=1,\qquad a\cdot 1=0\cdot a,\]
i.e., by the wreath recursion
\[a=(01)(1, a).\]
It follows that, in this notation, $a(x_0x_1\ldots x_{n-1})=y_0y_1\ldots y_{n-1}$ if and only if $\sum_{i=0}^{n-1}y_i2^i=1+\sum_{i=0}^{n-1}x_i2^i\pmod{2^n}$, i.e., $a$ acts as adding 1 to an integer written in base 2.
\end{example}

\begin{example}
\label{ex:basilicaimgcomput}
Consider the polynomial $f(z)=z^2-1$. Denote $\M_0=\C\setminus\{0, -1\}$, and $\M_1=f^{-1}(\M_0)=\C\setminus\{0, 1, -1\}$. Consider the topological virtual endomorphism \[f, \iota\colon \M_1\arr\M_0,\]
where $\iota$ is the identity embedding.

\begin{figure}
    \centering
    \includegraphics{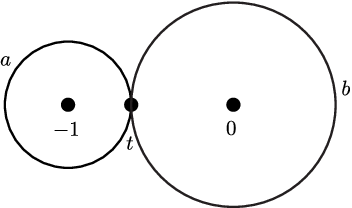}
    \caption{Generators of $\img(z^2-1)$}
    \label{fig:basilicaab}
\end{figure}

The fundamental group of $\M_0$ is generated by loops $a$ and $b$ going around $-1$ and $0$, as shown in Figure~\ref{fig:basilicaab}. We choose the fixed point $t=\frac{1-\sqrt{5}}2$ as the basepoint for the fundamental group. We have $f^{-1}(t)=\{t, -t\}$.

Consider the elements $x_0$ and $x_1$ of $\bim_{f, \iota}$ defined by the path from $t$ to $-t$ shown in Figure~\ref{fig:basilicaimg} and the trivial path at $t$, respectively.

\begin{figure}
    \centering 
    \includegraphics{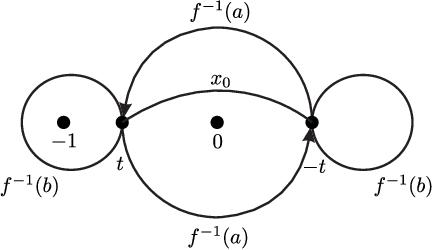}
    \caption{Computation of $\img(z^2-1)$}
    \label{fig:basilicaimg}
\end{figure}

We have the following relations in $\bim_{f, \iota}$. We replace here $x_0, x_1$ by $0$ and $1$, respectively.
\begin{alignat*}{2}
a\cdot 0&=1,&\qquad a\cdot 1&=0\cdot b,\\
b\cdot 0&=0,&\qquad b\cdot 1&=1\cdot a.
\end{alignat*}
Equivalently, $\img(z^2-1)$ is generated by the wreath recursion
\[a=(01)(1, b),\qquad b=(1, a).\]
\end{example}

\begin{defi}
    We say that a self-similar group $(G, \bim)$ is \emph{contracting} if there exists a finite set $\nuke\subset G$ such that for every $g\in G$ there exists $n\ge 0$ such that $g|_v\in\nuke$ for all $v\in\xs$ of length at least $n$.

    The smallest set $\nuke$ satisfying this condition is called the \emph{nucleus} of the group (with respect to the basis $\alb$).
\end{defi}

\begin{example}
Let $a$ be the generator of the iterated monodromy group of the angle-doubling map, see Example~\ref{ex:angledoubling}. Then, we have
\[a^2|_1=a^2|_0=a.\]
Consequently, by~\eqref{eq:sections}, for every element $g$ of the iterated monodromy group there exists $n$ such that $g|_v\in\{1, a, a^{-1}\}$ for all $v\in\xs$ of length at least $n$. It follows that the group is contracting with the nucleus $\nuke=\{1, a, a^{-1}\}$.
\end{example}

\begin{example}
\label{ex:basilicanucleus}
Consider the group $\img(z^2-1)$ from Example~\ref{ex:basilicaimgcomput}. We have 
\begin{alignat*}{2}
a^2&=(b, b), &\quad b^2&=(1, a^2),\\
ab&=(01)(1, ba), &\quad ba&=(01)(a, b),\\
ab^{-1}&=(01)(1, ba^{-1}),&\quad a^{-1}b&=(01)(b^{-1}, a).
\end{alignat*}
These equalities imply that the nucleus of this group is $\{1, a, a^{-1}, b, b^{-1}, ab^{-1}, ba^{-1}\}$.
\end{example}

The following proposition is proved in~\cite[Proposition~2.11.5]{nek:book}.

\begin{proposition}
\label{pr:sumgeometric}
    Let $(G, \bim)$ be a self-similar group. Choose a basis $\alb\subset\bim$. Suppose $G$ is contracting with respect to $\alb$. Then for every finite set $A\subset\bim$ there exists a finite set $\nuke_A\subset G$ such that every product $a_1\otimes a_2\otimes\cdots\otimes a_n$ of elements of $A$ is equal to $v\cdot g$ for $v\in\xs$ and $g\in\nuke_A$.
\end{proposition}

\begin{corollary}
    If $(G, \bim)$ is contracting with respect to one basis, then it is contracting with respect to any basis.
\end{corollary}

\subsection{Self-similarity complex and the limit space}

Let $(G, \bim)$ be a finitely generated self-similar group. Choose a basis $\alb\subset\bim$ and a finite generating set $S\subset G$. Then the semigroup $\bim^*$ is generated by $S\cup\alb$. The quotient of the \emph{left} Cayley graph of $\bim^*$ with respect to this generating set by the right $G$-action is called the \emph{self-similarity complex}. If we naturally identify its set of vertices with $\xs$, then its edges are of the following two types:
\begin{enumerate}
    \item \emph{vertical edges} connecting $v$ to $xv$ for $v\in\xs, x\in\alb$;
    \item \emph{horizontal edges} connecting $v$ to $g(v)$ for $v\in\xs, g\in S$.
\end{enumerate}

The following is proved in~\cite{nek:hyplim}, see also~\cite[Theorem~3.8.6]{nek:book}.

\begin{proposition}
    If $(G, \bim)$ is contracting, then its self-similarity complex is Gromov hyperbolic.
\end{proposition}

\begin{example}
\begin{figure}
    \centering 
    \includegraphics{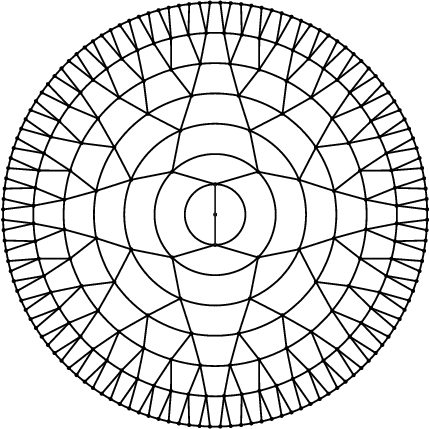}
    \caption{The self-similarity complex for the odometer}
    \label{fig:admcom}
\end{figure}
See Figure~\ref{fig:admcom}, where a part of the self-similarity complex for the adding machine action of $\Z$ from Example~\ref{ex:angledoubling} is shown.
\end{example}

The boundary of the self-similarity complex of a contracting self-similar group $(G, \bim)$ is called the \emph{limit space} of $(G, \bim)$, and is denoted $\lims[(G, \bim)]$ or just $\lims[G]$. 

The limit space can be alternatively described in the following way, suitable also for infinitely generated groups.

\begin{proposition}
\label{pr:lims}
Every point of $\lims[(G, \bim)]$ is the limit $\lambda(\ldots x_2x_1)$ of a sequence $\emptyset, x_1, x_2x_1, x_3x_2x_1, \ldots$ for some $\ldots x_2x_1\in\xmo$. The map $\lambda\colon\xmo\arr\lims[(G, \bim)]$ is a quotient map. We have $\lambda(\ldots x_2x_1)=\lambda(\ldots y_2y_1)$ if and only if there exists a finite set $N\subset G$ and a sequence $g_n\in N$ (equivalently, a sequence $g_n\in\nuke$, where $\nuke$ is the nucleus) such that $g_n(x_n\ldots x_2x_1)=y_n\ldots y_2y_1$ for every $n$.
\end{proposition}

The equivalence described in Proposition~\ref{pr:lims} is invariant under the shift $\ldots x_2x_1\mapsto \ldots x_3x_2$, hence the shift induces a continuous map $\si\colon\lims\arr\lims$, which we call the \emph{limit dynamical system}.

\begin{example}
\label{ex:realcircle}
Consider the adding machine from Example~\ref{ex:angledoubling}. Denote, for $w=\ldots x_2x_1\in\{0, 1\}^{-\omega}$ by $\theta(w)$ the point of $\R/\Z$ represented by the real number $\sum_{k=1}^\infty x_k2^{-k}$, i.e., by the binary real $.x_1x_2\ldots$. If $a^{n_k}$ is a sequence such that $a^{n_k}(x_k\ldots x_1)=y_k\ldots y_1$, then $|\theta(\ldots x_2x_1)-\theta(\ldots y_2y_1)|\le 2^{-k}n_k$. Consequently, if the sequence $(n_k)$ is bounded, then $\theta(\ldots x_2x_1)=\theta(\ldots y_2y_1)$. 

Conversely, if $\theta(w_1)=\theta(w_2)$, then the sequences $w_i$ are of the form $\ldots 0001x_n\ldots x_2x_1$ and $\ldots 1110x_n\ldots x_2x_1$. We have
\[a(111\ldots 10x_n\ldots x_2x_1)=000\ldots 01x_n\ldots x_2x_1,\]
hence $w_1$ and $w_2$ are equivalent (represent the same point of the limit space).

It follows that the  limit space in this case is homeomorphic to the circle $\R/\Z$ and the limit dynamical system is the angle doubling map.
\end{example}

The previous example is generalized to all locally expanding coverings as follows.

\begin{theorem}
\label{th:limstheorem}
    Let $f\colon\M\arr\M$ be a locally expanding covering map, where $\M$ is a compact metric space. Then the limit dynamical system of $\img(f)$ is topologically conjugate to $f$. In particular, the limit space of $\img(f)$ is homeomorphic to $\M$.
\end{theorem}

The limit space $\lims$ of a contracting group $(G, \bim)$ has a natural structure of a (developable) \emph{orbispace}, i.e., a quotient of a topological space by a proper action of a group. Namely, we lift the equivalence relation from Proposition~\ref{pr:lims} to an equivalence relation on $\xmo\times G$ in the following way.

\begin{defi}
    Consider the topological space $\xmo\times G$, where $G$ is discrete. We say that $\ldots x_2x_1\cdot g, \ldots y_2y_1\cdot h\in\xmo\times G$ are \emph{asymptotically equivalent} if there exists a finite set $N\subset G$ and a sequence $g_n\in N$ such that
    \[g_n\cdot x_n\ldots x_2x_1\cdot g=y_n\ldots y_2y_1\cdot h\]
    for all $n$. 

    The quotient of $\xmo\times G$ by the asymptotic equivalence relation is denoted $\limg=\limg[(G, \bim)]$, and is called the \emph{limit $G$-space}. The right action of $G$ on $\limg$ is induced by the natural right action of $G$ on $\xmo\times G$.
\end{defi}

\begin{example}
Similar arguments to those in Example~\ref{ex:realcircle} show that the limit $G$-space of the binary adding machine is naturally homeomorphic to $\R$ with the usual action of $\Z$ on it. The image of a sequence $\ldots x_2x_1\cdot a^n\in\xmo\times G$ is the real number $n+\sum_{k=1}^\infty x_k2^{-k}$.
\end{example}

The limit space $\lims$ is naturally the space of orbits $\limg/G$. One can show that the action of $G$ on $\limg$ is proper (see~\cite[Proposition~4.6.1]{nek:book}).

\begin{defi}
\label{def:regularssgroup}
Let $(G, \bim)$ be a self-similar group, and let $\alb\subset\bim$ be a basis. We say that $(G, \bim)$ is \emph{regular} if for every $x_1x_2\ldots\in\xo$ there exists $n$ such that $g(x_1x_2\ldots x_n)\ne x_1x_2\ldots x_n$ or $g|_{x_1x_2\ldots x_n}=1$. (This is equivalent to the condition that the groupoid of germs of the action of $G$ on $\xo$ is \emph{principal}, i.e., that all its isotropy groups are trivial.)
\end{defi}

The following is proved in~\cite[Proposition~4.5.23]{nek:dyngroups}.

\begin{theorem}
\label{th:regularimg}
    The action of $G$ on $\limg$ is free if and only if $(G, \bim)$ is regular. 

    A self-similar contracting group is regular if and only if it is equivalent to the iterated monodromy group of a locally expanding covering map $f\colon\M\arr\M$.

The limit dynamical system of the iterated monodromy group of a locally expanding covering map $f\colon\M\arr\M$ is topologically conjugate to $f$.
\end{theorem}

Here we use a more general definition of the iterated monodromy group that also applies when $\M$ is not path-connected, see~\cite[Subsection~4.3.1]{nek:dyngroups}.

\section{Metrics on $\lims$}

Throughout the paper, we write $x\doteq_C y$ or just $x\doteq y$ for $x-C\le y\le x+C$. Similarly, we write $x\asymp_C y$ or just $x\asymp y$ for $C^{-1}x\le y\le Cx$. 

\subsection{Ahlfors-regular quasiconformal metrics}

Let $(G, \bim)$ be a contracting group. Choose a basis $\alb\subset\bim$, and let $\nuke\subset G$ be the associated nucleus. The \emph{tile} $\til\subset\limg$ is the image of $\xmo\cdot 1\subset\xmo\times G$ in $\limg$. 

For $v\in\xs$, we denote by $\til_v$ the image of the set $\xmo v$ in $\lims$. If $v\in\alb^n$, then we call $\til_v$ the \emph{tile of $n$th level}. By~\cite[Proposition~3.3.5]{nek:book}, two tiles of the same level $\til_v$ and $\til_u$ intersect if and only if there exists $g\in\nuke$ such that $u=g(v)$.

\begin{lemma}
\label{lem:intersectingtiles}
Let $v, u\in\xs$ be such that $|v|\ge |u|$. Then the tiles $\til_v, \til_u$ intersect if and only if there exists $g\in\nuke$ such that $u$ is a suffix of $g(v)$.
\end{lemma}

\begin{proof}
Let $v=v_1u_1$, where $|u_1|=|u|$. Suppose that $g(v_1u_1)$ has $u$ as a suffix, i.e., that $g(v_1u_1)=g(v_1)u$. Then $u=g|_{v_1}(u_1)$ and $g|_{v_1}\in\nuke$, hence $\til_{u_1}$ and $\til_u$ intersect. But $\til_{u_1}\subset\til_v$, hence $\til_v$ and $\til_u$ also intersect.

Suppose that $\til_v$ and $\til_u$ intersect. Then there exists a point of $\lims$, which can be represented both by a sequence of the form $\ldots v_1u_1$ and by a sequence $\ldots v_0u\in\xmo$ for some $v_0\in\alb^{|v_1|}$. It follows that these two sequences are asymptotically equivalent, hence by Proposition~\ref{pr:lims}, there exists $g\in\nuke$ such that $g(v_1u_1)=v_0u$.
\end{proof}

Similarly, a \emph{tile of level $n$} in $\limg$ is a set of the form $\til\otimes v$ for $v\in\bim^{\otimes n}$. We also have $\til\otimes v\cap\til\otimes u\ne\emptyset$, for $v, u\in\bim^{\otimes n}$, if and only if there exists $g\in\nuke$ such that $u=g\cdot v$ in the biset $\bim^{\otimes n}$.

For $\xi, \zeta\in\lims$, denote by $\ell_\alb(\xi, \zeta)$  the largest $n$ such that $\xi$ and $\zeta$ can be represented by sequences $\ldots x_2x_1$ and $\ldots y_2y_1\in\xmo$ such that $y_n\ldots y_2y_1=s(x_n\ldots x_2x_1)$ for some $s\in\nuke$. 
Equivalently, it is the largest $n$ such that $\xi$ and $\zeta$ belong to intersecting tiles of $n$th level. If $\alb$ and $\mathsf{Y}$ are two bases of $\bim$, then there exists $C>0$ such that $\ell_\alb(\xi, \zeta)\doteq_C\ell_{\mathsf{Y}}(\xi, \zeta)$ for all $\xi, \zeta\in\lims$.

The function $\ell_\alb(\xi, \zeta)$ is equal (up to an additive constant) to the classical Gromov product on the boundary $\lims=\partial\Sigma$ of the self-similarity complex $\Sigma$ based at the empty word $\emptyset$. Consequently, it satisfies the hyperbolicity condition
\begin{equation}
\label{eq:logscale}
\ell_\alb(\xi, \zeta)\ge \min(\ell_\alb(\xi, \eta), \ell_\alb(\eta, \zeta))-C
\end{equation}
for some $C>0$, which is not hard to prove directly, using the definition of contracting groups.

If $(G, \bim)$ is the iterated monodromy group of a locally expanding covering map $f\colon\M\arr\M$, then, after identification of $\M$ with $\lims$ using Theorem~\ref{th:regularimg}, the function $\ell$ becomes equal, up to an additive constant, to the log-scale $\ell_U$ defined in the introduction.

Denote, for $\xi\in\lims$ and a non-negative integer $n$,
\[T_n(\xi)=\{\zeta\in\lims\colon \ell_\alb(\xi, \zeta)\ge n\}.\]
The set $T_n(\xi)$ is equal to the union of tiles of $n$th level that intersect tiles of the $n$th level containing $\xi$. The union of tiles of $n$th level containing $\xi$ is a neighborhood of $\xi$ (see~\cite[Proposition~3.4.1]{nek:book}), hence $T_n(\xi)$ is also a neighborhood of $\xi$.
We also denote by $T_n(\xi)$, for $\xi\in\limg$, the neighborhood of $\xi$ equal to the union of the $n$th level tiles of $\limg$ intersecting the $n$th level tiles containing $\xi$.

\begin{defi}
  Let $\X$ be a space with a log-scale $\ell$ (i.e., a function satisfying~\eqref{eq:logscale} for some $C$).  A metric $d$ on $\X$ is said to be a \emph{metric of exponent $\alpha$}, where $\alpha$ is a positive real number, if it satisfies
    \[d(\xi_1, \xi_2)\asymp e^{-\alpha\ell(\xi_1, \xi_2)}\]
    for all $\xi_1, \xi_2\in\X$.
\end{defi}

Note that any two visual metrics of exponent $\alpha$ are bi-Lipschitz equivalent to each other.
A visual metric $d_\alpha(\xi, \eta)$ of exponent $\alpha$  can be defined as the infimum of sums
\[\sum_{i=0}^m e^{-\alpha\ell(\xi_i, \xi_{i+1})}\]
over all sequences $\xi=\xi_0, \xi_1, \ldots, \xi_m=\eta$. The function $d_\alpha$ will be symmetric, satisfy the triangle inequality, and $d_\alpha(\xi, \xi)=0$ for all $\xi$. The only potentially missing condition for it to be a metric may be positivity. 

\begin{defi}
\label{def:lowerexponentforlog}
Let $\ell$ be a log-scale on a set $\X$. We say that $\alpha$ is a \emph{lower exponent} of $\ell$ if there exists $C>0$ such that for any $n$, any $\xi, \zeta\in\X$ and any sequence $\xi_0=\xi, \xi_1, \ldots, \xi_m=\zeta$ of points of $\X$ such that $\ell(\xi_i, \xi_{i+1})\ge n$ for all $i=1, \ldots, m-1$, we have
\[m\ge Ce^{\alpha(n-\ell(\xi, \zeta))}.\]
\end{defi}

The following characterization of visual metrics is proved in Theorem~\cite[Theorem~3.4]{nek:smale}.

\begin{theorem}
The supremum of values $\alpha$ for which $d_\alpha$ is a metric is equal to the supremum of the lower exponents of the log-scale $\ell$.
\end{theorem}

A \emph{uniform metric of exponent $\alpha$} on the limit space $\lims$ is a visual metric of exponent $\alpha$ associated with the log-scale $\ell_\alb$ on $\lims$.

The limit spaces of $(G, \bim)$ and $(G, \bim^{\otimes n})$ are identified by the map induced by the natural identification of $\xmo$ with $(\alb^n)^{-\omega}$. The limit dynamical system of the latter group is the $n$th iteration of the limit dynamical system of the former. It is easy to check that the sets of the uniform metrics on $\lims$ coming from $(G, \bim)$ and $(G, \bim^{\otimes n})$ are the same. More precisely, a metric on $\lims$ is uniform of exponent $\alpha$ for $(G, \bim)$ if and only if it is uniform of exponent $n\alpha$ for $(G, \bim^{\otimes n})$.

If $d_1$ and $d_2$ are uniform metrics of exponents $\alpha_1$, $\alpha_2$, respectively, then we have
\[\frac{d_1(x, y)}{d_1(x, z)}\le C_1e^{-\alpha_1(\ell(x, y)-\ell(x, z))}\]
and
\[C_2e^{-\alpha_2(\ell(x, y)-\ell(x, z))}\le \frac{d_2(x, y)}{d_2(x, z)}\]
for some $C_1, C_2>0$, hence $\frac{d_1(x, y)}{d_1(x, z)}\le\eta\left(\frac{d_2(x, y)}{d_2(x, z)}\right)$ for $\eta(t)=C_2C_1^{-\alpha_2/\alpha_1}t^{\alpha_2/\alpha_1}$.

Two metrics $d_1, d_2$ are said to be \emph{quasi-symmetric} (see~\cite{heinonen}) if there exists an increasing homeomorphism $\eta\colon [0, \infty)\arr [0, \infty)$ such that
\[\frac{d_1(x, y)}{d_1(x, z)}\le\eta\left(\frac{d_2(x, y)}{d_2(x, z)}\right)\]
for all $x, y, z$ such that $x\ne z$. 
As we have seen above, any two uniform metrics on the limit space $\lims$ are quasi-symmetric to each other.

The \emph{uniform measure} $\mu$ on $\lims$ is the push-forward of the uniform Bernoulli measure on $\xmo$ by the quotient map $\xmo\arr\lims$.  

\begin{proposition}
Let $d_\alpha$ and $\mu$ be a uniform metric of exponent $\alpha$ and the uniform measure on $\lims$, respectively. Then there exists $C>1$ such that 
\[\mu(B_\xi(R))\asymp_C R^{\frac{\log \mathsf{D}}{\alpha}}\]
for every $\xi\in\lims$ and every $R<\diam\lims$, where $\mathsf{D}=|\bim/G|$.

In particular, the Hausdorff dimension of $(\lims, d_\alpha)$ is equal to $\frac{\log\mathsf{D}}{\alpha}$.
\end{proposition}

\begin{proof}
Let $\alb$ be a basis of $\bim$. We have $\mathsf{D}=|\alb|$.
Let a point $\xi, \zeta\in\lims$ be represented by sequences $\ldots x_2x_1, \ldots y_2y_1\in\xmo$, respectively. The set $\xmo x_n\ldots x_2x_1$ belongs to the preimage of $T_\xi(n)$ in $\xmo$ and \[\mu(T_\xi(n))\ge \mathsf{D}^{-n}.\] 

On the other hand, if $\zeta\in T_\xi(n)$, then $\zeta$ and $\xi$ belong to intersecting tiles of level $n$, hence there exist $v\in\alb^n$ and $s\in\nuke$ such that $\xi\in\til_v$ and $\zeta\in\til_{s(v)}$. Therefore, there exists $s\in\nuke^3$ such that $y_n\ldots y_2y_1=s(x_n\ldots x_2x_1)$. It follows that the preimage of $\til_\xi(n)$ is contained in $\bigcup_{s\in\nuke^3}\xmo s(x_n\ldots x_2x_1)$, hence \[\mu(T_\xi(n))\le|\nuke|^3\cdot\mathsf{D}^{-n}.\]

It is enough to prove the inequality in the proposition for all $R<R_0$ for some $R_0>0$. Let $C_0$ be such that
\[d_\alpha(\xi_1, \xi_2)\asymp_{C_0} e^{-\alpha\ell(\xi_1, \xi_2)}\] for all $\xi_1, \xi_2\in\lims$.

Then, for every $\xi\in\lims$ and $R>0$, we have
\[T_\xi\left(-\log(C_0^{-1}R)/\alpha\right)\subseteq B_\xi(R)\subseteq T_\xi\left(-\log(C_0R)/\alpha\right).\]
Consequently, for all $R<C_0^{-1}$ and all $\xi\in\lims$, we have
\[\mathsf{D}^{\log(C_0^{-1}R)/\alpha}\le\mu(B_\xi(R))\le |\nuke|^3\mathsf{D}^{\log(C_0R)/\alpha},\]
which is equivalent to
\[\mathsf{D}^{-\log C_0/\alpha}\cdot R^{\frac{\log\mathsf{D}}{\alpha}}\le |\nuke|^3\mathsf{D}^{\log C_0/\alpha}\cdot R^{\frac{\log\mathsf{D}}{\alpha}},\]
which finishes the proof.
\end{proof}

\begin{defi}
\label{def:AR}
Denote by $\mathcal{AR}=\mathcal{AR}(G, \bim)$ the set of metrics $d$ on $\lims$ satisfying the following two conditions.
\begin{enumerate}
\item The metric $d$ is quasi-symmetric to uniform metrics on $\lims$.
\item There exist $C>1$, $\alpha>0$, and a measure $\mu$ on $\lims$ such that
\[\mu(B_\xi(R))\asymp_C R^\alpha\]
for every $\xi\in\lims$ and $R\le\diam\lims$.
\end{enumerate}

The \emph{(Ahlfors-regular) conformal dimension} $\cdim(G, \bim)$ is the infimum of the Hausdorff dimensions of the metrics in the class $\mathcal{AR}(G, \bim)$.
\end{defi}

We have $\mathcal{AR}(G, \bim)=\mathcal{AR}(G, \bim^{\otimes n})$ for every $n\ge 1$, after the natural identification of $\lims[(G, \bim)]$ and $\lims[(G, \bim^{\otimes n})]$.

\begin{proposition}
\label{pr:diamoftiles}
Choose a metric $d\in\mathcal{AR}(G, \bim)$, and let $\mu$ be the associated Hausdorff measure of dimension $\alpha$. 

There exists a constant $C>1$ such that $\mu(\til_v)\asymp_C(\diam\til_v)^{\alpha}$ for all $v\in\xs$.

There exist $\lambda\in (0, 1)$ and $C>1$ such that $\diam\til_v\le C\lambda^{|v|}$ for all $v\in\xs$.
\end{proposition}

\begin{proof}
Since the tiles $\til\cdot g$ cover $\limg$ and are closed, by the Baire category theorem, the interior of $\til$ is non-empty. Let $\xi$ be a point in the interior of $\til$. Then there exists $n_0$ such that $T_\xi(n_0)\subset\til$. Then for every $n\ge 1$ and $v\in\bim^{\otimes n}$ we have $T_{\xi\otimes v}(n+n_0)\subset \til\otimes v$. Mapping these sets to $\lims$, we get that every tile $\til_v$ of $\lims$ contains a set of the form $T_\zeta(|v|+n_0)$ for some $\zeta\in\til_v$. Note that $\til_v\subset T_\zeta(|v|)$. It follows that there exist radii $r, R$ such that
\[B_\zeta(r)\subset\til_v\subset B_\zeta(R)\]
and $R/r\le C_1$ for some constant $C_1$. Condition (2) of Definition~\ref{def:AR} implies then that $\mu(\til_v)\asymp (\diam\til_v)^\alpha$, as $r<\diam\til_v<2R$.

It is shown in~\cite[Proposition~3.8]{mattebonnekzheng:liouvills} that $\sup_{\xi\in\lims}\diam T_n(\xi)\le C_2\lambda^n$ for some $C_2>0$ and $\lambda\in (0, 1)$, which implies the second statement of our proposition, since $\til_v\subset T_n(\xi)$ for every $v\in\alb^n$ and $\xi\in\til_v$. 
\end{proof}

\subsection{Quasi-morphisms and the associated measures}

\begin{defi}
\label{def:qcoc}
Let $G$ be a group acting on $\xs$. A \emph{quasi-morphism} is a map $\nu\colon \xs\arr [0, \infty)$ such that there exists $C>0$ such that
    \[\nu(v_1v_2)\doteq_C\nu(v_1)+\nu(v_2)\]
    for all $v_1, v_2\in\xs$.

    We say that the quasi-morphism $\nu$ is \emph{positive} if 
    \[\lim_{n\to\infty}\min_{v\in\alb^n}\nu(v)\to\infty.\]

    We say that it is \emph{$G$-quasi-invariant} if for every $g\in G$ there exists $C_g>0$ such that 
    \[\nu(g(v))\doteq_{C_g} \nu(v)\]
    for all $v\in\xs$.
\end{defi}

For example $\nu(v)=|v|$ is a positive $G$-quasi-invariant quasi-morphism. In fact, it is $G$-invariant and is a homomorphism from the semigroup $\xs$ to the additive semigroup $[0, \infty)$.

Let $\nu$ be a positive quasi-morphism, and let $C$ be a constant satisfying the condition of the definition of a quasi-morphism. We may assume (after changing $C$ to a larger number) that $\nu(x)<C$ for all $x\in\alb$. By the positivity of $\nu$, there exists $m$ such that $\nu(v)>C+1$ for all $v\in\alb^m$. Then, by induction on $k$, we get $\nu(w)>C+k$ for all $w\in\alb^{mk}$. It follows that we have
\[m^{-1}|v|\le \nu(v)\le C|v|\]
for all $v\in\xs$.

Consequently, for any two positive quasi-morphisms $\nu_1, \nu_2\colon \xs\arr [0, \infty)$, we have $\nu_1\asymp\nu_2$.

The next construction is a particular case of Patterson measures, specified for the case of the boundary of the self-similarity complex $\Sigma$. For a more general setting, see~\cite{nek:psmeasure}. We provide here a proof for completeness.

\begin{theorem}
Let $\nu\colon \xs\arr [0, \infty)$ be a positive quasi-morphism. There exists $\beta>0$ such that for every sufficiently large $R$ we have
\[|\{v\in\xs\colon\nu(v)\le R\}|\asymp e^{\beta R}.\]
\end{theorem}

We call the number $\beta$ satisfying the conditions of the theorem the \emph{$\nu$-entropy}. For example, if $\nu(v)=|v|$, then $\nu$-entropy is equal to $\log|\alb|$.

\begin{proof}
Denote $V(R)=\{v\in\xs\colon\nu(v)\le R\}$. It is finite, since $\nu(v)\asymp |v|$ for all $v\in\xs$.

Let $m$ be such that $\nu(v)\ge C+1$ for all $v\in\alb^m$, where $C$ is a constant satisfying the inequality in Definition~\ref{def:qcoc}. Let $M$ be the maximal value of $\nu(x)$ for $x\in\alb$.

Let $R_1, R_2$ be arbitrary positive numbers. Consider a word $v\in\xs$ such that $\nu(v)\le R_1+R_2$. Let $v_1$ be the longest prefix of $v$ such that $\nu(v_1)\le R_1$. 
Let $v_2\in\xs$ be such that $v=v_1v_2$. If $x$ is the first letter of $v_2$, then $R_1<\nu(v_1x)\le\nu(v_1)+\nu(x)+C\le\nu(v_1)+M+C$. Consequently, $\nu(v_1)>R_1-M-C$.

Then $\nu(v_1)+\nu(v_2)\le \nu(v_1v_2)+C\le R_1+R_2+C$, hence $\nu(v_2)\le R_1-\nu(v_1)+R_2+C<R_1-(R_1-M-\Delta)+R_2+C=
R_2+m+2C$. 

We have shown that every word $v$ such that $\nu(v)\le R_1+R_2$ can be split into two words $v=v_1v_2$ such that $\nu(v_1)\le R_1$ and $\nu(v_2)\le R_2+M+2C$. It follows that
\[|V(R_1+R_2)|\le |V(R_1)|\cdot |V(R_2+M+2C)|.\]
Equivalently,
\[|V(R_1+R_2-K)|\le|V(R_1)|\cdot |V(R_2)|\]
for $K=M+2C$ and all $R_1, R_2$ such that $R_1+R_2>K$.

Consider the function $\gamma(R)=|V(R+K)|$. 
Then, for all $R_1, R_2>0$, we have
\begin{multline*}\gamma(R_1+R_2)=|V(R_1+R_2+K)|=|V(R_1+K+R_2+K-K)|\le\\ 
|V(R_1+K)|\cdot|V(R_2+K)|=\gamma(R_1)\gamma(R_2).
\end{multline*}

It then follows from Fekete's Lemma that the limit $\lim_{R\to\infty}\frac{\log\gamma(R)}{R}$ exists and is equal to $\inf_{R>0}\frac{\log |V(R+K)|}{R}$. Let us denote the limit by $\beta$. We have 
\[|V(R+K)|\ge e^{\beta R}\] for all $R>0$,
hence 
\[C_1e^{\beta R}\le |V(R)|\]
for $C_1=e^{-\beta K}$ and all $R>K$.

On the other hand, for any pair $R_1, R_2>0$,  the map $V(R_1)\times V(R_2)\colon (v_1, v_2)\mapsto v_1v_2$ takes values in $V(R_1+R_2+C)$ and is injective. Consequently,
\[|V(R_1+R_2+C)|\ge |V(R_1)|\cdot |V(R_2)|\]
for all $R_1, R_2>0$.
The function $\gamma(R)=|V(R-C)|$ satisfies then
\begin{multline*}
    \gamma(R_1+R_2)=|V(R_1+R_2-C)|=|V(R_1-C+R_2-C+C)|\ge\\
    |V(R_1-C)|\cdot|V(R_2-C)|=\gamma(R_1)\gamma(R_2).
\end{multline*}
Consequently, the limit $\beta=\lim_{R\to\infty}\frac{\log\gamma(R)}R$ is equal to the supremum of $\frac{\log\gamma(R)}R$, consequently,
\[\frac{\log|V(R-C)|}R\le\beta\]
for all $R>C$, i.e., 
\[|V(R-C)|\le e^{\beta R},\]
which gives us the estimate
\[|V(R)|\le C_2e^{\beta R}\]
for $C_2=e^{\beta C}$ and all $R>0$.
\end{proof}

\begin{proposition}
\label{pr:entropycriterion}
    Let $\nu$ be a positive quasi-morphism, and let $\beta$ be the $\nu$-entropy. Suppose that $\beta_0>0$ is such that there exists $\lambda\in (0, 1)$ such that for every $v\in\xs$ we have
    \[\sum_{x\in\alb}e^{-\beta_0\nu(xv)}\le\lambda e^{-\beta_0\nu(v)}.\]
    Then $\beta\le\beta_0$.
\end{proposition}

\begin{proof}
The inequality in the proposition implies that for every $n$ and $v\in\xs$ we have
\[\sum_{u\in\alb^n}e^{-\beta_0\nu(uv)}\le \lambda^n e^{-\beta_0\nu(v)}.\]

It follows that the series $\Lambda=\sum_{v\in\xs}e^{-\beta_0\nu(v)}$ is convergent.

Consider the set $V(R)=\{v\in\xs\colon\nu(v)\le R\}$. For 
every $v\in V(R)$, we have $e^{-\beta_0\nu(v)}\ge e^{-\beta_0 R}$. Consequently, $\Lambda\ge |V(R)|\cdot e^{-\beta_0 R}$, hence $|V(R)|\le\Lambda e^{\beta_0 R}$ for all $R$, which implies that $\beta\le\beta_0$.
\end{proof}

\begin{theorem}
Let $\nu$ be a positive quasi-morphism, and let $\beta$ be the $\nu$-entropy.

    There exist measures $\mu_+$ and $\mu_-$ on $\xo$ and $\xmo$, respectively, such
    \[\mu_+(v\xo)\asymp e^{-\beta\nu(v)}\asymp\mu_-(\xmo v)\]
    for all $v\in\xs$.

    If $\nu$ is $G$-quasi-invariant for an action of $G$ by automorphisms of the rooted tree $\xs$, then the measure $\mu_+$ is quasi-invariant under the action of $G$ on $\xo$.
\end{theorem}

\begin{proof}
    Let us prove the first statement of the theorem for $\xo$, the proof for $\xmo$ is the same.

Consider $\xo\cup\xs$ as the natural compactification of $\xs$ by $\xo$.
    For every $s>\beta$, and a continuous function $f\colon \xo\cup\xs\arr\C$, consider the series
\[\mathcal{P}_s(f)=\sum_{v\in\xs}f(v)e^{-s\nu(v)}.\]
It is a positive linear functional on $C(\xo\cup\xs)$.

Let $C$ be such that \[C^{-1}e^{\beta R}\le |\{v\in\xs\colon\nu(v)\le R\}|\le C e^{\beta R}\]
for all $R>0$.

Let $S(n)=\{v\in\xs\colon n-1< \nu(v)\le n\}$.
Then  we have
\begin{multline*}\mathcal{P}_s(f)=\sum_{n=1}^\infty \sum_{v\in S(n)}f(v)e^{-s\nu(v)}\le C\|f\|_\infty\sum_{n=1}^\infty e^{-s(n-1)}e^{\beta n}=\\ 
C\|f\|_\infty\sum_{n=0}^\infty e^{-sn}e^{\beta(n+1)}=
Ce^\beta\|f\|_\infty\sum_{n=0}^\infty e^{(\beta-s)n}=\frac{C_1\|f\|_\infty}{1-e^{\beta-s}}.
\end{multline*}

Define the measure $\mu_s$ on $\xs\cup\xo$ by
\[\int f\;d\mu_s=(1-e^{\beta-s})\mathcal{P}_s(f)\]
for all continuous functions $f\colon \xs\cup\xo\arr\C$. By the above, we have $\int f\;d\mu_s\le C_1\|f\|_\infty$ for all $f\in C(\xs\cup\xo)$ and all $s>\beta$.

It follows from the Uniform Boundedness Principle that the set $\{\mu_s\colon s>\beta\}$ is bounded in the dual of $C(\xs\cup\xo)$. Hence, by the Banach-Alaoglu Theorem, there exists a weakly convergent sequence $\mu_{s_k}$ for $s_k\to\beta+$. Let $\mu$ be its limit. 

For every $v\in\xs$, the sequence $(1-e^{\beta-s_k})e^{-s_k\nu(v)}$ converges to 0, hence $\mu$ is supported on $\xo$.

Let $v\in\xs$, and let $\xi_v$ be the indicator of $v\xs\cup v\xo$ (which is a continuous function). Then
\[\mu(v\xo)=\lim_{k\to\infty}(1-e^{\beta-s_k})\sum_{w\in\xs}e^{-s_k\nu(vw)}.\]
We have \[e^{-s_k(\nu(w)+\nu(v)+C)}\le e^{-s_k\nu(vw)}\le e^{-s_k(\nu(w)+\nu(v)-C)}.\]
Multiplying the inequalities by $(1-e^{\beta-s})$ and adding them over all $w\in\xs$, we get
\[e^{-C}e^{-s_k\nu(v)}\mu_{s_k}(\xs\cup\xo)\le\mu_{s_k}(v\xs\cup v\xo)\le e^C e^{-s_k\nu(v)}\mu_{s_k}(\xs\cup\xo).\] Passing to the limit, we get
\[K_1 e^{-\beta\nu(v)}\le\mu(v\xo)\le K_2 e^{-\beta\nu(v)}\]
for $K_1=e^{-\Delta}\mu(\xo)$ and $K_2=e^\Delta\mu(\xo)$.

Note that, as a consequence of the estimates, we get that for any $v_1, v_2\in\xs$, the measure $\mu$ is quasi-invariant under the transformation $w\mapsto vw$ with the Radon-Nikodym derivative belonging to the interval $(C_1^{-1}e^{-\beta\nu(v)}, C_1e^{-\beta\nu(v)})$ for a  constant $C_1$, not depending on $v$.

It follows that in order to show that $\mu$ is quasi-invariant under the action of $G$, it is enough to show that it is quasi-invariant under the action of the elements of the nucleus. But this follows from an estimate $\nu(s(v))\doteq \nu(v)$ for all $v\in\xs$ and $s\in\nuke$.
\end{proof}

\subsection{Visual metrics associated with a quasi-morphism}
Fix a contracting group $(G, \bim)$ and a basis $\alb\subset\bim$.

Let $\nu\colon \xs\arr [0, \infty)$ be a positive $G$-quasi-invariant quasi-morphism. 

Let $k$ be such that $\nuke^3|_v\subset\nuke$ for all $v\in\alb^k$.
Denote by $\ell_\nu(\xi, \eta)$ the maximal value of $\max(\nu(x_n\ldots x_2x_1), \nu(y_n\ldots y_2y_1))$ for all pairs of sequences $\ldots x_2x_1, \ldots y_2y_1\in\xmo$ representing $\xi, \eta$ and all $n$ such that $s(x_n\ldots x_2x_1)=y_n\ldots y_2y_1$ for some $s\in\nuke$.

If $\ldots x_2'x_1', \ldots y_2'y_1'$ is another pair of sequences representing $\xi, \eta$, then there exists $s'\in\nuke^3$ such that $s'(x_n'\ldots x_2'x_1)=y_n'\ldots y_2'y_1$. Consequently, there exists $s''\in\nuke$ such that $s''(x_{n-k}'\ldots x_2'x_1')=y_{n-k}'\ldots y_2'y_1'$.

Since $\nu$ is a $G$-quasi-invariant quasi-morphism, we conclude that there exists a constant $\Delta_1$ such that if $\ldots x_2x_1, \ldots y_2y_1$ are arbitrary sequences representing $\xi, \eta$, and $n$ is the maximal index such that $s(x_n\ldots x_2x_1)=y_n\ldots y_2y_1$ for some $s\in\nuke$, then
\[\ell_\nu(\xi, \zeta)\doteq_{\Delta_1}\nu(x_n\ldots x_2x_1).\]

We also get that $\ell_\nu$ satisfies the inequality~\eqref{eq:logscale}, i.e., is a log-scale. Consequently, for all sufficiently small positive values of $\alpha$ there exists a \emph{visual metric associated with $\nu$}, i.e., a metric $d_{\nu, \alpha}$ on $\lims$ such that 
\[d_{\nu, \alpha}(\xi, \eta)\asymp e^{-\alpha\ell_\nu(\xi, \eta)}.\]

\begin{defi}
Let $\nu\colon\xs\arr [0, \infty)$ be a quasi-invariant quasi-morphism for a self-similar contracting group $(G, \bim)$. Then the \emph{$\nu$-dimension} $\dim_\nu(G, \bim)$ is the infimum of the Hausdorff dimensions of the visual metrics $d_{\nu, \alpha}$ on $\lims$. 
\end{defi}

We have $\dim_\nu(G, \bim)=\frac{\beta}{\alpha}$, where $\beta$ is the $\nu$-entropy, and $\alpha$ is the supremum of the lower exponents of the log-scale $\ell_\nu$ (see Definition~\ref{def:lowerexponentforlog}).

\begin{proposition}
\label{pr:quasimorphismsquasisymmetric}
Any two visual metrics associated with positive $G$-quasi-invariant quasi-morphisms are quasi-symmetric to each other.
\end{proposition}

\begin{proof}
It is enough to show that the visual metric associated with a positive $G$-quasi-invariant quasi-morphism $\nu$ is quasi-symmetric to a uniform metric.

Let $\xi\in\lims$, and let $\ldots x_2x_1$ be a sequence representing it. Denote 
\[\beta(n)=\nu(x_n\ldots x_2x_1).\]

There exist constants $\Delta>0$ and $C_1>1$ such that
\[\beta(n_1)-\beta(n_2)\doteq_{\Delta}\nu(x_{n_1}x_{n_1-1}\ldots x_{n_2+1})\]
and
\[\nu(x_{n_1}\ldots x_{n_2+1})\asymp_{C_1} n_1-n_2\]
for all $n_1>n_2$.

It follows that 
\[
\beta(n_1)-\beta(n_2)\ge\left\{\begin{array}{rr}
C^{-1}(n_1-n_2)-\Delta & \text{if $n_1\ge n_2$}\\
C(n_1-n_2)-\Delta & \text{if $n_1<n_2$.}
\end{array}\right.
\]
Denote $\gamma(n)=C^{-\mathop{\mathrm{sign}}(n)}n$, so that we have
\[\beta(n_1)-\beta(n_2)\ge\gamma(n_1-n_2)-\Delta\]
for all $n_1, n_2$. Note that $\gamma$ is a strictly increasing homeomorphism $\gamma\colon \R\arr\R$.

By the definition of the functions $\ell_\nu$ and $\ell$ (i.e., the function $\ell_{|\cdot|}$), there exists $\Delta_1>0$ such that
\[\ell_\nu(\xi, \eta)\doteq_{\Delta_1}\beta(\ell(\xi, \eta))\]
for all $\eta\in\lims$.

Let $\alpha$ and $\alpha_0$ be such that the visual metric $|\cdot|_{\nu, \alpha}$ and the uniform metric $|\cdot|_{\alpha_0}$ are defined.

Then there exist  constants $C_i>1$ such that
\begin{multline*}\frac{d_{\nu, \alpha}(\xi, \eta)}{d_{\nu, \alpha}(\xi, \zeta)}\le C_1e^{-\alpha(\ell_\nu(\xi, \eta)-\ell_\nu(\xi, \zeta))}\le C_1e^{2\alpha\Delta_1}e^{-\alpha(\beta(\ell(\xi, \eta))-\beta(\ell(\xi, \zeta)))}\le \\
C_2e^{-\alpha \gamma(\ell(\xi, \eta)-\ell(\xi, \zeta))}.
\end{multline*}

There also exist a constants $C_3>1$ and $C_4>0$ such that
\[\frac{d_{\alpha_0}(\xi, \eta)}{d_{\alpha_0}(\xi, \zeta)}\ge C_3^{-1}e^{-\alpha_0(\ell(\xi, \eta)-\ell(\xi, \zeta))},\]
hence
\[\ell(\xi, \eta)-\ell(\xi, \zeta)\ge -\alpha_0^{-1}\log\left(\frac{d_{\alpha_0}(\xi, \eta)}{d_{\alpha_0}(\xi, \zeta)}\right)-C_4\]

Consequently, 
\[\frac{d_{\nu, \alpha}(\xi, \eta)}{d_{\nu, \alpha}(\xi, \zeta)}\le\eta\left(\frac{d_{\alpha_0}(\xi, \eta)}{d_{\alpha_0}(\xi, \zeta)}\right)\]
for
\[\eta(t)=C_2\exp\left(-\alpha\gamma\left(-\alpha_0^{-1}\log t-C_4\right)\right),\]
which shows that $d_{\nu, \alpha}$ and $d_{\alpha_0}$ are quasi-symmetric.
\end{proof}

\begin{corollary}
Note that metrics on $\lims$ associated with two quasi-morphisms $\nu_1\colon(\alb^n)^*\arr (0, \infty)$ and $\nu_2\colon(\alb^m)^*\arr(0, \infty)$ are quasi-symmetric to each other.
\end{corollary}

\begin{proof}
Restrict the quasi-morphisms to $(\alb^{nm})^*$ and then use Proposition~\ref{pr:quasimorphismsquasisymmetric} for the self-similar group $(G, \bim^{\otimes nm})$.
\end{proof}

\begin{theorem}
Let $\nu\colon \xs\arr [0, \infty)$ be a positive $G$-quasi-invariant quasi-morphism. Let $\beta$ be the $\nu$-entropy.

Consider a visual metric $d_{\nu, \alpha}$, and let $\mu$ be the push forward to $\lims$ of the measure on $\xmo$ associated with $\nu$. Let $B_\xi(R)$ denote the ball of radius $R$ with center in $\xi$ defined using the metric $d_{\nu, \alpha}$. Then there exists a constant $C>1$ such that
\[\mu(B_\xi(R))\asymp_C R^{\beta/\alpha}\]
for all $R<\diam\lims$ and $\xi\in\lims$.
\end{theorem}

Consequently, the metric $d_{\nu, \alpha}$ belongs to the class $\mathcal{AR}(G, \bim)$ and has Hausdorff dimension $\beta/\alpha$.

\begin{proof}
Let $\ldots x_2x_1\in\xmo$ be a sequence representing $\xi$. Then, for some constants $C_1>1$, the neighborhood $T_\xi(n)$ contains the ball $B_\xi(C_1^{-1}e^{-\alpha\nu(x_n\ldots x_2x_1)})$ and is contained in the ball $B_\xi(C_1e^{-\alpha\nu(x_n\ldots x_2x_1)})$. 

The preimage of the set $T_\xi(n)$ in $\xmo$ contains $\xmo x_n\ldots x_2x_1$ and is contained in $\bigcup_{g\in\nuke^3}\xmo g(x_n\ldots x_1)$. Since $\nu$ is $G$-quasi-invariant, it follows that there exists a constant  $C_2>1$ such that
\[C_2^{-1}e^{-\beta\nu(x_n\ldots x_2x_1)}\le\mu(T_\xi(n))\le C_2e^{-\beta\nu(x_n\ldots x_2x_1)}.\]

Consequently, for $R_n=e^{-\alpha\nu(x_n\ldots x_2x_1)}$, we have
\[\mu(B_\xi(C_1^{-1}R_n))\le\mu(T_\xi(n))\le C_2 R_n^{\beta/\alpha}\]
and
\[\mu(B_\xi(C_1R_n))\ge\mu(T_\xi(n))\ge C_2^{-1}R_n^{\beta/\alpha}.\]

Let $R\in (0, R_1)$. Let $n_1$ be the maximum of the set of natural numbers $n$ such that $C_1^{-1}R_n\ge R$, which exists, since $R_n\to 0$ as $n\to\infty$. Let $n_0$ be the minimum of the set of natural numbers $n$ such that $C_1R_n\le R$. Since the ratio $R_n/R_{n+1}$ is bounded from below and from above by positive constants, the ratios $R_{n_0}/R$ and $R_{n_1}/R$ are also bounded from below and from above by positive constants.

We have 
\[\mu(B_\xi(R))\le\mu(B_\xi(R_{n_1}))\le C_2R_{n_1}^{\beta/\alpha}\le C_3R^{\beta/\alpha}\]
and
\[\mu(B_\xi(R))\ge\mu(C_1R_{n_0})\ge C_2^{-1}R_{n_0}^{\beta/\alpha}\ge C_3^{-1}R^{\beta/\alpha}\]
for some constant $C_3>1$.
\end{proof}

\subsection{Local quasisimilarities}

Denote  $\widetilde\til_v=\bigcup_{g\in\nuke^2}\til_{g(v)}$, for $v\in\xs$. If $\xi\in\lims$ is represented by a sequence $\ldots x_2x_1\in\xmo$, then the neighborhood $T_\xi(n)$ satisfies
\[\widetilde\til_{x_{n+k}\ldots x_2x_1}\subseteq T_\xi(n)\subseteq \widetilde\til_{x_n\ldots x_2x_1},\]
where $k$ is such that $\nuke^2|_{\alb^k}\subset\nuke$.

Let $A_{n, v}$ be ball of radius $2$ in the graph $\G_n$ of the action of $\nuke$ on $\alb^n$. We consider $A_{n, v}$ as a labeled rooted directed graph. There are only finitely many isomorphism classes of graphs $A_{n, v}$ for all $n, v$.

\begin{proposition}
\label{pr:Sv1Sv2}
Suppose that $A_{n_1, v_1}$ and $A_{n_2, v_2}$ are isomorphic. Then the map \[S_{v_1, v_2}\colon \bigcup_{g\in\nuke^2}\xmo g(v_1)\arr\bigcup_{g\in\nuke^2}\xmo g(v_2)\colon wg(v_1)\mapsto wg(v_2)\]
induces a well defined homeomorphism $\widetilde\til_{v_1}\arr\widetilde\til_{v_2}$.
\end{proposition}

\begin{proof}
Since $A_{n_1, v_1}$ and $A_{n_2, v_2}$ are isomorphic, $g_1(v_1)=g_2(v_1)$ for $g_1, g_2\in\nuke^2$ is equivalent to $g_1(v_2)=g_2(v_2)$, hence the map $wg(v_1)\mapsto wg(v_2)$ is well defined.

Since the cylindrical sets $\xmo v$ are compact, it is enough to show that the maps $S_{v_1, v_2}$ preserve the equivalence relation. A sequence $\ldots x_2x_1g_1(v_1)$ is equivalent to $\ldots y_2y_1g_2(v_1)$ if and only if there exists a sequence $s_n\in\nuke$ such that $s_n\cdot x_n=y_n\cdot s_{n-1}$ and $s_0g_1(v_1)=g_2(v_1)$. The latter condition is equivalent to $s_0g_1(v_2)=g_2(v_2)$. Therefore, the same sequence $s_n$ will show that $S_{v_1, v_2}(\ldots x_2x_1g_1(v_1))=S_{v_1, v_2}(\ldots y_2y_1g_2(v_1))$.
\end{proof}

The following proposition is a version of a result of~\cite[Section~6]{haispilgr:subhyperbolic}. It also implies that $(\lims, d_{\nu, \alpha})$ is \emph{approximately self-similar} in the sense of~\cite{bourdonkleiner}. 

\begin{proposition}
\label{pr:localssimilarities}
Let $d_{\nu, \alpha}$ be a visual metric of exponent $\alpha$ associated with a quasi-morphism $\nu$. There exists $C>1$ such that for every pair $v_1, v_2$ such that the homeomorphism $S_{v_1, v_2}\colon \widetilde\til_{v_1}\arr\widetilde\til_{v_2}$ is defined,  we have
\[\frac{d_{\nu, \alpha}(S_{v_1, v_2}(\xi_1), S_{v_1, v_2}(\xi_2))}{d_{\nu, \alpha}(\xi_1, \xi_2)}\asymp_C e^{-\alpha(\nu(v_2)-\nu(v_1))}\]
for every pair $\xi_1, \xi_2\in\widetilde\til_{v_1}$ of different points.
\end{proposition}

\begin{proof}
Let $k_0\ge 1$ be such that $\nuke^2|_{\alb^{k_0}}=\nuke$.

Let $\xi_1, \xi_2\in \widetilde\til_{v_1}$. Then $\xi_i$ can be represented by sequences $\ldots x_2x_1g_1(v_1)$ and $\ldots y_2y_1g_2(v_1)$ for some $g_1, g_2\in\nuke$.  Then $\ell_\nu(\xi_1, \xi_2)\ge \nu(v_1)-C$ for some uniform constant $C$. Since $\nu$ is a quasi-morphism, we have
\[\ell_\nu(\xi_1, \xi_2)\doteq \nu(x_m\ldots x_2x_1)+\nu(v_1),\] where $m$ is the largest index such that $g(x_m\ldots x_2x_1)=y_m\ldots y_2y_1$ for some $g\in\nuke^3$. Similarly, 
\[\ell_\nu(S_{v_1, v_2}(\xi_1), S_{v_1, v_2}(\xi_2))\doteq \nu(x_m\ldots x_2x_1)+\nu(v_2).\] It follows that \[\ell_\nu(S_{v_1, v_2}(\xi_1), S_{v_1, v_2}(\xi_2))\doteq \ell_\nu(\xi_1, \xi_2)+\nu(v_2)-\nu(v_1),\] which implies the estimate of the proposition.
\end{proof}

Suppose now that $(G, \bim)$ is regular, see Definition~\ref{def:regularssgroup} and Theorem~\ref{th:regularimg}.

For a fixed metric $d$, we denote by $B_\xi(r)/D$ the ball (with respect to metric $d$) of radius $r$ with center in $\xi$ as a metric space with the metric $\frac 1Dd$.

The following proposition characterizes visual metrics as metrics, with respect to which local branches of $f^{-n}$ are quasisimilarities with a uniform bi-Lipschitz distortion.

\begin{proposition}
\label{pr:uniformlocalsim}
A metric $d$ on $\lims$ is a visual metric associated with a $G$-quasi-invariant positive quasi-morphism $\nu\colon\xs\arr[0, \infty)$ if and only if there exist $L>1$ and $\delta>0$ such that for every $\xi\in\lims$ and every $n\ge 1$ there exists a neighborhood $U$ of $\xi$ such that $f^n\colon U\arr B_{f^n(\xi)}(\delta)/D$ is an $L$-bi-Lipschitz homeomorphism for some $D>0$ (depending on $\xi$ and $n$).
\end{proposition}

\begin{proof}
Let us first prove that a visual metric $d_{\nu, \alpha}$ satisfies the conditions of the proposition.

By the regularity condition, there exists $k_0$ such that for every $g\in\nuke^3$ and every $v\in\alb^{k_0}$ we have either $g(v)\ne v$ or $g(v)=v$ and $g|_v=1$.

Then the isomorphism class of $A_{n, u}$ for $u\in\alb^n$, $n\ge k_0$, depends only on the prefix of length $k$ of the word $u$.
Consequently, the restriction of $f^n$ to $\widetilde\til_{v_0v}$ for $v_0\in\alb^{k_0}, v\in\alb^n$ is the homeomorphism $S_{v_0v, v_0}\colon \widetilde\til_{v_0v}\arr\widetilde\til_{v_0}$, satisfying Proposition~\ref{pr:localssimilarities}. Hence there exists a constant $L>1$, not depending on $v_0$ and $v$, such that $f^n$ restricted to $\widetilde\til_{v_0v}$ is an $L$-bi-Lipschitz homeomorphism to $\frac 1D\widetilde\til_{v_0}$ for some $D$.

For every $\xi\in\til_{v_0}$ we have $T_\xi(k_0)\subset\widetilde\til_{v_0}$. There exists $\delta>0$ such that $B_\xi(\delta)\subset T_\xi(k_0)$ for all $\xi\in\lims$. Then, for every $\xi\in\lims$ and $n$, we can take $U$ equal to the preimage of $B_{f^n(\xi)}(\delta)$ under the homeomorphism $S_{v_0v, v_0}\colon \widetilde\til_{v_0v}\arr\widetilde\til_{v_0}$, where $v_0\in\alb^{k_0}$ and $v\in\alb^n$ are such that $\xi\in\til_{v_0v}$.

Conversely, suppose that $d$ satisfies the conditions of the proposition. There exists $k_1\ge k_0$ such that $\diam(\til_v)<\delta$ for all $v\in\alb^{k_1}$.

Let $\xi$ be a point of $\lims$ represented by a sequence $\ldots x_2x_1\in\xmo$.
We have \[f^{n-k_1}(\til_{x_n\ldots x_1})\subset\til_{x_n\ldots x_{n-k_1+1}}\subset B_{f^{n-k_1}(\xi)}(\delta).\]
Consequently, there exists $D=D_{x_n\ldots x_2x_1}$ such that we have 
\begin{equation}
\label{eq:Dv}
d(f^{n-k_1}(\xi_1), f^{n-k_1}(\xi_2))\asymp_L D_{x_n\ldots x_2x_1}d(\xi_1, \xi_2)
\end{equation}
for all pairs of different points $\xi_1, \xi_2\in\til_{x_n\ldots x_2x_1}$.

Consider a word $x_{n+k_1}\ldots x_2x_1$, and two points $\xi_1, \xi_2\in\til_{x_{n+k_1}\ldots x_2x_1}$. Denote $\zeta_i=f^{n-k_1}(\xi_i)$. 
Since $\xi_i\in\til_{x_n\ldots x_2x_1}$, the estimate~\eqref{eq:Dv} is written as
\[d(\zeta_1, \zeta_2)\asymp_L D_{x_n\ldots x_2x_1}d(\xi_1, \xi_2).\]
We have $\zeta_i\in\til_{x_{n+k_1}\ldots x_{n-k_1+1}}$, so, using~\eqref{eq:Dv} with $\xi_i$ replaced by $\zeta_i$ and $n$ by $2k_1$, we get
\[d(f^n(\xi_1), f^n(\xi_2))=d(f^{k_1}(\zeta_1), f^{k_1}(\zeta_2))\asymp_L
D_{x_{n+k_1}\ldots x_{n-k_1+1}}d(\zeta_1, \zeta_2)\asymp_{L_1} d(\zeta_1, \zeta_2),\]
where $L_1=\max_{v\in\alb^{2k_1}}D_v$.
It follows that
\begin{equation}
\label{eq:Dv2}
d(f^n(\xi_1), f^n(\xi_2))\asymp_{L^2L_1}D_{x_n\ldots x_2x_1}d(\xi_1, \xi_2)
\end{equation}
for all $n\ge k_1$ and $\xi_1, \xi_2\in\til_{x_{n+k_1}\ldots x_2x_1}$. 

For every $v\in\alb^{n+k_1}$, the set $f^n(\til_v)$ is a tile of the level $k_1$, so it has diameter less than $\delta$. Consequently, by the condition of the proposition, the restriction of $f^n$ to $\til_v$ is bi-Lipschitz.
Since the number of words $v\in\xs$ of a given length is finite, we can choose $L$ large enough and define $D_v=1$ for all $|v|<k_1$, so that the estimate~\eqref{eq:Dv} holds for all words $x_n\ldots x_2x_1$ and all $\xi_1, \xi_2$ contained in one tile of the level $n+k_1$.

Consider now a word $x_{n+m+k_1}\ldots x_2x_1$ for some $n, m\ge 0$. Choose two points $\xi_1, \xi_2\in\til_{x_{n+m+k_1}\ldots x_2x_1}$. Then, by the proven above, we have
\[d(f^{n+m}(\xi_1), f^{n+m}(\xi_2))\asymp_{L^2L_1} D_{x_{n+m}\ldots x_2x_1}d(\xi_1, \xi_2).\]

We have $f^m(\xi_i)\in \til_{x_{n+m+k_1}\ldots x_{m+2}x_{m+1}}$, so
\[d(f^{n+m}(\xi_1), f^{n+m}(\xi_2))\asymp_{L^2L_1} D_{x_{n+m}\ldots x_{m+2}x_{m+1}}d(f^m(\xi_1), f^m(\xi_2)).\]
We also have $\xi_i\in\til_{x_{m+k_1}\ldots x_2x_1}$, so
\[d(f^m(\xi_1), f^m(\xi_2))\asymp_{L^2L_1} D_{x_m\ldots x_2x_1}d(\xi_1, \xi_2).\]

We proved that
\begin{equation}
\label{eq:Dmorphism}
D_{x_{n+m}\ldots x_1}\asymp_{L^4L_1} D_{x_{n+m}\ldots x_{n+1}}D_{x_n\ldots x_2x_1}
\end{equation}
for all words $x_{n+m}\ldots x_2x_1$.

Since $f$ is locally expanding, $\max_{v\in\alb^n}D_v$ converges to $0$ as $v\to\infty$. Consequently, the estimate~\eqref{eq:Dmorphism} implies that $\nu(v)=-\log D_v$ extends to a positive quasi-morphism $\nu\colon \xs\arr[0, \infty)$.

Let us show that $\nu$ is $G$-quasi-invariant. Let $v\in\xs$ and $g\in\nuke$. Then $\til_v\cap\til_{g(v)}\ne\emptyset$. Let $\xi\in\til_v\cap \til_{g(v)}$, and let $\ldots x_2x_1, \ldots y_2y_1$ be sequences representing $\xi$ such that $x_n\ldots x_2x_1=v$ and $y_n\ldots y_2y_1=g(v)$. We assume that $n\ge k_1$. Then $f^{n-k_1}(\til_v\cup\til_{g(v)})$ is contained in $B_{f^n(\xi)}(\delta)$, so 
\[d(f^{n-k_1}(\xi_1), f^{n-k_1}(\xi_2))\asymp_L Dd(\xi_1, \xi_2)\]
for all $\xi_1, \xi_2\in\til_v\cup\til_{g(v)}$. It follows that $D\asymp_{L^2} D_v$ and $D\asymp_{L^2} D_{g(v)}$. Consequently, $D_v\asymp_{L^4} D_{g(v)}$, hence $\nu$ is $G$-quasi-invariant.

Let us show that $d$ is bi-Lipschitz equivalent to a visual metric $d_{\nu, 1}$ of exponent $\alpha=1$ associated with $\nu$. Let $\ldots x_2x_1, \ldots y_2y_1\in\xmo$ represent the points $\xi_1, \xi_2\in\lims$, and let $n$ be the largest index such that $g(x_n\ldots x_2x_1)=y_n\ldots y_2y_1$ for some $g\in\nuke$. Then $d_{\nu, 1}(\xi_1, \xi_2)\asymp e^{-\nu(x_n\ldots x_2x_1)}$. 

On the other hand, by the same arguments as in the proof of $G$-quasi-invariance of $\nu$, we have \[d(f^{n-k_1}(\xi_1), f^{n-k_1}(\xi_2))\asymp e^{\nu(x_n\ldots x_2x_1)}d(\xi_1, \xi_2).\]

By Proposition~\ref{pr:localssimilarities}, we have
\[d_{\nu, 1}(f^{n-k_1}(\xi_1)-f^{n-k_1}(\xi_2))\asymp e^{\nu(x_n\ldots x_2x_1)}d_{\nu, 1}(\xi_1, \xi_2)\asymp 1.\]
Since $\lims$ is compact, this implies that there exists a uniform constants $\epsilon>0$ independent of $\xi_1, \xi_2$ such that 
\[\epsilon^{-1}>d(f^{n-k_1}(\xi_1), f^{n-k_1}(\xi_2))>\epsilon.\]
Consequently, $d(\xi_1, \xi_2)\asymp e^{-\nu(x_n\ldots x_2x_1)}\asymp d_{\nu, 1}(\xi_1, \xi_2)$.
\end{proof}

\begin{example}
Let $f$ be a hyperbolic complex rational function. It follows from Koebe distortion theorem that the spherical metric on the Riemann sphere restricted to the Julia set of $f$ satisfies the condition of Proposition~\ref{pr:uniformlocalsim}. (See also a direct proof of it in~\cite[Proposition~3.7]{nek:psmeasure}.) Consequently, it is bi-Lipschitz to a visual metric associated with a quasi-morphism.
\end{example}

\subsection{Lifting metrics and measures from $\lims$ to $\limg$}
\label{ss:liftingmm}

Let $d$ be a metric on $\lims$. Let $P\colon \limg\arr\lims$ be the quotient map.

\begin{defi}
    \label{def:liftofmetric}
    The \emph{lift} of $d$ to $\limg$ is defined for $\xi_1, \xi_2\in\limg$ as the infimum $\tilde d(\xi_1, \xi_2)$ of the expression
\[\sum_{i=1}^m\diam(P(\til\otimes v_i)),\]
over all sequences $v_1, v_2, \ldots, v_m\in\bim^{\otimes n}$ and all $n$ such that $\xi_1\in\til\otimes v_1$, $\xi_2\in\til\otimes v_m$ and $\til\otimes v_i\cap\til\otimes v_{i+1}\ne\emptyset$ for all $i=1, 2, \ldots, m-1$. The function $\tilde d(\xi_1, \xi_2)$ is defined to be infinite if such a sequence does not exist.
\end{defi}

\begin{proposition}
The function $\tilde d$ is a $G$-invariant metric (with infinite distances allowed) compatible with the topology on $\limg$.

The map $P\colon(\limg, \tilde d)\arr(\lims, d)$ is distance non-increasing. For every $v\in\xs$ the diameter of $\til\otimes v$ in $(\limg, \tilde d)$ is equal to the diameter of $\til_v=P(\til\otimes v)$ in $(\lims, d)$.
\end{proposition}

\begin{proof}
It is easy to check that 
the defined function $\tilde d$ satisfies the triangle inequality, is symmetric, and $G$-invariant. 

Let $\xi_1, \xi_2\in\limg$. 
Let $(v_i)$ be a sequence satisfying the conditions of Definition~\ref{def:liftofmetric}. Choose a point $\zeta_i\in\til\otimes v_i\cap\til\otimes v_{i+1}$ for every $i=1, 2, \ldots, m-1$, and set $\zeta_0=\xi_1$ and $\zeta_m=\xi_2$. Then we have \[d(P(\xi_1), P(\xi_2))\le
\sum_{i=0}^{m-1}d(P(\zeta_i), P(\zeta_{i+1}))\le 
\sum_{i=1}^m\diam(P(\til\otimes v_i)),\]
since $\zeta_{i-1}, \zeta_i\in\til\otimes v_i$. Consequently, $d(P(\xi_1), P(\xi_2))\le\tilde d(\xi_1, \xi_2)$.

If $\xi_1\ne\xi_2$, then there exists $n$ such that the neighborhoods $T_n(\xi_i)$ are disjoint. Then $P(T_n(\xi_i))$ is a neighborhood of $P(\xi_i)$, since $P$ is open (see, for example, the proof of~\cite[Proposition~2.19]{nek:dimension}). Let $\epsilon>0$ be the distance from $P(\xi_1)$ to the complement of $P(T_n(\xi_1))$. Then for any chain in the definition of $\tilde d(\xi_1, \xi_2)$ the corresponding sum of diameters is greater than the minimal diameter of $\til_v$ for $|v|<n$, or is greater than $\epsilon$. It follows that $\tilde d(\xi_1, \xi_2)>0$ for any $\xi_1\ne \xi_2$.

For any two points $\xi_1, \xi_2\in \til\otimes v$ we have $\tilde d(\xi_1, \xi_2)\le\diam(P(\til\otimes v))$, hence $\diam_{\tilde d}(\til\otimes v)\le\diam_d(P(\til\otimes v))$. On the other hand, since $P(\til\otimes v)$ is compact, there exist $\xi_1, \xi_2\in\til\otimes v$ such that $d(P(\xi_1), P(\xi_2))=\diam(P(\til\otimes v))$. Then $\tilde d(\xi_1, \xi_2)\ge d(P(\xi_1), P(\xi_2))=\diam(P(\til\otimes v))$.
Consequently, 
\begin{equation}
    \label{eq:diamforlift}
    \diam(\til\otimes v)=\diam(P(\til\otimes v))
\end{equation}
for all $v\in\bim^*$.
 
The equality $\diam(\til\otimes v)=\diam(P(\til\otimes v))$ implies that this metric is compatible with the topology, since the sets $T_n(\xi)$ form a basis of neighborhoods of a point $\xi$ both in $\limg$ and in $\lims$.
\end{proof}

The action of $G$ on $\limg$ is proper and topologically free (i.e., the set of fixed points of every non-trivial element of $G$ is nowhere dense), see~\cite[Proposition~4.6.2]{nek:book}. Therefore, we can pull back a measure $\mu$ on $\lims$ to a measure $\mu^*$ on $\limg$ given by the condition 
\[\int f(\xi)\;d\mu^*(\xi)=\int\sum_{P(\xi)=\zeta}f(\xi)\;d\mu(\zeta)\]
for all compactly supported continuous functions $f\colon\limg\arr\C$.

Note that since the projection $P\colon \limg\arr\lims$ is at most $|\nuke|$-to-one, we have
\begin{equation}
    \label{eq:measureforlift}
    \mu(P(\til\otimes v))\le\mu^*(\til\otimes v)\le |\nuke|\cdot \mu(P(\til\otimes v))
\end{equation}
for all $v\in\bim^*$.

\section{Combinatorial modulus}

This section is based on ~\cite{carrascopiaggio}. Its results do not apply directly to the general limit space of a
contracting group, since we do not assume local connectedness.

 On the other hand, our setting is much more restrictive, so some parts of the proof of the main result of~\cite{carrascopiaggio} are easier in our setting (especially when using the techniques of self-similar groups).

We also establish a connection between conformal dimension and $G$-quasi-invariant quasi-morphisms.

\subsection{Definition}

Fix a self-similar contracting group $(G, \bim)$, and choose a basis $\alb$.  Let $\nuke$ be the corresponding nucleus, and let $S\subset G$ be a finite symmetric (i.e., such that $S=S^{-1}$) set containing the nucleus.  Let $\G_{n, S}=\G_n$ be the corresponding graph of the action on the level $\alb^n$ of $\xs$. It is the graph with the set of vertices $\alb^n$ and the set of directed edges $S\times\alb^n$, where the edge $(s, v)$ starts in $v$ and ends in $s(v)$. If we label the edge $(s, v)$ by $(s, s|_v)$, we get the \emph{dual Moore diagram} of the generating set.

We will write paths in $\G_n$ as sequences $(s_m, v_m), \ldots, (s_2, v_2), (s_1, v_1)$, where $s_i\in S$, $v_i\in\alb^n$, and $s_i(v_i)=v_{i+1}$.

Let $\mathcal{A}$ be a class of functions $\phi\colon\alb^n\arr [0, \infty)$ or $\phi\colon S\times\alb^n\arr [0, \infty)$ on the set of vertices or edges of $\G_{n, S}$ (we will choose the appropriate classes $\mathcal{A}$ later). We denote then, for $p>0$, the \emph{combinatorial modulus}
\[M_{\mathcal{A}, p}(n)=\inf_{\phi\in\mathcal{A}}\sum_x \phi(x)^p,\]
where the sum is taken over all $x\in\alb^n$ or $x\in S\times\alb^n$, respectively.

We also denote 
\[M_{\mathcal{A}, \infty}(n)=\inf_{\phi\in\mathcal{A}}\max_x \phi(x),\]
where the maximum is also taken over all $x\in\alb^n$ or $x\in S\times\alb^n$, respectively.

We will now define several classes of \emph{admissible functions} $\mathcal{A}$. Different classes will be convenient in different situations, but we will show that the corresponding combinatorial moduli are equivalent in the sense that any two classes $\mathcal{A}, \mathcal{B}$ satisfy $M_{\mathcal{A}, p}\asymp_C M_{\mathcal{B}, p}$, where the constant $C$ does not depend on $n$.

If $(e_m, \ldots, e_2, e_1)$ is an edge-path in $\G_n$, where $e_i=(s_i, v_i)$, then we have $v_{i+1}=s_i(v_i)$, so that the end of the path is equal to the image of the beginning under the action of the product $g=s_m\cdots s_2s_1$ of the first halves of the labels in the dual Moore diagram. We also have $(s_m\cdots s_2s_1)|_{v_1}=s_m|_{v_m}\cdots s_2|_{v_2}s_1|_{v_1}$, i.e., the section $g|_{v_1}$ is equal to the product of the second halves of the labels of the path in the dual Moore diagram.

If $\phi$ is a function on the set of vertices or the set of edges of $\G_{n, S}$, and $\gamma$ is a path in $\G_{n, S}$, then we denote by $\phi(\gamma)$ the sum of values of $\gamma$ along the path (in all vertices or all edges of $\gamma$, respectively).

We say that the path is \emph{essential} if $(s_m\cdots s_1)|_{v_1}\notin\nuke$, i.e., if the product of the corresponding sections $s_i|_{v_i}$ is not in the nucleus.  More generally, if $N$ is any finite subset containing the nucleus, then we say that the path is \emph{$N$-essential} if $(s_m\cdots s_1)|_{v_1}\notin N$.

We say that a function $\phi\colon\alb^n\arr [0, \infty)$ (resp.\ a function $\phi\colon\nuke\times\alb^n\arr [0, \infty)$ on the set of edges of $\G_n$) is \emph{admissible} if for every essential path $\gamma$ in $\G_n$ we have $\phi(\gamma)\ge 1$. 

When describing a dual Moore diagram in this section, we label an edge $(s, x)$ just by $s|_x$. Then, for a path starting in $v$ and corresponding to a product $s_1s_2\cdots s_n$, the section $(s_1s_2\cdots s_n)|_v$ is equal to the product of the labels along the path. Moreover, out of two arrows $(s, v)$ from $v$ to $s(v)$ labeled by $s|_v$ and $(s^{-1}, s(v))$ labeled by $s|_v^{-1}=s^{-1}|_{s(v)}$, we leave only one, and allow paths to go against the arrows, with the agreement that if we go agains an arrow, we invert its label.

\begin{example}
Consider the adding machine from Example~\ref{ex:angledoubling}. Its $n$th level dual Moore diagram is a cycle of $2^n$ edges, with all labels trivial except for the arrow from $11\ldots 1$ to $00\ldots 0$ labeled by $a$. Then every essential path makes a full loop and contains all arrows of the diagram. 
\end{example}

\begin{example}
The $n$th level dual Moore diagram of the Grigorchuk group~\cite{grig85_en} is a chain from $11\ldots 10$ to $11\ldots 11$ consisting of multiple edges and loops such that all their labels are trivial except for the labels $b, c, d$ of three loops at $11\ldots 11$ and two loops labeled by $a$ at $11\ldots 10$. Since $\{1, b, c, d\}$ and $\{1, a\}$ are groups contained in the nucleus, every essential path must pass through $11\ldots 10$ and $11\ldots 11$.
\end{example}

\begin{proposition}
Let $\mathcal{A}$ and $\overline{\mathcal{A}}$ be the sets of admissible functions on $\alb^n$ and $S\times\alb^n$, respectively. Then for every $p\in [1, \infty]$ we have $M_{\mathcal{A}, p}(n)\asymp_{C_p} M_{\overline{\mathcal{A}}, p}(n)$.
\end{proposition}

\begin{proof}
Let $\phi\colon\alb^n\arr[0, \infty)$ be an admissible function.
If $e$ is an edge starting in $v_1$ and ending in $v_2$, then define $\phi_1(e)=\phi(v_1)+\phi(v_2)$. Then $\phi_1$ is an admissible function on the set of edges and we have \[\phi_1^p(e)=(\phi(v_1)+\phi(v_2))^p\le \max(1, 2^{p-1})(\phi(v_1)^p+\phi(v_2)^p),\]
if $p$ is finite.

 It follows that \[\sum\phi_1^p(e)\le \max(1, 2^{p-1})|S|\sum_{v\in\alb^n}\phi(v)^p.\] Consequently, \[M_{\overline{\mathcal{A}}, p}(n)\le\max(1, 2^{p-1})|S| \cdot M_{\mathcal{A}, p}(n).\]
  
 We also obviously have $\max\phi_1(e)\le 2\max\phi(v)$, so $M_{\overline{\mathcal{A}}, \infty}(n)\le 2M_{\mathcal{A}, \infty}(n)$.

On the other hand, suppose that $\phi$ is an admissible function on the set of edges. Define $\phi_0(v)$ as the sum $\sum_{s\in S}\phi(s, v)$ of values of $\phi$ in all edges adjacent to $v$. Then $\phi_0$ is admissible. We have $\phi_0(v)^p\le |S|^{p-1}\sum_{s\in\nuke} \phi(s, v)^p$, hence \[\sum_{v\in\alb^n}\phi_0(v)^p\le 2|S|^{p-1}\sum_{s\in\nuke, v\in\alb^n}\phi(s, v)^p.\]

 Consequently, $M_{\mathcal{A}, p}(n)\le 2|S|^{p-1} M_{\overline{\mathcal{A}}, p}(n)$.
Similarly, $M_{\mathcal{A}, \infty}(n)\le 2|S| M_{\overline{\mathcal{A}}, \infty}(n)$.
\end{proof}

\begin{proposition}
\label{pr:moduluschanginggenset}
Let $N$ be a finite subset of $G$ containing the nucleus $\nuke$, and let $S$ be a finite symmetric generating set, also containing the nucleus.
Let $\mathcal{A}_\nuke$ be the set of $\nuke$-admissible functions on the set of edges of the graph $\G_{n, \nuke}$. Let $\mathcal{A}_S$ be the set of $N$-admissible functions on the set of edges of $\G_{n, S}$.

Then $M_{\mathcal{A}_\nuke, p}(n)\asymp_C M_{\mathcal{A}_S, p}(n)$, where $C$ does not depend on $n$.
\end{proposition}

\begin{proof}
There exists $n_1$ such that for every $s\in S$ and every $v\in\alb^{n_1}$ we have $g|_v\in\nuke$.
Let $n>n_1$, and consider the map $P\colon \G_{n, S}\arr\G_{n-n_1, \nuke}$ mapping an edge $(s, x_1x_2\ldots x_n)$ of $\G_{n, S}$ to the edge $g|_{x_1x_2\ldots x_{n_1}}, x_{n_1+1}x_{n_1+2}\ldots x_n)$ of $\G_{n-n_1, \nuke}$. In other words, we remove the prefix of length $n_1$ and replace $s$ by the corresponding section. The map $P$ is  $|\alb^{n_1}|$-to-one.

Let $\phi$ be an admissible (i.e., $\nuke$-admissible) function on the set of edges of $\G_{n-n_1, \nuke}$. Define the function $\tilde\phi$ on $\G_{n, \mathcal{S}}$ by $\tilde\phi(e)=\phi(P(e))$.

For every $N$-essential edge path $\gamma$ in $\G_{n, S}$, the corresponding sections of $\gamma$ and $P(\gamma)$ coincide. Since the former does not belong to $N$, the latter does not belong to $\nuke\subset N$. Consequently, $P(\gamma)$ is also essential. Consequently, the sum of values of $\phi$ along the edges of $P(\gamma)$ is not less than 1. But it is equal to the sum of values of $\tilde\phi$ along the edges of $\gamma$. Consequently, $\tilde\phi$ is admissible. We have $\sum\left(\tilde\phi(e)\right)^p\le |\alb|^{n_1}\sum\phi(e)^p$. Consequently, \[M_{\mathcal{A}_S, p}(n)\le |\alb|^{n_1}\cdot M_{\mathcal{A}_\nuke, p}(n)\] for all $n$.

Suppose now that $\phi$ is an $N$-admissible function on $\G_{n, S}$. Let $n_2$ be such that for every $h\in N$ and $v\in\alb^{n_2}$ we have $h|_v\in\nuke$.
Let $\gamma$ be a path in $\G_{n, \nuke}$ connecting a vertex $x_1x_2\ldots x_n$ to $y_1y_2\ldots y_n$, and let $g\in G$ be the corresponding product such that $g(x_1x_2\ldots x_n)=y_1y_2\ldots y_n$. There exists $r$ such that every element of $\nuke$ is a product of elements of $S$.

Let $\gamma'$ be the path in $\G_{n, S}$ obtained by replacing each edge $(s_i, v_i)$ by the path in $\G_{n, S}$ corresponding to the decomposition of $s_i$ as a product of elements of $S$. The path $\gamma'$ has the same endpoints as $\gamma$, and the same product $g$.

Suppose that $\gamma$ is essential, i.e., $g|_{x_1x_2\ldots x_n}\notin\nuke$. Since $N|_{x_{n-n_2+1}\ldots x_n}\subset\nuke$, it follows that $g|_{x_1x_2\ldots x_{n-n_2}}\notin N$. Consequently, the projection $\gamma''$ of the path $\gamma'$ to a path in $\G_{n-n_2, S}$ obtained by removing the suffixes of length $n_2$ is $N$-essential.

Let $\phi$ be an $N$-admissible function on $\G_{n-n_2, S}$. Define $\tilde\phi$ on an edge $(s, x_1x_2\ldots x_n)$ of $\G_{n, \nuke}$ as the sums of values of $\phi$ along the corresponding path in $\G_{n-n_2, S}$ of length at most $r$ starting in $x_1x_2\ldots x_{n-n_2}$ (giving $s$ as the product of the first labels of the edges). Then the sum of values of $\tilde\phi$ along $\gamma$ will be equal to the sum of values of $\phi$ along $\gamma''$. It follows that $\tilde\phi$ is admissible.

If an edge $(s, x_1x_2\ldots x_{n-n_2})$ of $\G_{n-n_2, S}$ appears in the sum defining the value of $\tilde\phi$ for an edge $(s, z_1z_2\ldots z_n)$ of $\G_{n, \nuke}$, then $z_1z_2\ldots z_{n-n_2}$ is on distance at most $r$ from $x_1x_2\ldots x_{n-n_2}$ in $\G_{n-n_2, S}$. It follows that a given edge of $\G_{n-n_2, S}$ appears in these sums at most $D=\deg(\bim)^{n_2}(1+|S|+\cdots+|S|^r)$ times. Using the inequality $(a_1+\cdots+a_k)^p\le k^{p-1}(a_1^p+\cdots a_k^p)$, we conclude that the sum of values of the $p$th powers of $\tilde\phi$ is not larger than $r^{p-1}D\sum\phi(e)^p$. Consequently, 
\[M_{\mathcal{A}_\nuke, p}(n)\le r^{p-1}D M_{\mathcal{A}_S p}(n)\]
for all $n$, which finishes the proof.
\end{proof}

\begin{proposition}
\label{pr:moduluschangealb}
Let $\mathcal{A}_\alb, \mathcal{A}_{\mathsf{Y}}$ be the sets of admissible functions on the sets of vertices of graphs $\G_n$ defined using different bases $\alb$ and $\mathsf{Y}$ of $\bim$, then 
\[M_{p, \mathcal{A}_\alb}(n)\asymp M_{p, \mathcal{A}_{\mathsf{Y}}}(n).\]
\end{proposition}

\begin{proof}
We identify the trees $\xs$ and $\mathsf{Y}^*$ by the natural isomorphism mapping $v\in\xs$ to the unique representative in $\mathsf{Y}^*$ of the right orbit $v\cdot G$.
 Let $\nuke_\alb$  and $\nuke_{\mathsf{Y}}$ be the corresponding nuclei. 
  
 By Proposition~\ref{pr:sumgeometric}, there exists a finite set $A\subset G$ such that if $h_1$ is the section of an element $g\in G$ in a word $v_1\in\xs$ (for the standard action associated with $\alb$), and $h_2$ is the section of $g$ in the corresponding word $v_2\in\mathsf{Y}^*$ (such that $v_2\in v_1\cdot G$), then there exist $f_1, f_2\in B$ such that $h_2=f_1^{-1}h_1f_2$.
 In particular, if the $\alb$-section of $g$ in $v$ belongs to $\nuke_\alb$, then the $\mathsf{Y}$-section of $g$ in the corresponding word belongs to $B^{-1}\nuke_\alb B$. Note that this implies that $\nuke_{\mathsf{Y}}\subset B^{-1}\nuke_\alb B$, since all sections in long enough words must belong to the nucleus. We may assume that $B$ contains the identity, so we also have $\nuke_\alb\subset B^{-1}\nuke_\alb B$.

Consequently, if a path $\gamma$ in $\G_{n, \nuke_\alb\cup\nuke_{\mathsf{Y}}}$ is such that the corresponding product of sections defined using $\alb$ does not belong to $B^{-1}\nuke_\alb B$, then the product of the sections defined using $\mathsf{Y}$ does not belong to $\nuke_{\mathsf{Y}}$. 

In other words, every path in $\G_{n, \nuke_\alb\cup\nuke_\mathsf{Y}}$ that is $B^{-1}\nuke_\alb B$-essential in the sense of sections defined using $\alb$ is $\nuke_{\mathsf{Y}}$-essential in the sense of sections defined using $\mathsf{Y}$. 

Consequently, every $\nuke_{\mathsf{Y}}$-admissible function in the sense of sections defined by $\mathsf{Y}$ is $B^{-1}\nuke_\alb B$-admissible in the sense of sections defined using $\alb$. This, together with Proposition~\ref{pr:moduluschanginggenset}, finishes the proof.
\end{proof}

\subsection{Asymptotic $p$-modulus}

For a function $\phi\colon\alb^n\arr[0, \infty)$, denote 
\[\hat\phi(v)=\min_{s\in\nuke}\phi(s(v)).\] 



Let $\widehat{\mathcal{A}}$ be the set of functions $\phi\colon\alb^n\arr[0, \infty)$ such that for every essential vertex-path $v_1, v_2, \ldots, v_m$ in the graph $\G_{n, \nuke}$ we have
\[\hat\phi(v_1)+\hat\phi(v_2)+\cdots+\hat\phi(v_{m-2})>1.\]

We say that the path is \emph{almost essential} if it can be continued by one edge (at the end of the directed path) to an essential path.

\begin{lemma}
\label{lem:hatequivalent}
Let $\mathcal{A}$ be the set of admissible functions $\alb^n\arr [0, \infty)$. Then there exists $C$, not depending on $n$, such that $M_{\mathcal{A}, p}(n)\asymp M_{\widehat{\mathcal{A}}, p}(n)$.
\end{lemma}

\begin{proof}
We obviously have $\widehat{\mathcal{A}}\subset\mathcal{A}$. Consequently, $M_{\widetilde{\mathcal{A}}, p}(n)\ge M_{\mathcal{A}, p}$. 

Let $\phi\in \mathcal{A}$. Consider the function $\psi(v)=\sum_{s\in\nuke^3}\phi(s(v))$. Then for every $v\in\alb^n$ there exists $s\in\nuke$ such that $\hat\psi(v)=\psi(s(v))$. For any $s_1, s_2\in\nuke$, the set $\nuke^3s$ contains $1$, $s_1$ and $s_2s_1$, hence
\[\hat\psi(v)=\psi(s(v))=\sum_{h\in\nuke^2}\phi(hs(v))\ge \phi(v)+\phi(s_1(v))+\phi(s_2s_1(v)).\]

Consequently, for any essential vertex-path $v_1, v_2, \ldots, v_m$, we have $\hat\psi(v_i)\ge \phi(v_i)$ for $i=1, 2, \ldots, m-1$, and $\hat\psi(v_{m-2})\ge\phi(v_{m-2})+\phi(v_{m-1})+\phi(v_m)$. Consequently,
\[\hat\psi(v_1)+\hat\psi(v_i)+\cdots+\hat\psi(v_{m-2})\ge\phi(v_1)+\phi(v_2)+\cdots+\phi(v_m)>1.\]
It follows that $\psi\in\widehat{\mathcal{A}}$.

We have $(\psi(v))^p\le |\nuke^3|^{p-1}\sum_{s\in\nuke^3}(\phi(s(v))^p$, consequently,
\[\sum_{v\in\alb^n}(\psi(v))^p\le |\nuke^3|^{p-1}\sum_{v\in\alb^n}\sum_{s\in\nuke^3}(\phi(s(v))^p=|\nuke^3|^p\sum_{v\in\alb^n}(\phi(v))^p,\]
hence $M_{\widehat{\mathcal{A}}, p}(n)\le |\nuke^3|^pM_{\mathcal{A}, p}(n)$.
\end{proof}

\begin{lemma}
\label{lem:productadmissible}
Let $\phi_1\colon\alb^{n_1}\arr [0, \infty)$ and $\phi_2\colon\alb^{n_2}\arr [0, \infty)$ belong to $\widehat{\mathcal{A}}$. Then the function $\phi\colon\alb^{n_1+n_2}\arr[0, \infty)$ given by $\phi(v_1v_2)=\phi_1(v_1)\phi_2(v_2)$ also belongs to $\widehat{\mathcal{A}}$.
\end{lemma}

\begin{proof}
Let $u_1w_1, u_2w_2, \ldots, u_mw_m$ be an essential vertex-path in $\alb^{n_1+n_2}$, where $u_i\in\alb^{n_1}$ and $w_i\in\alb^{n_2}$. Let $s_i\in\nuke$ be such that $s_i(u_iw_i)=u_{i+1}w_{i+1}$.

Set $k_0=1$, and define $k_{i+1}$ inductively, by the condition that $u_{k_i}, u_{k_1+1}, \ldots, u_{k_{i+1}}$ is the shortest almost essential beginning of the path $u_{k_i}, u_{k_i+1}, \ldots, u_m$, if it exists. Let $l$ be the maximal value for which $k_l$ is defined.

Consider one of the segments $u_{k_i}, \ldots, u_{k_{i+1}}$.
Let $k_i\le j<k_{i+1}$. We have
\[\hat\phi(u_jw_j)=\min_{s\in\nuke}\phi(s(u_jw_j))=\min_{s\in\nuke}\phi_1(s(u_j))\phi_2(s|_{u_j}(w_j))\ge\hat\phi_1(u_j)\min_{s\in\nuke}\phi_2(s|_{u_j}(w_j)).\]
Let $s\in\nuke$ be an element realizing the minimum $\min_{s\in\nuke}\phi_2(s|_{u_j}(w_j))$. By construction, the path $u_{k_i}, \ldots, u_j, s(u_j)$ is not essential, i.e., the section $h=(ss_j\cdots s_{k_i+1}s_{k_i})|_{u_{k_i}}$ belongs to the nucleus.

But we have $h(w_{k_i})=s|_{u_j}(w_j)$. Therefore, $\phi_2(s|_{u_j}(w_j))=\phi_2(h(w_{k_i}))\ge \hat\phi(w_{k_i})$.

We have shown that 
\[\hat\phi(u_jw_j)\ge \hat\phi_1(u_j)\hat\phi(w_{k_i}).\] 

Since the path $u_{k_i}, u_{k_i+1}, \ldots, u_{k_{i+1}}$ is almost essential, we can continue it by one edge to an essential path, so the definition of the class $\widehat{\mathcal{A}}$ implies that
\[\sum_{j=k_i}^{k_{i+1}-1}\hat\phi(u_jw_j)\ge \hat\phi(w_{k_i})\sum_{j=k_i}^{k_{i+1}-1}\hat\phi(u_j)>\hat\phi(w_{k_i}).\]

Since the paths $u_{k_i}w_{k_i}, \ldots, u_{k_{i+1}}w_{k_{i+1}}$ are not essential, there exists $h\in\nuke$ such that $w_{k_{i+1}}=h(w_{k_i})$. Consequently, we get a path $w_1=w_{k_0}, w_{k_1}, w_{k_2}, \ldots, w_{k_l}, w_m$ in $\alb^{n_2}$ with the same section as the original path in $\alb^{n_1+n_2}$. Since the original path is essential, we have $\sum_{i=0}^{l-1}\hat\phi_2(w_{k_i})>1$ if $k_l\ne m$, and $\sum_{i=0}^{l-2}\hat\phi_2(w_{k_i})>1$ if $k_l=m$.

If $k_l\ne m$, then $k_l-1<k_l<m$, so $k_l-1\le m-2$, and
\[\sum_{i=1}^{m-2}\hat\phi(u_iw_k)\ge \sum_{i=0}^{l-1}\hat\phi(w_{k_i})>1.\]
If $k_l=m$, then $k_{l-1}-1<k_{l-1}<m$, so $k_{l-1}-1\le m-2$, and
\[\sum_{i=1}^{m-2}\hat\phi(u_iw_k)\ge \sum_{i=0}^{l-2}\hat\phi(w_{k_i})>1,\]
which shows that $\phi\in\widehat{\mathcal{A}}$.
\end{proof}

As a corollary, we get the following.

\begin{proposition}
Let $\mathcal{A}$ be one of the classes of admissible or strongly admissible functions, defined above.

The limits
\[\theta_p=\lim_{n\to\infty} M_{\mathcal{A}, p}(n)^{\frac 1{np}},\qquad\theta_\infty=\lim_{n\to\infty}M_{\mathcal{A}, \infty}(n)^{\frac 1n}\]
exist for every $p>0$ and do not depend on the choice of the class $\mathcal{A}$ defined above.

The inequality $\theta_p<1$ is equivalent to the condition $\lim_{n\to\infty} M_{\mathcal{A}, p}(n)=0$ and to the condition $\inf_n M_{\mathcal{A}, p}(n)=0$.
\end{proposition}

\begin{proof} Let us prove the statement for finite $p$, the case $p=\infty$ is similar.
By Lemma~\ref{lem:hatequivalent}, it is enough to prove the proposition for the class $\widehat{\mathcal{A}}$. 

Let $\phi_1, \phi_2, \phi$ be as in Lemma~\ref{lem:productadmissible}.

Since $\sum_{w\in\alb^{n_1+n_2}}\phi(w)^p=\sum_{v\in\alb^{n_1}}\phi_1(v)^p\sum_{u\in\alb^{n_2}}\phi_2(u)^p$, Lemma~\ref{lem:productadmissible} implies the inequality we have \[M_{\widehat{\mathcal{A}}, p}(n_1+n_2)\le M_{\widehat{\mathcal{A}}, p}(n_1) M_{\widehat{\mathcal{A}}, p}(n_2),\]
which finishes the proof of our proposition by Fekete's lemma.
\end{proof}

We call $\theta_p=\theta_p(G, \bim)$ the \emph{asymptotic $p$-modulus}. It follows from Propositions~\ref{pr:moduluschangealb} that it does not depend on the choice of the basis of $\bim$. We have $\theta_p(G, \bim^{\otimes k})=\theta_p(G, \bim)^k$.

\begin{proposition}
\label{pr:asympmodulus}
The asymptotic moduli satisfy, for all $1<p<q\le\infty$, the inequalities:
    \[\theta_p\ge \theta_q,\qquad \mathsf{D}^{-1/p}\theta_p\le \mathsf{D}^{-1/q}\theta_q,\]
    where $\mathsf{D}=|\bim/G|$.
\end{proposition}

The above proposition implies  that $\theta_p$ is a continuous non-increasing function.

\begin{proof}
Since $M_p(n)^{1/p}$ for $p\in (1, \infty)$ and $M_\infty(n)$ for $p=\infty$ are equal to the infimum of the $\ell^p$-norms of admissible functions $\phi\colon \alb^n\arr [0, \infty)$, we get $M_p(n)^{1/p}\ge M_q(n)^{1/q}$ for all $1<p<q<\infty$ and $M_p(n)^{1/p}\le M_\infty(n)$ for all $1<p<\infty$. This implies the inequality $\theta_p\ge\theta_q$ for all $1<p<q\le\infty$.

In the other direction, H\"older's inequality for $\ell^p$-norms in $\R^m$, implies the inequality $M_p(n)^{1/p}\le \mathsf{D}^{n\left(\frac 1p-\frac 1q\right)}M_q(n)^{1/q}$ for all $1<p<q<\infty$ and $M_p(n)^{1/p}\le\mathsf{D}^{\frac np}$ for all $1<p<\infty$. Passing to the limit, we get the inequality $\mathsf{D}^{-1/p}\theta_p\le \mathsf{D}^{-1/q}\theta_q$.
\end{proof}

\begin{example}
\label{ex:Grigmodulus}
As we have seen above, a function on the set of edges of the dual Moore diagram of the adding machine group is strongly admissible if and only if the sum of its values is greater than or equal to $1$. It follows from the inequalities between generalized means that the $p$-modulus $M_p$ (for the class of strongly admissible functions) is realized by the constant function equal to $2^{-n}$. Consequently, $\theta_p=\lim_{n\to\infty}\left(\frac{2^n}{2^{np}}\right)^{1/np}=2^{\frac 1p-1}$.
The same arguments show that the Grigorchuk group has the same value $\theta_p=2^{\frac 1p-1}$ of the asymptotic modulus as the binary adding machine.
\end{example}

\begin{figure}
    \centering 
    \includegraphics{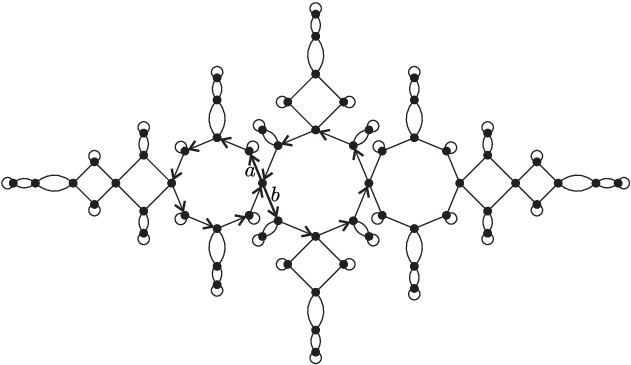}
    \caption{Schreier graph of $\img(z^2-1)$}
    \label{fig:basilicagraph}
\end{figure}

\begin{example}
\label{ex:thetapbasilica}
The dual Moore diagram for the action of the generators $a, b$ of $\img(z^2-1)$ on $\{0, 1\}^6$ is shown in Figure~\ref{fig:basilicagraph}. Any two simple cycles of the diagram are either disjoint or intersect in a single vertex. There are only two arrows with non-trivial labels: one labeled by $a$ and one labeled by $b$. Both of them start at $11\ldots 1$. The nucleus of $\img(z^2-1)$ is $\{1, a^{\pm 1}, b^{\pm 1}, (ab^{-1})^{\pm 1}\}$, see~\ref{ex:basilicanucleus}. Consequently, $M_p$ is realized on a constant function supported on two simple cycles containing the vertex $11\ldots 1$. The lengths of these cycles in a diagram of level $n$ are $2^{n/2}$ and $2^{n/2}$, if $n$ is even, and $2^{(n+1)/2}$ and $2^{(n-1)/2}$, if $n$ is odd. It follows that $\theta_p=(\sqrt{2})^{\frac 1p-1}$.
\end{example}

\subsection{Combinatorial modulus and conformal dimension}
\begin{theorem}
\label{th:modulus}
Let $(G, \bim)$ be a contracting self-similar group, let $\alb\subset\bim$ be a basis.

If  $p>\cdim(G, \bim)$, then $\theta_p(G, \bim)<1$.

If $\theta_p<1$, then there exists $n\ge 1$, a positive $G$-quasi-invariant quasi-morphism $\nu\colon \xs\arr (0, \infty)$ and a visual metric on $\lims$ associated with $\nu$ of Hausdorff dimension less than or equal to $p$. In particular, $\cdim(G, \bim)\le p$.
\end{theorem}

\begin{proof}
Let us prove the first statement of the theorem.
By Proposition~\ref{pr:diamoftiles} there exist a metric  and a measure $\mu$ on $\lims$, and a constant $C>1$ such that \[\diam(\til_v)^\alpha\asymp_C \mu(\til_v)\] for all $v\in\xs$. Consider the lifts of this metric and the measure $\mu$ to $\limg$, as described in Subsection~\ref{ss:liftingmm}.

Then there exists $C_1>1$ such that
\[\mu(\til\otimes v)\asymp_{C_1}\diam(\til\otimes v)^\alpha\]
for all $v\in\bim^*$.

Let $r_0>0$ be less than distance between any two tiles $\til$ and $\til\cdot g$ for $g\notin\nuke$. It exists, since any such tiles are disjoint.

Define $\phi_n\colon \alb^n\arr\R$ by $\phi_n(v)=\diam(\til\otimes v)$. Consider an edge-path $\gamma=((s_1, v_1), (s_2, v_2), \ldots, (s_m, v_m))$ in $\Gamma_n$. Consider its lift to the graph of $\bim^{\otimes n}$ starting in $v_1$. It passes through the vertices $u_1=v_1, u_2=s_1\cdot u_1, u_3=s_2\cdot u_2, \ldots, u_{m+1}=s_m\cdot u_m$. We have $u_{m+1}=(s_m\cdots s_1(v_1))\cdot (s_m\cdots s_1)|_{v_1}$. In particular, $\til\otimes v_1\subset\til$ and $\til\otimes u_m\subset\til\cdot (s_m\cdots s_1)|_{v_1}$. It follows that distance between any two points $\xi\in\til\otimes v_1$ and $\zeta\in\til\otimes u_{m+1}$ is more than $r_0$, if $(s_m\cdots s_1)|_{v_1}\notin\nuke$.

On the other hand, since $u_i=v_i\cdot (s_{i-1}\cdots s_1)|_{v_1}$, we have $\diam(\til\otimes u_i)=\diam(\til\otimes v_i)=\phi_n(v_i)$, so
 $d(\xi, \zeta)\le\sum_{i=1}^{m+1}\diam(\til\otimes u_i)=\sum_{i=1}^{m+1}\phi(v_i)$. 
 
 It follows that $r_0^{-1}\phi_n\colon \alb^n\arr\R$ is an admissible function. Let $p>\alpha$. Then $\sum_{v\in\xs}(\phi_n(v))^p\le\max_{v\in\alb^n}(\phi_n(v))^{p-\alpha}\sum_{v\in\xs}(\phi_n(v))^\alpha\le C_1\max_{v\in\alb^n}(\diam(\til\otimes v))^{p-\alpha}\sum_{v\in\xs}\mu(\til\otimes v)$ for some constant $C_1$. We have $\sum_{v\in\xs}\mu(\til\otimes v)\le|\nuke|\mu(\til)$, since the sets $\til\otimes v$ for $v\in\alb^n$ form a cover of $\til$ of multiplicity at most $|\nuke|$. Since $\max_{v\in\alb^n}\diam(\til\otimes v)\le C_2\lambda^n$ for some $C_2>1$ and $\lambda_0\in (0, 1)$, see Proposition~\ref{pr:diamoftiles}, we conclude that $M_n(p)\le C_3(\lambda_0^{p-\alpha})^n$ for all $n$. Consequently, $\theta_p\le\lambda_0^{p-\alpha}<1$.

Let us prove the second statement of the theorem. The idea is, given $\theta_p<1$,  to construct  a $G$-quasi-invariant quasi-morphism such that the associated visual metric of exponent $1$ exists and has Haudorff dimension equal to an arbitrary number greater than $p$, using the following proposition.

\begin{proposition}
\label{prop:exponetiated}
Let $\nu\colon\xs\arr [0, \infty)$ be a $G$-quasi-invariant quasi-morphism. 

Suppose that for every path $(s_{m-1}, x_{m-1}v_{m-1}), \ldots, (s_2, x_2v_2), (s_1, x_1v_1)$ in $\G_n$, where $x_i\in\alb$ and $v_i\in\alb^{n-1}$, such that the path $(s_{m-1}\cdots s_2s_1)|_{x_1}\ne\nuke$, we have
\[e^{-\nu(x_1v_1)}+e^{-\nu(x_2v_2)}+\cdots+e^{-\nu(x_{m-2}v_{m-2})}>e^{-\nu(v_1)}\]
and
\[e^{-\nu(x_2v_2)}+e^{-\nu(x_3v_3)}+\cdots+e^{-\nu(x_{m-1}v_{m-1})}>e^{-\nu(v_i)}\]
for all $i=2, 3, \ldots, m-1$.

Then there exists a visual metric $d_\nu(\xi_1, \xi_2)\asymp e^{-\ell_\nu(\xi_1, \xi_2)}$ on $\lims$ of exponent $1$.
\end{proposition}

\begin{proof}
Let us denote $\delta(v)=e^{-\nu(v)}$.

We say that a sequence $v_1, v_2, \ldots, v_m\in\xs$ is a \emph{chain} if $\til_{v_i}\cap\til_{v_{i+1}}$ for all $i=1, \ldots, m-1$. Its \emph{mass} is the sum $\delta(v_1)+\delta(v_2)+\cdots+\delta(v_m)$. Its \emph{length} is $m$. We say that the chain \emph{connects} points $\xi_1, \xi_2\in\lims$ if $\xi_1\in\til_{v_1}$ and $\xi_2\in\til_{v_m}$.

Define $d(\xi_1, \xi_2)$ as the infimum of masses of all chains connecting $\xi_1$ to $\xi_2$. The function $d(\xi_1, \xi_2)$ is symmetric, satisfies the triangle ineaquality, and $d(\xi, \xi)=0$ for all $\xi\in\lims$. Therefore, the proof of the proposition will be complete, if we show that $d(\xi_1, \xi_2)\asymp e^{-\ell_\nu(\xi_1, \xi_2)}$.

We may consider, in the definition of $d(\xi_1, \xi_2)$, only chains without repetitions, since otherwise, we can remove a segment of the chain between the repetitions, and get a chain of smaller mass and length, connecting the same points.

Let $u=x_k\ldots x_2x_1$ and $v=y_l\ldots y_2y_1$ for $k\ge l$. 
Suppose that $x_1v_1, x_2v_2, \ldots, x_mv_m\in\alb^n$ is a chain consisting of words of equal lengths, where $x_i\in\alb$ and $v_i\in\alb^{n-1}$. Let $s_i\in\nuke$ be such that $s_i(x_iv_i)=x_{i+1}v_{i+1}$. We get a path $(s_{m-1}, x_{m-1}v_{m-1}), \ldots, (s_2, x_2v_2), (s_1, x_1v_1)$ in $\G_n$.

Similarly to the proof of Lemma~\ref{lem:productadmissible}, find a sequence $k_0=1<k_1<k_2<\ldots<k_l\le m$ such that each path $(s_{k_{i+1}-1}, x_{k_{i+1}-1}), \ldots, (s_{k_{i+1}}, x_{k_i+1}), (s_{k_i}, x_{k_i})$ from $x_{k_i}$ to $x_{k_{i+1}}$ in $\G_1$ is almost essential but not essential. The last segment connecting $x_{k_l}$ to $x_m$ is not almost essential, if $k_l<m$.

By the condition of the proposition, we have
\[\delta(x_{k_i}v_{k_i})+\delta(x_{k_i+1}v_{k_i+1})+\cdots+\delta(x_{k_{i+1}-1}v_{k_{i+1}-1})>\delta(v_{k_i}).\]

Note that the path from $x_{k_i}$ to $x_{k_{i+1}}$ is not essential (if it is, then the path obtained by removing the last vertex is almost essential), i.e., 
$(s_{k_{i+1}}\cdots s_{k_i+1}s_{k_i})|_{x_{k_i}}\in\nuke$. We have $(s_{k_{i+1}}\cdots s_{k_i+1}s_{k_i})|_{x_{k_i}}(v_{k_i})=v_{k_{i+1}}$, hence $\til_{v_{k_i}}$ and $\til_{v_{k_{i+1}}}$ are adjacent. Note that $\til_{v_{k_i}}\supset \til_{x_{k_i}v_{k_i}}$ and $\til_{v_{k_{i+1}}}\supset \til_{x_{k_{i+1}}v_{k_{i+1}}}$

We see that if we replace in the the chain $x_1v_1, \ldots, x_mv_m$ the initial segment $x_1v_1, \ldots, x_{k_{l-1}}v_{k_{l-1}}$ by $v_1, v_{k_1}, \ldots, v_{k_{l-1}}$, we get a chain of smaller mass and length. All paths $(s_j, x_j), \ldots, (s_{k_{l-1}+1}, x_{k_{l-1}+1}, (s_{k_{l-1}}, x_{k_{l-1}})$ for $k_{l-1}\le j<k_l$, and all paths $(s_j, x_j), \ldots, (s_{k_l+1}, x_{k_l+1}), (s_{k_l}, x_{k_l})$ for $k_l\le j<m$ are not essential.
It follows that $x_jv_j\in\alb\cdot \nuke(v_{k_{l-1}})$ for $k_{l-1}\le j<k_l$ and $x_jv_j\in\alb\cdot\nuke(v_{k_l})$ for $k_l\le j<m$.

Since we assume that the chain has no repetitions, the length of the segment $x_{k_{l-1}+1}v_{k_{l-1}+1}, \ldots, x_mv_m$ is less than $2|\alb|\cdot|\nuke|$.
Consequently, we may consider only chains such that the length of its segments consisting of words of constant length is at most $2|\alb|\cdot|\nuke|$.

Suppose now that we have a chain $v_0, x_1v_1, \ldots, x_mv_m, v_{m+1}$, where $|v_1|=|v_2|=\ldots=|v_m|=n$ and $|v_0|, |v_{m+1}|\le n$. Then there exist elements $s_0, s_1, \ldots, s_m\in\nuke$ such that $s_i(x_iv_i)=x_{i+1}v_{i+1}$ for $i=1, 2, \ldots, m$, the word $v_{m+1}$ is a suffix of $s_m(x_mv_m)$, and $v_0$ is a suffix of $s_0^{-1}(x_1v_1)$, see Lemma~\ref{lem:intersectingtiles}. Let $x_0, x_{m+1}\in\alb$ and $v_0', v_{m+1}'\in\alb^n$ be such that $x_0v_0'=s_0^{-1}(x_1v_1)$ and $x_{m+1}v_{m+1}'=s_m(x_mv_m)$. Then $v_0$ and $v_{m+1}$ are suffixes of $v_0'$ and $v_{m+1}'$, respectively.

If the path $(s_m, x_m), \ldots (s_1, x_1), (s_0, x_0)$ is not essential, then $s(v_0')=v_{m+1}'$ for $s=(s_m\cdots s_1s_0)|_{x_0}\in\nuke$, hence $\til_{v_0}\cap \til_{v_{m+1}}\ne\emptyset$, so we can delete the segment $x_1v_1, \ldots, x_mv_m$ from the chain, and get a shorter chain of smaller mass, connecting the same points.

Suppose that the path $(s_m, x_m), \ldots, (s_1, x_1), (s_0, x_0)$ is essential. Find its shortest almost essential beginning $(s_{k-1}, x_{k-1}), \ldots, (s_1, x_1), (s_0, x_0)$ connecting $x_0$ to $x_k$ for $k\le m$. Let $(s, x_k), (s_{k-1}, x_{k-1}), \ldots, (s_1, x_1), (s_0, x_0)$ be its essential continuation by one edge. Then, by the condition of the proposition, we have 
\[\delta(v_k)<\delta(x_1v_1)+\delta(x_2v_2)+\cdots+\delta(x_kv_k).\]
Since the path from $x_0$ to $x_k$ is not essential, the tiles $\til_{v_0'}$ and $\til_{v_k}$ intersect, henc $\til_{v_0}$ and $\til_{v_k}$ intersect. We can replace the segment $x_1v_1, \ldots, x_kv_k$ by $v_k$, and get a shorter chain of smaller mass. 

Thus, we can shorten the chain to get a chain of smaller mass in both cases. 
It follows that $d(\xi_1, \xi_2)$ is equal to the infimum of the masses of chains connecting $\xi_1$ to $\xi_2$ such that the segments of the chain consisting of words of constant length have at most $2|\alb|\cdot|\nuke|$ elements and 
are not adjacent to two shorter words.

Suppose that $v_1, v_2, \ldots, v_m$ is such a chain. Let $l$ be an index such that $|v_l|$ is minimal. Then $|v_i|$ is non-increasing for $i\le l$ and non-decreasing for $i\ge l$. It follows that there exist $s_i, r_i\in\nuke$ and $u_i\in\xs$ such that if  $v_{i-1}=s_i(u_iv_i)$ for all $i\le l$, and $v_{i+1}=r_i(u_iv_i)$ for all $i\ge l$. Therefore,
\begin{align*}v_1&=s_2\left(u_2 s_3\left(u_3\ldots s_l\left(u_lv_l\right)\ldots \right)\right)\\
v_m&=r_{m-1}\left(u_{m-1}r_{m-2}\left(u_{m-2}\ldots r_l\left(u_lv_l\right)\ldots\right)\right),
\end{align*}
where no more than $2|\alb|\cdot|\nuke|$ of consecutive words in the sequence $u_i$ can be empty.

Applying Proposition~\ref{pr:sumgeometric} to the set $A=\nuke^{2|\alb|\cdot|\nuke|}$, we find a finite set $\nuke_A$, depending only on $(G, \bim)$ and $\alb$ such that the above equalities imply
\[v_1=w_1 h_1(v_l),\qquad v_m=w_2 h_2(v_l),\]
for some $w_i\in\xs$ and $h_1, h_2\in\nuke_A$. 
The element $h_2h_1^{-1}$ maps the suffix $h_1(v_l)$ of $v_1$ to the suffix $h_2(v_l)$ of $v_m$.

Let $t$ be such that $(\nuke_A\cdot \nuke_A^{-1})|_{\alb^t}\subset\nuke$, and let $v_1', v_m'$ be the suffixes of $v_1, v_m$, respectively, of length $|v_l|-t$. Then there is an element of the nucleus mapping $v_1'$ to $v_m'$. We have $\til_{v_1'}\supset \til_{v_1}$, $\til_{v_m'}\supset\til_{v_m}$, and $\til_{v_1'}\cap\til_{v_2'}\ne\emptyset$. We get a chain connecting $\xi_1$ to $\xi_2$ of mass $\delta(v_1')+\delta(v_2')$. Since $\nu$ is a $G$-quasi-invariant quasi-morphism, and $v_i'$ are images of suffixes of length $|v_l|-t$ of $v_l$ under elements of $\nuke$, there exists a constant $C>0$ such that $\delta(v_i')\le C_1\delta(v_l)$. 

Since $\xi_1\in\til_{v_1}\subset\til_{v_1'}$ and $\xi_2\in\til_{v_m}\subset\til_{v_m'}$, the points $\xi_1, \xi_2$ can be represented by left-infinite sequences with suffixes $v_1'$ and $v_m'$, respectively. Since they are images of each other under the action of elements of the nucleus, $\ell_\nu(\xi_1, \xi_2)\ge\nu(v_1')-C$ for some constant $C_2\ge 0$.

The mass of the original chain $v_1, \ldots, v_m$ is greater or equal to $\delta(v_l)\ge C_1^{-1}(\delta(v_1')+\delta(v_2'))\ge C_1\delta(v_1')\ge C_3e^{-\ell_\nu(\xi_1, \xi_2)}$, which finishes the proof.
\end{proof}

After passing to an iterate $(G, \bim^{\otimes n})$, we may assume that the basis $\alb$ is such that $(\nuke^2)|_{\alb}=\nuke$.

Consider the class $\mathcal{B}$  of functions $\phi\colon\alb^n\arr [0, \infty)$ satisfying the following condition. For every essential path in $\G_{n, \nuke}$ passing through vertices $v_1, v_2, \ldots, v_m$, we have
\begin{equation}
\label{eq:Lhcond1}
\hat\phi(v_1)+\hat\phi(v_2)+\cdots+\hat\phi(v_{m-2})>1
\end{equation}
and
\begin{equation}
\label{eq:Lhcond2}
\hat\phi(v_2)+\hat\phi(v_3)+\cdots+\hat\phi(v_{m-1})>1,
\end{equation}
where, as before, $\hat\phi(v)=\min_{s\in\nuke}\phi(s(v))$.

By the same arguments as in the proof of Lemma~\ref{lem:hatequivalent}, we have $M_{\mathcal{B}, p}(n)\asymp M_{\mathcal{A}, p}(n)$ for the class $\mathcal{A}$ of admissible functions.

Assume that $\theta_p<1$. Then $M_{\mathcal{B}, p}(n)\to 0$ as $n\to\infty$.
Let $\lambda_0$ be a small positive number, which we will choose later. Let $n_0$ be such that $M_{\mathcal{B}, p}(n)<\lambda_0$. Let us pass to  $\alb^{n_0}$, so that we have $M_{\mathcal{B}, p}(1)<\lambda_0$. 
Let $\phi\in\mathcal{B}$ be such that $\sum_{x\in\alb}\phi(x)^p<\lambda_0$.

Note that it follows that $\phi(x)\in [0, 1)$ for every $x\in\alb$. By adding a small constant $\epsilon$ to $\phi$, we may assume that $0<\eta_-\le\phi(x)\le\eta_+<1$ for some constants $\eta_-, \eta_+$ and $\sum_{x\in\alb}\phi(x)^p\le 2\lambda_0<1$. (Provided $\lambda_0<1/2$.)

Consider the metric on the self-similarity complex $\Sigma$ defined by the following lengths of the generators of $\bim$:
\begin{itemize}
    \item $\|x\|=-\log\phi(x)$ for $x\in\alb$,
    \item $\|s\|=M$ for $s\in\nuke$, where $M>2\max\{\|x\|\colon x\in\alb\}$.
\end{itemize}
Let $\Sigma_\phi$ be the obtained metric graph.


The original lengths of paths and distances in $\Sigma$ will be called \emph{combinatorial}.

\begin{lemma}
\label{lem:geodesicsinSigmaphi}
A geodesic path in $\Sigma_\phi$ can not have more than two consecutive horizontal edges, and it has all descending edges appearing before all the ascending ones
\end{lemma}

\begin{proof}
Consider a horizontal path of length $kM$ starting in a vertex $xv\in\xs$ for $x\in\alb$ and corresponding to the word $s_k\cdots s_2s_1$ of elements of $\nuke$. Then $v$ and $(s_k\cdots s_2s_1)|_x(v)$ can be connected by a horizontal path $\gamma$ of combinatorial length at most $\lceil k/2\rceil$. If $k>1$, then $\lceil k/2\rceil\le k-1$. Then the path consisting of the descending edge $(xv, v)$, followed by the path $\gamma$ and then by the ascending edge from $(s_k\cdots s_2s_1)|_x(v)$ to $(s_k\cdots s_2s_1)(v)$ has length less than $M/2+(k-1)M+M/2=kM$, so $\gamma$ is not a geodesic path in $\Sigma_\phi$. This shows that a geodesic path in $\Sigma_\phi$ can not have two consecutive horizontal edges.

A path consisting of an ascending edge $(v, xv)$, horizontal edge $(xv, s(xv))$, and a descending edge $(s(xv), s|_x(v))$ has length $\|x\|+M+\|s(x)\|>M$. But there is a horizontal edge from $v$ to $s|_x(v)$ (or these vertices coincide), which shows that we can not have an ascending edge before a descending edge in a geodesic in $\Sigma_\phi$.
\end{proof}

Denote by $\nu(v)$, for $v\in\xs$, be the distance in $\Sigma_\phi$ from $v$ to $\emptyset$.

\begin{lemma}
    The function $\nu\colon \xs\arr (0, \infty)$ is a positive $G$-quasi-invariant quasi-morphism.
\end{lemma}

\begin{proof}
Positivity of $\nu$ follows from the fact that $\Sigma_\phi$ is quasi-isometric to $\Sigma$. For every $g\in\nuke$ we have $|\nu(v)-\nu(g(v))|\le\|s\|=M$, by the triangle inequality, hence $\nu$ is $G$-quasi-invariant.

For every $v\in\xs$, the geodesic paths from $\emptyset$ to $v$ in $\Sigma_\phi$ contains only ascending or horizontal edges, since, by Lemma~\ref{lem:geodesicsinSigmaphi}, in any geodesic path all descending edges come before the ascending ones, but there are no descending or horizontal edges starting in $\emptyset$.

Let $\gamma_1$, $\gamma_2$, and $\gamma$ be the geodesic paths starting in $\emptyset$ and ending in $v_1, v_2, v_1v_2$ for some $v_1, v_2\in\xs$. Since all edges of these paths are ascending or horizontal, the labels read along them belong to $\alb\cup\nuke$, and no consecutive labels belong to $\nuke$. Let $\overline w_1, \overline w_2, \overline w$ be the corresponding words in $\alb\cup\nuke$, and let $w_1, w_2, w\in\bim^*$ be the corresponding products of the generators of the semigroup $\bim^*$. We obtain $w_1(\emptyset)=v_1, w_2(\emptyset)=v_2, w(\emptyset)=v_1v_2$. Since $\nuke^2|_\alpha=\nuke$ and no two consecutive generators in the wors $\overline w_i, \overline w$ belong to $\nuke$, we get that $w_1=v_1\cdot h_1, w_2=v_2\cdot h_2$, and $w=v_1v_2\cdot h_3$ for some $h_i\in\nuke^2$.

Let $\gamma_1'$ be the path starting in $h_1^{-1}(v_2)$ and labeled by the same labels as $\gamma_1$. Then it will end in $w_1h_1^{-1}(v_2)=v_1h_1h_1^{-1}(v_2)=v_1v_2$. Since the labels of $\gamma_1'$ and $\gamma_1$ are the same, there lengths are equal. Consequently,
\[\nu(v_1v_2)\le d(v_1v_2, h_1^{-1}(v_2))+d(h_1^{-1}(v_2), v_2)+d(v_2, \emptyset)\le \nu(v_1)+\nu(v_2)+2M.\]

Let $w_1'$ be the beginning of $\overline w$ (seen as an element of $\bim^*$) such that $w_1'\in\bim^{\otimes |v_1|}$. Let $w_2'$ be the rest of $w$, so that $w=w_1'w_2'$. Since $w=v_1v_2\cdot h_3$, and there are no two consecutive elements of $\nuke$ in $\overline w$, we have, for some $h_1'\in\nuke^2$, that $w_1'=v_1\cdot h_1'$. We have $v_1v_2\cdot h_3=v_1\cdot h_1'\cdot (h_1')^{-1}v_2\cdot h_3$. Consequently, $v_2'=(h_1')^{-1}v_2$ is a vertex of $\gamma$ splitting it into a path from $\emptyset$ to $v_2'$ labeled by $w_2'=(h_1')^{-1}v_2\cdot h_3$ and the path from $v_2'$ to $v_1v_2$ labeled by $w_1'=v_1\cdot h_1'$. The length of the first path is greater or equal to the distance from $\emptyset$ to $v_2$ minus $|v_2-(h_1')^{-1}(v_2)|\le 2M$. The length of the second path is the same as the length of the path with the same labels connecting $\emptyset$ to $w_1'(\emptyset)=v_1$, which is greater or equal to $\nu(v_1)$. Consequently,
\[\nu(v_1v_2)\ge \nu(v_2)-2M+\nu(v_1),\]
which finishes the proof that $\nu$ is a quasi-morphism.
\end{proof}

It follows from Lemma~\ref{lem:geodesicsinSigmaphi} that a geodesic path in $\Sigma_\phi$ from $v\in\xs$ to $\emptyset$ either starts with a descending edge, or with a horizontal edge followed by a descending edge. Consequently, for every $xv\in\xs$ there exists $s\in\nuke$ such that the word $yu=s(xv)$ satisfies     
\[\nu(yu)=-\log(\phi(y))+\nu(u).\]
Note that it follows from the triangle inequality that
\[\nu(xv)\le-\log(\phi(x))+\nu(v)\]
for all $x\in\alb$ and $v\in\xs$.

Denote $\delta(v)=e^{-\nu(v)}$. Since $\nu$ is $G$-quasi-invariant, we have $\delta\asymp\hat\delta$.
Let us show that $\hat\delta$ satisfies the conditions of Proposition~\ref{prop:exponetiated}, which will show that there exists a metric $d$ on $\lims$ such that $d(\xi_1, \xi_2)\asymp e^{-\ell_\nu(\xi_1, \xi_2)}$. Consider a path \[(s_{m-1}, x_{m-1}v_{m-1}), \ldots, (s_2, x_2v_2), (s_1, x_1v_1)\] in $\G_{n, \nuke}$ connecting $x_1v_1$ to $x_mv_m$, where $x_i\in\alb$ and $v_i\in\alb^{n-1}$, such that the path $(s_{m-1}\cdots s_2s_1)|_{x_1}\notin\nuke$, as in the conditions of Proposition~\ref{prop:exponetiated}.

Let us prove the first inequality
\[\hat\delta(x_1v_1)+\hat\delta(x_2v_2)+\cdots+\hat\delta(x_{m-2}v_{m-2})>\hat\delta(v_1).\]

Let $(s_{k-1}, x_{k-1}), \ldots, (s_2, x_2), (s_1, x_1)$
be the shortest almost essential beginning of the path. It connects $x_1$ to $x_k$. We have $k<m$, since the path for $k=m-1$ is almost essential.

Let $1\le i\le k-1$, and let $h_i\in\nuke$ be such that $\hat\delta(x_iv_i)=\delta(h_i(x_iv_i))$. Since the beginning of the path connecting $x_1$ to $x_iv_i$ is not almost essential, the section $s=(h_is_i\cdots s_1)|_{x_1}$ belongs to the nucleus, hence we have
\[\hat\delta(x_iv_i)=\delta(h_i(x_iv_i))\ge\phi(h_i(x_i))\delta(s(v_1))\ge\hat\phi(x_i)\hat\delta(v_1).\]

Consider a continuation \[(s, x_k),  (s_{k-1}, x_{k-1}), \ldots, (s_2, x_2), (s_1, x_1)\] of the path to an essential path.

Then $\phi\in\mathcal{B}$ implies that
\begin{multline*}
\hat\delta(x_1v_1)+\hat\delta(x_2v_2)+\cdots+\hat\delta(x_{m-2}v_{m-2})\ge \\
\hat\delta(x_1v_1)+\hat\delta(x_2v_2)+\cdots+\hat\delta(x_{k-1}v_{k-1})\ge\\
\left(\hat\phi(x_1)+\hat\phi(v_2)+\cdots+\hat\phi(x_{k-1})\right)\hat\delta(v_1)>\delta(v_1).
\end{multline*}

Let us prove the second inequality of Proposition~\ref{prop:exponetiated}. Choose $j\in\{2, 3, \ldots, m-1\}$. If there is an almost essential path
\[(s_j, x_j), (s_{j+1}, x_{j+1}), \ldots\]
or an almost essential path
\[(s_{j-1}^{-1}, x_j), (s_{j-2}^{-1}, x_{j-1}), \ldots\] going in the opposite direction, then the same argument as above will finish the proof. If all such paths are not almost essential, then the same arguments as above show that $\hat\delta(x_iv_i)\ge\hat\phi(x_i)\hat\delta(v_j)$ for all $j$, hence
\begin{multline*}
\hat\delta(x_2v_2)+\hat\delta(x_3v_3)+\cdots+\hat\delta(x_{m-1}v_{m-1})\ge\\
\left(\hat\phi(x_2)+\hat\phi(x_3)+\cdots+\hat\phi(x_{m-1})\right)\hat\delta(x_j)>\delta(x_j).
\end{multline*}

Consequently, the function $\hat\delta$ satisfies the conditions of Proposition~\ref{prop:exponetiated}, hence there exists a visual metric associated with $\nu$ of exponent $1$.

\begin{lemma}
For any $x\in\alb$ and $v\in\xs$ there exists $s\in\nuke$ such that 
\[\nu(xv)-\nu(v)\ge -\log\phi(s(x)).\]
\end{lemma}

\begin{proof}
If the geodesic path from $xv$ to $\emptyset$  starts with the vertical edge $(xv, v)$, then
\[\nu(xv)=-\log\phi(x)+\nu(v).\]

Suppose that it starts with a horizontal edge $(xv, s(xv))$ for some $s$. Since we can not have two consecutive horizontal edges, the next edge of the geodesic path is the vertical edge $(s(xv), s|_x(v))$, hence
\[\nu(xv)=M-\log\phi(s(x))+\nu(s|_x(v)).\]
We have a horizontal edge $(v, s|_x(v))$, hence, by triangle inequality,
\[\nu(v)\le M+\nu(s|_x(v)).\]
It follows that
\[\nu(xv)-\nu(v)\ge -\log\phi(s(x)),\]
which finishes the proof.
\end{proof}

It follows that for every $v\in\xs$ we have
\begin{multline*}
\sum_{x\in\alb}e^{-p(\nu(xv)-\nu(v))}\le\sum_{x\in\alb}e^{p\log\phi(s_x(x))}\\
=\sum_{x\in\alb}\phi(s_x(x))^p\le|\nuke|\sum_{x\in\alb}\phi(x)^p<2|\nuke|\lambda_0,
\end{multline*}
for some $s_x\in\nuke$.

By Proposition~\ref{pr:entropycriterion}, the $\nu$-entropy is less than or equal to $p$, if $\lambda_0<2|\nuke|^{-1}$. Since there is a visual metric associated with $\nu$ of exponent $1$, we get the inequality $\cdim(G, \bim)\le p$, finishing the proof.\end{proof}

Theorem~\ref{th:modulus} and Proposition~\ref{pr:asympmodulus} imply the following.

\begin{corollary}
    We have
    \[\cdim (G, \bim)=\inf\{p\colon \theta_p<1\}\] and
    \[\cdim(G, \bim)\le \left(\frac 1p-\frac{\log\theta_p}{\log \mathsf{D}}\right)^{-1}<p\]
for every $p\in (1, \infty]$ such that $\theta_p<1$, where $\mathsf{D}=|\bim/G|$.
\end{corollary}



\section{Contraction coefficients}

\subsection{Definition}

Let $(G, \bim)$ be an infinite finitely generated (not necessarily faithful) self-similar group. Denote by $l(g)$ the length of an element of $G$ with respect to a fixed generating set. Let $\alb$ be a basis of $\bim$.

For every $p\in (0, \infty)$, define
\[\eta_p(G, \bim)=\lim_{n\to\infty}\left(\limsup_{l(g)\to\infty}\frac{\left(\sum_{v\in\alb^n}l(g|_v)^p\right)^{1/p}}{l(g)}\right)^{1/n}\]
and for $p=\infty$, define
\[\eta_\infty(G, \bim)=\lim_{n\to\infty}\left(\limsup_{l(g)\to\infty}\frac{\max_{v\in\alb^n}l(g|_v)}{l(g)}\right)^{1/n}.\]

If $G$ is not finitely generated, then we define $\eta_p(G, \bim)=\eta_p(\langle\nuke\rangle, \bim)$, where $\nuke$ is the nucleus. This does not change the definition for finitely generated groups.

The following properties of the numbers $\eta_p$ are proved in~\cite[Proposition~4.3.15]{nek:dyngroups}.

\begin{proposition}
The limit $\eta_p=\eta_p(G, \bim)$ exists and does not depend on the choices of the generating set of $G$ or of the basis $\alb$.  We have $\eta_\infty\ge\mathsf{D}^{-1}$. For all $0<p<q\le\infty$ we have
\[\eta_p\ge\eta_q,\qquad \mathsf{D}^{-1/p}\eta_p\le\mathsf{D}^{-1/q}\eta_q,\]
for $\mathsf{D}=|\bim/G|$.
\end{proposition}

Define
\[\ldim(G, \bim)=\inf\{p\colon\eta_p(G, \bim)<1\}.\]

\begin{proposition}
\label{pr:etatheta}
Let $(G, \bim)$ be a contracting (not necessarily faithful) group. Then we have
\[\eta_p(G, \bim)\le\theta_p(G, \bim)\]
for every $p\in (1, \infty]$, and therefore,
\[\ldim(G, \bim)\le\cdim(G, \bim).\]
\end{proposition}

\begin{proof}
Recall that we call a path in the graph $\G_n$ almost essential if it can be continued by one edge to an essential path. We say that a function is \emph{strongly admissible} if the sum of its values along any almost essential path is greater than 1.

Let $\phi_0\colon S \times\alb^n\arr[0, \infty)$ be an $(S, \nuke)$-strongly admissible function, where $S$ is a generating set of $G$.  We assume that $\nuke\subset S$.

 Consider a product $g=s_1s_2\cdots s_m$ of elements of $S$ such that $m=l(g)$. Consider the corresponding path in $\G_n$ starting in $v$. It has length $\phi(g, v)=\sum_{i=1}^m \phi(s_i, s_{i+1}\cdots s_m(v))$. Split it into maximal almost essential segments (the last segment is non-essential, but not necessarily almost essential). Suppose that there are $k$ of them. Then $l(g|_v)\le k+1$. Since the $\phi$-length of each of the segments is $\ge 1$, we get that $\phi$-length of the path is at least $k\ge l(g|_v)-1$. 

Then, using the triangle inequality for the $\ell^p$-norms of the vectors
\[\vec F_i=\left(\phi(s_i, s_{i+1}\cdots s_m(v))\right)_{v\in\alb^n}\in\R^{\alb^n}\] and $\vec U=(1)_{v\in\alb^n}$, we get
\begin{multline*}\left(\sum_{v\in\alb^n}l(g|_v)^p\right)^{1/p} \le\left(\sum_{v\in\alb^n}(\phi(g, v)+1)^p\right)^{1/p}=\\
\left(\sum_{v\in\alb^n}\left(\sum_{i=1}^m\phi(s_i, s_{i+1}\cdots s_m(v))+1\right)^p\right)^{1/p}=\\
\|\vec F_1+\cdots+\vec F_m+\vec U\|_p\le\|\vec F_1\|_p+\cdot\|\vec F_m\|_p+|\alb|^{n/p}=\\
\sum_{i=1}^m\left(\sum_{v\in\alb^n}\phi(s_i, v)^p\right)^{1/p}+|\alb|^{n/p}\le l(g)\cdot\left(\sum_{v\in\alb^n, s\in S}\phi(s, v)^p\right)^{1/p}+|\alb|^{n/p}.
\end{multline*}

By varying $\phi_0$, we get:
\[\limsup_{l(g)\to\infty}\frac{\left(\sum_{v\in\alb^n}l(g|_v)^p\right)^{1/p}}{l(g)}\le CM_p(n)^{1/p},\]
where $C$ does not depend on $n$. Taking the $n$th root and passing to the limit, we get $\eta_p\le\theta_p$.
\end{proof}







\subsection{Portraits of group elements}

Denote, for $g\in G$, by $\mathcal{L}(g)$ the set of words $v\in\xs$ such that $g|_v\notin\nuke$. It is a finite prefix-closed subset of $\xs$.

\begin{proposition}
\label{pr:sizeofportrait}
Let $p>\ldim(G, \bim)$. Then there exists a constant $C>0$ such that for every $g\in G$ we have 
\[|\mathcal{L}(g)|\le Cl(g)^p.\]
\end{proposition}

\begin{proof}
We have $\eta_p<1$. Let $\lambda\in (\eta_p, 1)$. Then there exist $R, n_0>1$ such that for all $g\in G$ such that $l(g)\ge R$ we have 
\[\sum_{v\in\alb^{n_0}}l(g|_v)^p<\lambda^{n_0}l(g)^p.\]
Let $\mathcal{L}_{n_0}(A)$ be the intersection of $\mathcal{L}(g)$ with $(\alb^{n_0})^*$. Denote by, for $v\in\alb^*$ by $v'$ the longest prefix of $v$ of length divisible by $n_0$. Then $|v|-|v'|<n_0$, so the map $v\mapsto v'$ is $K_1$-to-one for some fixed $K_1$. The set $\mathcal{L}(g)$ is prefix-closed, hence $|\mathcal{L}(g)|\le K_1 |\mathcal{L}_{n_0}(g)|$.

Let $\mathcal{L}'_{n_0}(g)$ be the set of elements $v\in(\alb^{n_0})^*$ such that $l(g|_v)\ge R$. Since there exists $k$ such that $g|_v\in\nuke$ for all $v\in\alb^k$ and $g\in G$ such that $l(g)<R$, there exists $K_2$ such that  $|\mathcal{L}_{n_0}(g)|\le K_2|\mathcal{L}'_{n_0}(g)|$. Then $|\mathcal{L}(g)|\le K_1K_2|\mathcal{L}_{n_0}'(g)|$.

We have
\[\sum_{v\in\mathcal{L}'_{n_0}(g)}l(g|_v)^p\le l(g)^p+\lambda^{n_0} l(g)^p+\lambda^{2n_0}l(g)^p+\cdots=\frac{l(g)^p}{1-\lambda^{n_0}}. \]

It follows that
\[|\mathcal{L}(g)|\le K_1K_2|\mathcal{L}_{n_0}'(g)|\le\frac{K_1K_2}{1-\lambda^{n_0}} l(g)^p,\]
which finishes the proof.
\end{proof}

\begin{defi}
The \emph{portrait of an element $g\in G$} is the set $\mathcal{L}(g)$ together with the map $\alpha\colon\mathcal{L}(g)\arr S_{\alb}$ mapping $v$ to the permutation induced by $g|_v$ on $\alb$, and the map $\beta\colon vx\mapsto g|_{vx}\in\nuke$ for all leaves $v$ of the subtree $\mathcal{L}(g)\subset\xs$ and all $x\in\alb$.
\end{defi}

The portrait uniquely determines the action of $g$ on $\xo$ by the rule
\[g(x_1x_2\ldots )=\alpha(\emptyset)(x_1) \alpha(x_1)(x_2) \alpha(x_1x_2)(x_3)\ldots \beta(x_1x_2\ldots x_n)(x_{n+1}x_{n+2}\ldots),\]
where $x_1x_2\ldots x_{n-1}$ is the longest prefix of $x_1x_2\ldots$ belonging to $\mathcal{L}(g)$.

An easy combinatorial counting argument proves the following corollary of Proposition~\ref{pr:sizeofportrait}.

\begin{corollary}
The word growth of a contracting group $(G, \bim)$ is bounded, for every $\epsilon>0$, above by $\exp(C n^{\ldim(G, \bim)+\epsilon})$ for some $C>0$.
\end{corollary}

Note that this corollary is trivial if $\ldim(G, \bim)\ge 1$, but it shows that groups satisfying $\ldim(G, \bim)<1$ have subexponential growth.

\begin{example}
Consider the Grigorchuk group $G=\langle a, b, c, d\rangle$ from~\cite{grigorchuk:80_en}. 
Take the generating set $\{a, b, c, d, ab, ac, ad\}$. Let $\eta$ be the real root of the polynomial $x^3+x^2+x-2$, and define the lengths of the generators, as in~\cite{bartholdi:upperbd}:
\begin{alignat*}{2}
\|a\|&=1-\eta^3, &\qquad \|b\|&=\eta^3,\\
  \|c\|&=1-\eta^2, &\qquad   \|d\|&=1-\eta.
\end{alignat*}
Then 
\begin{align*}
\|ab\|&=\|a\|+\|b\|=1,\\
\|ac\|&=\|a\|+\|c\|=\eta,\\
\|ad\|&=\|a\|+\|d\|=\eta^2.
\end{align*}
We have $ab=\sigma(a, c)$, $ac=\sigma(a, d)$, $ad=\sigma(1, b)$. We see that for each of them we have $\eta\cdot l(g)=l(g|_0)+l(g|_1)$. It follows that $\eta_1(G)\le\eta$.

Consider the substitution
\[\psi\colon ba\mapsto cadaba,\quad ca\mapsto baba,\quad da\mapsto caca,\]
see~\cite{bartholdiErschler:imbeddings}.

The wreath recursion satisfies $\psi(g)=\sigma (ga, ag)$ or $\psi(g)=(aga, g)$, and we have $l(\psi^n(ad))\asymp (2/\eta)^n$, where $\eta\approx 0.8015$ is the root of $x^3+x^2+x-2=0$.

We get \[\sum_{v\in\alb^k}l(\psi^n(ad)|_v)^p\asymp 2^k(2/\eta)^{-kp}l(\psi^{n-k}(ad))^p,\] which implies \[\eta_p\le 2^{\frac 1p-1}\eta.\]
Note that since $\eta_\infty=\theta_\infty=1/2$, we can not have equality for large $p$.

On the other hand, it is known that $\eta_1\le\eta$, see~\cite{bartholdi:upperbd}. Consequently, $\eta_p\le 2^{\frac 1p-1}\eta$ for all $p\in (0, 1]$. In particular, we get $\ldim(G, \bim)=\frac{\log 2}{\log 2-\log\eta}\approx 0.7674$.

We see that $\eta_p=2^{\frac 1p-1}\eta$ for all $p\in (0, 1]$ and $2^{\frac 1p-1}\eta\le\eta_p\le 2^{\frac 1p-1}$ for all $p\in (0, \infty]$. We know $\theta_p=2^{\frac 1p-1}$ in this case (see Example~\ref{ex:Grigmodulus}).
\end{example}

\subsection{Thurston's $p$-map}
\label{ss:thurstonpmap}
The following approach to \emph{lower} estimates of $\ldim(G, \bim)$ (and hence for $\cdim(G, \bim)$) is related to Thurston’s theorem (see~\cite{DH:Thurston}) characterizing post-critically finite
branched self-coverings of the sphere that are realizable as complex rational functions. 

Let $\mathcal{C}$ be the set of conjugacy classes of elements of infinite order in $G$ and consider the vector space $\bigoplus_{[g]\in\mathcal{C}}\R\cdot[g]$ with the basis $\mathcal{C}$. For an element $[g]\in\mathcal{C}$ represented by an element $g\in G$, consider all the cycles $c=(x_1, \ldots, x_k)$ of the action of $g$ on $\alb$. For each of them, the conjugacy class $[g_c]$ of $g^k|_{x_i}$ does not depend on $i$. Define $T_p([g])$ as the sum of the vectors $k^{1-p}[g_c]$ over all cycles $c$ of the action of $g$ on $\alb$. The operator does not depend on the choice of $\alb\subset\bim$. It is also easy to check that the operator associated with $\bim^{\otimes n}$ is equal to $T_p^n$.

The following proposition is proved in~\cite[Theorem~4.10]{mattebonnekzheng:liouvills}.

\begin{proposition}
Let $A\subset\mathcal{C}$ be a finite set, and let $P_A$ be the projector
\[P_A([g])=\left\{\begin{array}{ll} [g] & \text{if $[g]\in A$,}\\
0 & \text{otherwise}\end{array}\right.\] onto the span $\R^A$ of $A$.

Then the spectral radius of the operator $P_AT_p\colon\R^A\arr\R^A$ is less than or equal to $\eta_p^p$.
\end{proposition}





\begin{example}
Consider the group $\img(z^2-1)$ computed in Example~\ref{ex:basilicaimgcomput}. The generators $a$ and $b$ have infinite order, since $a^2(0w)=0b(w)$, $a^2(1w)=1b(w)$ and $b(1w)=1a(w)$. These relations also imply that
\[T_p([a])=2^{1-p}[b],\qquad T_p([b])=[a],\]
i.e., that its matrix is $\left(\begin{array}{cc} 0 & 1\\ 2^{1-p} & 0\end{array}\right)$. The spectral radius of this matrix is $2^{\frac{1-p}2}$. We get hence an estimate $\sqrt{2}^{\frac 1p-1}\le\eta_p$. On the other hand, we know from Example~\ref{ex:thetapbasilica} that $\theta_p=\sqrt{2}^{\frac 1p-1}$. Since $\eta_p\le\theta_p$, we conclude that $\eta_p=\theta_p=\sqrt{2}^{\frac 1p-1}$ for all $p\ge 1$. 
\end{example}

\section{Algebras $\mathcal{D}_p$}
\label{subsec:algebras} 

\subsection{Definitions}

Denote by $C(\lims)$ the $C^*$-algebra of continuous functions $\lims\arr\C$, and let $C_0(\limg)$ be the $C^*$-algebra of continuous functions $\limg\arr\C$ vanishing at infinity.

Note that $C(\lims)$ is naturally identified with the algebra of continuous $G$-invariant functions $\limg\arr\C$. Namely, a function $f\colon \lims\arr\C$ is identified by its lift to $\limg$ by the quotient map $\limg\arr\lims$. We will often use this identifications and treat elements of $C(\lims)$ as $G$-invariant continuous functions on $\limg$.

\begin{defi}
\label{def:aplim}
For $p\ge 1$, let $\mathcal{D}_p(\limg)\subset C_0(\limg)$ be the set of functions $f\colon \limg\arr\C$ such that
\[\sum_{v\in\bim^*}|f(\xi_1\otimes v)-f(\xi_2\otimes v)|^p<\infty\]
for any two points $\xi_1, \xi_2\in\limg$.

For $p\ge 1$, let $\mathcal{D}_p(\lims)\subset C(\lims)$ be the set of functions $f$ such that their lift $f\colon \limg\arr\C$ satisfies
\[\sum_{v\in\xs}|f(\xi_1\otimes v)-f(\xi_2\otimes v)|^p<\infty\]
for all $\xi_1, \xi_2\in\limg$.

For $p>1$ and $\lambda\ge 0$, let $\mathcal{D}_{p, \lambda}(\limg)\subset C_0(\limg)$ be the set of functions $f\in C_0(\limg)$ such that for every pair $\xi_1, \xi_2\in\limg$ there exists $C>1$ such that
\[\sum_{v\in\bim^{\otimes n}}|f(\xi_1\otimes v)-f(\xi_2\otimes v)|^p\le C\lambda^n\]
for all $n\ge 1$.

For $p>1$ and $\lambda\ge 0$, let $\mathcal{D}_{p, \lambda}(\limg)\subset C(\lims)$ be the set of functions $f$ such that for every pair $\xi_1, \xi_2\in\limg$ there exists $C>1$ such that
\[\sum_{v\in\alb^{\otimes n}}|f(\xi_1\otimes v)-f(\xi_2\otimes v)|^p\le C\lambda^n\]
for all $n\ge 1$. 
\end{defi}

We obviously have $\mathcal{D}_{p, \lambda}(\lims)\subset\mathcal{D}_p(\lims)$ and $\mathcal{D}_{p, \lambda}(\limg)\subset\mathcal{D}_p(\limg)$.

The algebras in Definition~\ref{def:aplim} are invariant under composition of their elements with a Lipschitz function $\C\arr\C$.

\begin{proposition}
The sets $\mathcal{D}_p(\limg), \mathcal{D}_p(\lims), \mathcal{D}_{p, \lambda}(\limg), \mathcal{D}_{p, \lambda}(\lims)$ are $*$-subalgebras of the corresponding  $C^*$-algebras $C_0(\limg)$ and $C(\lims)$.

In particular, these sub-algebras are dense if and only if they separate the points (of $\limg$ and $\lims$, respectively).
\end{proposition}

\begin{proof}
Let us prove the statement for the set $\mathcal{D}_p(\limg)$. The proof of the other statement is analogous. The set $\mathcal{D}_p(\limg)$ is obviously $*$-invariant.

Denote by $\partial_{\xi_1, \xi_2}(f)$ the map $\bim\arr\C$ given by $\partial_{\xi_1, \xi_2}(f)(v)=f(\xi_1\otimes v)-f(\xi_2\otimes v)$. Then the sum in the definition of $\mathcal{D}_p(\limg)$ is equal to the $p$-th power of the $\ell^p$-norm  $\|\partial_{\xi_1, \xi_2}(f)\|_p$.

For any $f_1, f_2\in C_0(\limg)$, we have 
\[\partial_{\xi_1, \xi_2}(f_1+f_2)=\partial_{\xi_1, \xi_2}(f_1)+\partial_{\xi_1, \xi_2}(f_2),\] hence $\|\partial_{\xi_1, \xi_2}(f_1+f_2)\|_p\le\|\partial_{\xi_1, \xi_2}(f_1)\|_p+\|\partial_{\xi_1, \xi_2}(f_2)\|_p$, so $\mathcal{D}_p(\limg)$ is closed under addition.

We also have 
\begin{multline*}\partial_{\xi_1, \xi_2}(f_1f_2)(v)=f_1(\xi_1\otimes v)f_2(\xi_1\otimes v)-f_1(\xi_2\otimes v)f_2(\xi_2\otimes v)=\\ (f_1(\xi_1\otimes v)-f_1(\xi_2\otimes v))f_2(\xi_1\otimes v)+f_1(\xi_2\otimes v)(f_2(\xi_1\otimes v)-f_2(\xi_2\otimes v)),
\end{multline*} hence
\[\|\partial_{\xi_1, \xi_2}(f_1f_2)\|_p\le\|f_2\|\cdot\|\partial_{\xi_1, \xi_2}(f_1)\|_p+\|f_1\|\cdot\|\partial_{\xi_1, \xi_2}(f_2)\|_p,\]
where $\|f\|=\max_{\zeta\in\limg}|f(\zeta)|$ is the $C^*$-norm on $C_0(\limg)$. It follows that $\mathcal{D}_p(\limg)$ is closed under multiplication. Its $*$-invariance is obvious.
\end{proof}

\subsection{Density of the algebras}

\begin{proposition}
The algebra $\mathcal{D}_p(\lims)$ is dense in $C(\lims)$ if and only if $\mathcal{D}_p(\limg)$ is dense in $C_0(\limg)$. The algebra $\mathcal{D}_{p, \lambda}(\lims)$ is dense in $C(\lims)$ if and only if $\mathcal{D}_{p, \lambda}(\limg)$ is dense in $C_0(\limg)$.
\end{proposition}

\begin{proof}
We identify continuous functions on $\lims$ with continuous $G$-invariant functions on $\limg$ throughout the proof.

  Suppose that the algebra $\mathcal{D}_p(\lims)$ is dense. Let us show that the algebra $\mathcal{D}_p(\limg)$ separates the points.  Let $\zeta_1, \zeta_2\in\limg$ be different points. Since the action of $G$ on $\limg$ is proper, there exists a compact neighborhood $U$ of $\zeta_1$ in $\limg$ not containing $\zeta_2$ such that for every $g\in G$ we have either $U\cdot g=U$ and $\zeta_1\cdot g=\zeta_1$ or $U\cdot g\cap U=\emptyset$. Moreover, the set of elements $g\in G$ satisfying the first condition is equal to the stabilizer $G_{\zeta_1}$ of $\zeta_1$, and hence is finite.

Since $\mathcal{D}_p(\lims)$ is dense, there exits a $G$-invariant function $\tilde f\in\mathcal{D}_p(\lims)$ such that $\tilde f(\zeta_1)=1$ and $|\tilde f(\zeta)|\le 1/2$ for $\zeta\notin U$. Composing $\tilde f$ with a Lipschitz function $\C\arr\C$ equal to $1$ on a neighborhood of $1$ and to 0 for all $z\in\C$ such that $|z|\le 1/2$, we will get an element $f\in\mathcal{D}_p(\lims)$ supported in the interior of $U$ and such that $f(\zeta_1)=1$.

Since the sets $U\cdot g$ are disjoint or coincide, and $f$ is supported it their union, the function
\[h(\zeta)=\left\{\begin{array}{ll} f(\zeta) & \text{if $\zeta\in U$,}\\
0 & \text{otherwise}\end{array}\right.\] is continuous. It is compactly supported, hence is contained in $C_0(\limg)$.
  
Let $\xi_1, \xi_2\in\limg$ be arbitrary points. The distance between $\xi_1\otimes v$ and $\xi_2\otimes v$ uniformly on $v$ converges to 0 as the length of $v$ goes to infinity.

Consequently, there exists $n$ such that for every $v\in\bim^*$ with $|v|>n$ the following two conditions hold.  If $\xi_1\otimes v$ belongs to the support of $h$, then $\xi_2\otimes v\in U$, and if $\xi_2\otimes v$ belongs to the support of $h$, then $\xi_1\otimes v\in U$. 

Then, for all $v\in\bim^*$, $|v|>n$, we have either $h(\xi_1\otimes v)=h(\xi_2\otimes v)=0$, or $h(\xi_1\otimes v)=f(\xi_1\otimes v)$ and $h(\xi_2\otimes v)=f(\xi_2\otimes v)$. Note that the latter case happens only for at most $|G_{\zeta_1}|$ points in the right $G$-orbit of a given $v$.

It follows that $\sum_{v\in\bim^*, |v|\ge n}|h(\xi_1\otimes v)-h(\xi_2\otimes v)|^p\le |G_{\zeta_1}|\sum_{v\in\xs, |v|\ge n}|f(\xi_1\otimes v)-f(\xi_2\otimes v)|^p$, which implies that $h\in\mathcal{D}_p(\limg)$ if $f\in\mathcal{D}_p(\lims)$. We have $h(\zeta_1)=1$ and $h(\zeta_2)=0$, so we have shown that $\mathcal{D}_p(\limg)$ separates the points of $\limg$.
  
Suppose now that $\mathcal{D}_p(\limg)$ is dense. 
Let $\zeta_1, \zeta_2$ be two points of $\limg$ that have different images in $\lims$. Let $U$ be defined for $\zeta_1$ in the same way as above with additional condition that $\zeta_2\notin U\cdot g$ for every $g\in G$. We can achieve this by choosing $U$ small enough, as the images of $\zeta_1$ and $\zeta_2$ in $\lims=\limg/G$ are different, and the action of $G$ on $\limg$ is proper. 

By the same argument as above, there exists a real-valued function $h\in\mathcal{D}_p(\limg)$ supported in $U$ and such that $h(\zeta_1)=1$.

Define then $f(\zeta)=|G_{\zeta_1}|^{-1}\sum_{g\in G}h(\zeta\cdot g)$. Then $f$ is a $G$-invariant continuous function supported in $U\cdot G$ and satisfying $f(\zeta_1)=1$, $f(\zeta_2)=0$. It follows directly from the condition $h\in\mathcal{D}_p(\limg)$ that $f\in\mathcal{D}_p(\lims)$.
\end{proof}

We obviously have $\mathcal{D}_{p_1}(\lims)\subseteq\mathcal{D}_{p_2}(\lims)$ for $p_1<p_2$. Consequently, if $\mathcal{D}_{p_1}(\lims)$ is dense, then so is $\mathcal{D}_{p_2}(\lims)$. 

\begin{theorem}
\label{th:DpDlambda}
The conformal dimension $\cdim(G, \bim)$ is equal to the infimum of values of $p$ for which $\mathcal{D}_p(\lims)$ is dense in $C(\lims)$.

Moreover, for every $p>\cdim(G, \bim)$ the infimum of the values $\lambda$ for which $\mathcal{D}_{p, \lambda}(\lims)$ is dense in $C(\lims)$ is equal to $\theta_p$.
\end{theorem}

Denote by $p_s$ the infimum of $p$ for which $\mathcal{D}_p(\lims)$ is dense in $C(\lims)$. We will prove stronger statements than both inequalities $p_s\le\cdim(G, \bim)$ and $p_s\ge\cdim(G, \bim)$.

\begin{proposition}
If there exists a metric $d\in\mathcal{AR}(G, \bim)$ of Hausdorff dimension $p_0$, then for every $p>p_0$ and $\lambda\in (\theta_p, 1)$ the algebra $\mathcal{D}_{p, \lambda}(\limg)$ is dense in $C_0(\limg)$.
\end{proposition}

Since $\mathcal{D}_{p, \lambda}(\limg)\subset\mathcal{D}_p(\limg)$, the density of $\mathcal{D}_{p, \lambda}(\limg)$ implies the density of $\mathcal{D}_p(\limg)$. Therefore, the above proposition implies $p_s\le\cdim(G, \bim)$.

\begin{proof}
We may assume that $d$ is a visual metric associated with a $G$-quasi-invariant quasi-morphism $(\alb^n)^*\arr[0, \infty)$, by Theorem~\ref{th:modulus}.

Let us show that there exists $\lambda\in (\theta_p, 1)$ such that every function $f\colon \limg\arr\C$ which is Lipschitz with respect to the metric $d$ belongs to $\mathcal{D}_{p, \lambda}$.

Let $L$ be the Lipschitz constant for $f$. Let $\xi_1, \xi_2\in\limg$. We will denote by $P(\xi)$ the image in $\lims$ of a point $\xi\in\limg$.

There exists $n_0$ such that $\xi_i\otimes v\in\til\cdot\nuke$ for all $v\in\xs$, $|v|\ge n_0$, $i=1, 2$. Then, for all $v_1\in\alb^{n_0}$ and $v_2\in\xs$ we have $P(\xi_i\otimes v_1v_2)\in\til_{s_i(v_2)}$ for some $s_i\in\nuke$. It follows that $d(P(\xi_1\otimes v_1v_2)-P(\xi_2\otimes v_1v_2))$ is not greater than diameter of the set $\bigcup_{s\in\nuke}\til_{s(v_2)}$. Since each set $\til_{s(v)}$ intersects $\til_v$, the diameter of $\bigcup_{s\in\nuke}\til_{s(v_2)}$ is less than a constant times the diameter of $\til_{v_2}$, since $d$ is associated with a $G$-quasi-invariant quasi-morphism.

Using Proposition~\ref{pr:diamoftiles}, we see that for some constants $C_1, C_2, C_3$, and $\lambda\in (0, 1)$, not depending on $n$, $v_1\in\alb^{n_0}$ and $v_2\in\alb^n$, we have
\begin{multline*}
|f(P(\xi_1\otimes v_1v_2))-f(P(\xi_2\otimes v_1v_2))|^p\le\\
 L^pd(P(\xi_1\otimes v), P(\xi_2\otimes v))^p\le C_1\diam(\til_{v_2})^p\le \\ C_2\diam(\til_{v_2})^{p-p_0}\mu(\til_{v_2})\le C_3\lambda^{n(p-p_0)}\mu(\til_{v_2}),
 \end{multline*} where $\mu$ is the corresponding measure on $\lims$.
 
 It follows that for all $n\ge n_0$ we have
 \begin{multline*}\sum_{v\in\alb^n}|f(P(\xi_1\otimes v))-f(P(\xi_2\otimes v))|^p\le\\ 
 (\deg\bim)^{n_0}C_3\lambda^{(n-n_0)(p-p_0)}\sum_{v_2\in\alb^{n-n_0}}\mu(\til_{v_2})\le\\
  C_4\lambda_1^n |\nuke|\mu(\lims)=C_5\lambda_1^n
  \end{multline*}
  for $\lambda_1=\lambda^{p-p_0}$. Consequently, $f\in\mathcal{D}_{p, \lambda_1}$.
\end{proof}

\begin{proposition}
\label{pr:pcgreater}
Let $\xi\in\til$. Consider the algebra $\widehat{\mathcal{D}}_{p, \xi}$  consisting of functions $f\in C_0(\limg)$ such that $\lim_{n\to\infty}\sum_{v\in\bim^{\otimes n}}|f(\xi\otimes v)-f(\xi\otimes g\cdot v)|^p=0$ for all $g\in G$.

If $\widehat{\mathcal{D}}_{p, \xi}$ is dense in $C_0(\limg)$, then there exists a metric $d\in\mathcal{AR}(G, \bim)$ of dimension $p$.
\end{proposition}

It is easy to see that $\mathcal{D}_p(\limg)\subset\widehat{\mathcal{D}}_{p, \xi}$, so density of $\mathcal{D}_p(\limg)$ implies density of $\widehat{\mathcal{D}}_{p, \xi}$. Consequently, the above proposition implies the inequality $p_c\ge\cdim(G, \bim)$.

\begin{proof}
Choose a basis $\alb$ and let $\nuke$ be the corresponding nucleus.
Suppose that $\widehat{\mathcal{D}}_{p, \xi}$ is dense. This algebra is also invariant under post-compositions with Lipschitz functions, so there exists a function $f\in\widehat{\mathcal{D}}_{p, \xi}$ such that $f(\xi)=1$ for every $\xi\in \til$ and $f(\xi)=0$ for every $\xi\in\til\cdot g$ for $g\in G\setminus \nuke$.

For every $n$, define then the function $\phi$ on the edges of the graph $\Gamma_n$ by \[\phi(s, v)=\left(\sum_{g\in G}|f(\xi\otimes v\cdot g)-f(\xi\otimes s\cdot v\cdot g)|^p\right)^{1/p}.\]
Then $\sum_{v\in\alb^n}\phi(s, v)^p=\sum_{v\cdot g\in\bim^{\otimes n}}|f(\xi\otimes v\cdot g)-f(\xi\otimes s\cdot v\cdot g)|^p\to 0$, by the definition of $\widehat{\mathcal{D}}_{p, \xi}(\limg)$.

Let $(s_1, v_1), (s_2, v_2), \ldots, (s_m, v_m)$ be an edge-path in $\Gamma_n$. Here $v_{i+1}=s_i(v_i)$. Suppose that $(s_m\cdots s_2s_1)|_{v_1}\notin\nuke$. Consider the corresponding sequence 
\[\xi\otimes v_1, \xi\otimes s_1\cdot v_1, \xi\otimes s_2s_1\cdot v_1, \ldots, \xi\otimes s_m\cdots s_2s_1\cdot v_1.\]
Then $\xi\otimes v_1\in\til$, but 
\[\xi\otimes s_m\cdots s_2s_1\cdot v_1=\xi\otimes s_m(v_m)\cdot (s_m\cdots s_2s_1)|_{v_1}\in\til\cdot (s_m\cdots s_2s_1)|_{v_1}.\] It follows that $f(\xi\otimes v_1)=1$ and $f(\xi\otimes s_m\cdots s_2s_1\cdot v_1)=0$. Consequently,

\begin{multline*}|f(\xi\otimes v_1)-f(\xi\otimes s_1\cdot v_1)|+|f(\xi\otimes s_2s_1\cdot v_1)-f(\xi\otimes s_3s_2s_1\cdot v_1)|+\cdots\\ +|f(\xi\otimes s_{m-1}\cdots s_1\cdot v_1)-f(\xi\otimes s_m\cdots s_1\cdot v_1)|\ge 1.
\end{multline*}
We have $\phi(s_i, v_i)=\left(\sum_{g\in G}|f(\xi\otimes v_i\cdot g)-f(\xi\otimes s_i\cdot v_i\cdot g)|^p\right)^{1/p}$. One of the summands in this sum is $|f(\xi\otimes v_i\cdot (s_{i-1}\cdots s_2s_1)|_{v_1})-f(\xi\otimes s_i\cdot v_i\cdot (s_{i-1}\cdots s_2s_1)|_{v_1})|^p=|f(\xi\otimes (s_{i-1}\cdots s_2s_1)\cdot v_1)-f(\xi\otimes (s_i\cdots s_2s_1)\cdot v_1)|^p$.
Consequently, $\phi(s_i, v_i)\ge |f(\xi\otimes (s_{i-1}\cdots s_2s_1)\cdot v_1)-f(\xi\otimes (s_i\cdots s_2s_1)\cdot v_1)|$, which implies that $\sum_{i=1}^m\phi(s_i, v_i)\ge 1$, i.e., that $\phi$ is admissible.

We conclude that $M_n(p)\to 0$ as $n\to\infty$, so by Theorem~\ref{th:modulus}, there exists a metric $d\in\mathcal{AR}(G, \bim)$ of dimension $p$.
\end{proof}

\begin{example}
Consider the Julia set $\mathcal{J}$ of $f(z)=z^2-1$, i.e., the limit space of $\img(z^2-1)$. There is a topological conjugacy between the action of $f^2$ on the closure of the Fatou component containing $0$ with the action of $z^2$ on the closed unit disc. The restriction of this conjugacy to the boundary of the Fatou component can be also seen as the map induced by the morphism from the binary adding machine biset to the $\langle a\rangle$-biset generated by $\{11, 01\}\subset\bim_{z^2-1}^{\otimes 2}$.

Consider an arbitrary Lipschitz function on the unit circle, and transfer it to a function $\varphi$ on the boundary $B_0$ of the Fatou component of $0$ by the conjugacy. The closure of every component $C$ of $\mathcal{J}\setminus B_0$ intersects $B_0$ in a single point, so we can extend $\varphi$ to $C$ as a constant function, thus getting a continuous function on $\mathcal{J}$. The boundary $B$ of any other Fatou component of $\mathcal{J}$ is mapped onto $B_0$ homeomorphically by some iteration of $f$, so we can use that iteration to define in a similar way a copy of $\varphi$ on $B$ which is constant on the components of $\mathcal{J}\setminus B$. One can show that all such functions belong to $\mathcal{D}_{1, 1}$, and hence to $\mathcal{D}_p$ for all $p>1$. It is also easy to see that such functions generate a dense sub-algebra in $C(\mathcal{J})$. It follows that the conformal dimension of the limit space of $\img(z^2-1)$ is equal to 1.
\end{example}

\section{Uniform dimension}

Recall that the visual metrics associated with the morphism $\nu(v)=|v|$ are called \emph{uniform}. 

\begin{defi}
    Let $(G, \bim)$ be a contracting group. The \emph{uniform dimension} $\udim(G, \bim)$ is the infimum of the Hausdorff dimensions of the uniform metrics on $\lims$.
\end{defi}

We will need the following algebraic interpretation of the  contracting coefficient $\theta_\infty$.

\begin{proposition}
Let $(G, \bim)$ be a contracting group. Choose a basis $\alb\subset\bim$, and let $\nuke$ be the associated nucleus.

Denote by $L(n)$ the largest number for which $\nuke^{L(n)}|_v\subset\nuke$ for all $v\in\alb^n$. 

Then \[\lim_{n\to\infty}L(n)^{1/n}=\sup_n L(n)^{1/n}=\theta_\infty^{-1}.\]
\end{proposition}

\begin{proof}
The constant function equal to $L(n)^{-1}$ on the edges of $\G_n$ is admissible, since any essential path has length $\ge L(n)+1$. Consequently, $\overline M_\infty(n)\le L(n)^{-1}$. 

On the other hand, suppose that $\phi$ is an admissible function on the set of edges of $\G_n$. There exists an essential path of length $L(n)+1$. Since the sum of values of $\phi$ on the edges of this path is not less than $1$, there exists an edge $e$ such that $\phi(e)\ge (L(n)+1)^{-1}$.
Consequently, $\|\phi\|_\infty\ge (L(n)+1)^{-1}$. 

We have shown that 
\[(L(n)+1)^{-1}\le \overline M_\infty(n)\le L(n)^{-1}.\]
Consequently, $\lim_{n\to\infty}L(n)^{1/n}=\theta_\infty^{-1}$. 

Consider a product of length $L(n)L(m)$ of elements of $\nuke$, write it as $g_1g_2\ldots g_{L(m)}$ for $g_i\in\nuke^{L(n)}$ and let $w=v_1v_2$ for $v_1\in\alb^{n}$ and $v_2\in\alb^m$. Then $g_1g_2\ldots g_{L(m)}|_{v_1}=g_1|_{g_2\cdots g_{L(m)}(v_1)}\cdots g_{L(m)}|_{v_1}\in\nuke^{L(m)}$, so $(g_1\cdots g_{L(m)})|_{v_1v_2}\in\nuke$. It follows that $L(n+m)\ge L(n)L(m)$.
Consequently, by Fekete's Lemma, $\lim_{n\to\infty}L(n)^{1/n}=\sup_n L(n)^{1/n}$.
\end{proof}

\begin{proposition}
For every contracting group $(G, \bim)$, we have  \[\udim(G, \bim)=\frac{\log\mathsf{D}}{-\log\theta_\infty},\]
where $\mathsf{D}=|\bim/G|$.
\end{proposition}

\begin{proof}
Suppose that $\xi_1, \xi_2, \ldots, \xi_m$ is a sequence of points of $\lims$ such that $\ell(\xi_i, \xi_{i+1})\ge n$. If $v_1, v_2, \ldots, v_m$ are suffixes of length $n$ of sequences representing points $\xi_i$, then $v_{i+1}=s_i(v_i)$ for some $s\in\nuke$. If $u_1, u_m$ are the suffixes of length $k+1$ of words representing $\xi_1, \xi_m$, then there is no element $s\in\nuke$ such that $s(u_1)=u_m$. In particular, the section of $s_{m-1}\cdots s_2s_1$ in the suffix of length $n-k$ of $u_1$ does not belong to the nucleus. It follows that $m-1>L(n-k)$. Let $\lambda>\theta_\infty$. Then $L(n)\ge\lambda^{-n}$ for all $n$ large enough, so $m>\lambda^{-(n-k)}$ for all $n-k\ge m_0$ for some fixed $m_0$. If $n-k<m_0$, then $\lambda^{-(n-k)}<\lambda^{-m_0}$, so $m>1=\lambda^{m_0}\lambda{-m_0}>\lambda^{m_0}\lambda^{-(n-k)}$. Hence $m>C\lambda^{-(n-k)}$ for all $n\ge k$, where $C=\lambda^{m_0}$.

We have shown that $-\log\lambda$ is a lower exponent for the log-scale $\ell(\cdot, \cdot)$, see Definition~\ref{def:lowerexponentforlog}. It follows that there exists a uniform metric $d(\xi_1, \xi_2)\asymp e^{\log\lambda\cdot \ell(\xi_1, \xi_2)}$ of Hausdorff dimension $\frac{\mathsf{D}}{-\log\lambda}$. Since $\lambda$ was an arbitrary number greater than $\theta_\infty$, we get \[\udim(G, \bim)\le\frac{\log\mathsf{D}}{-\log\lambda}.\]

In the other direction, suppose that $d$ is a uniform  metric on $\lims$ of dimension $p$. Let $\delta$ be the distance in $\limg$ (with respect to the lifted metric) from the tile $\til$ to the union of the tiles $\til\cdot g$ for $g\in G\setminus\nuke$. By definition of $L(n)$, there exists a sequence $s_1, s_2, \ldots, s_{L(n)+1}$ of elements of the nucleus and a word $v_0\in\alb^n$ such that $g=(s_{L(n)+1}\cdots s_2s_1)|_{v_0}\notin\nuke$. Consider the elements $v_1=s_1\cdot v_0, v_2=s_2\cdot v_1, \ldots v_{L(n)+1}=s_{L(n)+1}\cdot v_{L(n)}$ of $\bim^{\otimes n}$. Then $\til\otimes v_i\cap \til\otimes v_{i+1}\ne\emptyset$ for all $i=0, \ldots, L(n)$, and  $\til\otimes v_{L(n)+1}\subset \til\cdot g$ for $g\notin\nuke$. Let $\xi_i\in\til\otimes v_i\cap \til\otimes v_{i+1}$. Then $d(\xi_i, \xi_{i+1})\le\diam(\til\otimes v_i)\le\diam(\til_{v_i})\asymp e^{-n\frac{\log\mathsf{D}}p}$

We have $\xi_0\in\til\otimes v_0\subset\til$ and $\xi_{L(n)}\in\til\otimes v_{L(n)+1}\subset\til\cdot g$ for $g\notin\nuke$. Consequently, $d(\xi_0, \xi_{L(n)})\ge\delta$, so 
\[\delta\le d(\xi_0, \xi_{L(n)})\le\sum_{i=0}^{L(n)-1}d(\xi_i, \xi_{i+1})\le CL(n)e^{-n\frac{\log\mathsf{D}}p},\]
hence $\log\delta\le \log C+\log L(n)-n\frac{\log\mathsf{D}}p$, hence
\[p\ge \frac{n\log\mathsf{D}}{\log C-\log\delta+\log L(n)}\to\frac{\log\mathsf{D}}{-\log\theta_\infty},\]
which finishes the proof.
\end{proof}

\begin{example}
    We have seen that, in the case of $\img(z^2-1)$, we have $\eta_p=\theta_p=2^{\frac{1-p}{2p}}$ for all $p\ge 1$. It follows that $\ldim(\img(z^2-1))=\cdim(\img(z^2-1))=1$, while $\udim(\img(z^2-1))=\frac{\log 2}{-\log 2^{1/2}}=2.$
\end{example}

\begin{proposition}
    For every contracting group $(G, \bim)$, we have $\eta_\infty=\theta_\infty$. 
\end{proposition}

\begin{proof}
Define the \emph{depth} $\rho(g)$ of $g\in G$ as the smallest $n$ such that $g|_v\in\nuke$ for all $v\in\alb^n$.

Let $\lambda>\eta_\infty$. Then there exists a finite set $A\subset G$ and $n_0$ such that $l(g|_v)\le\lambda^{n_0} l(g)$ for all $g\in G\setminus A$ and $v\in\alb^{n_0}$. This implies that for every $g\in G$ we have
\[\rho(g)\le\frac{\log l(g)}{-\log\lambda}+C,\]
where $C$ is any number greater than $\rho(g)$ for $g\in A$.

Recall that $L(n)$ was defined as the largest number such that $\nuke^{L(n)}|_v\subset\nuke$ for all $v\in\alb^n$. For every $n$ there exists $v_n\in\alb^n$ and $g_n\in\nuke^{L(n)+1}$ such that $g_n|_v\notin\nuke$. Then $\rho(g_n)\ge n$ and $l(g_n)\le L(n)+1$. Therefore,
\[n\le\rho(g_n)\le\frac{\log(L(n)+1)}{-\log\lambda}+C,\] 
hence
\[L(n)\ge \lambda^{-n+C}-1,\]
for all $n$, which implies $\theta_\infty=\lim_{n\to\infty}L(n)^{-1/n}\le\lambda$. Since $\lambda$ is an arbitrary number greater than $\eta_\infty$, we get $\theta_\infty\le\eta_\infty$. But we also proved in Proposition~\ref{pr:etatheta} that $\eta_\infty\le\theta_\infty$.
\end{proof}

It is proved in~\cite[Proposition~2.13.8]{nek:book} that $\frac{\log\mathsf{D}}{-\log\theta_\infty}+\epsilon$ is, for every $\epsilon>0$, an upper bound on the degree of polynomial growth of the orbits of the action of $G$ on $\xo$.

\begin{corollary}
\label{cor:udimcdim}
The following two conditions for a contracting group $(G, \bim)$ are equivalent:
\begin{enumerate}
\item $\udim(G, \bim)=\cdim(G, \bim)$;
\item $\theta_p=\mathsf{D}^{1/p}\theta_\infty$ for all $p\ge\cdim(G, \bim)$.
\end{enumerate}
\end{corollary}

\begin{proof}
We have $\mathsf{D}^{1/p}\theta_\infty\le 1$  if and only if $p>\frac{\log\mathsf{D}}{-\log\theta_\infty}$. Consequently, if $\theta_p=\mathsf{D}^{1/p}\theta_\infty$, then $\cdim(G, \bim)=\frac{\log\mathsf{D}}{-\log\theta_\infty}$.

On the other hand, we have $\theta_p\le\mathsf{D}^{1/p}\theta_\infty$ for all $(G, \bim)$ and all $p\in (0, \infty]$, by Proposition~\ref{pr:asympmodulus}. 

If there exists $p>\udim(G, \bim)$ such that $\theta_p<\mathsf{D}^{1/p}\theta_\infty$, then
\[\theta_\infty\theta_{\udim(G, \bim)}=\mathsf{D}^{-1/\udim(G, \bim)}\theta_{\udim(G, \bim)}\le \mathsf{D}^{-1/p}\theta_p<\theta_\infty,\]
hence $\theta_{\udim(G, \bim)}<1$, which implies that $\cdim(G, \bim)<\udim(G, \bim)$.
\end{proof}

\begin{proposition}
Let $(G, \bim)$ be a contracting self-similar group. The following conditions are equivalent.
\begin{enumerate}
\item $\ldim(G, \bim)=\udim(G, \bim)$.
\item $\eta_p=\mathsf{D}^{\frac 1p}\eta_\infty$ for all $p>\ldim(G, \bim)$.
\end{enumerate}
\end{proposition}

Note that since $\ldim(G, \bim)\le\cdim(G, \bim)\le\udim(G, \bim)$, the condition $\ldim(G, \bim)=\udim(G, \bim)$ implies $\cdim(G, \bim)=\udim(G, \bim)$.

\begin{proof}
The proof is the same as for Corollary~\ref{cor:udimcdim}.
\end{proof}

Let $(G, \bim)$ be a self-similar group. A subgroup $H\le G$ is called \emph{self-similar} if there exists a basis $\alb\subset\bim$ such that $h|_x\in H$ for all $h\in H$ and $x\in\alb$. Note that if $\nuke$ is the nucleus of $G$ with respect to $\alb$, then $\nuke\cap H$ is the nucleus of $H$, and the subgroup generated by $H\cap\nuke$ is self-similar. 

The identity map on $\xmo$ and the embedding $\xmo\times H\subset\xmo\times G$ induce a surjective semiconjugacy $\lims[H]\arr\lims[G]$ of the limit dynamical systems and an $H$-equivariant map $\limg[H]\arr\limg[G]$.

Let us describe some classes of groups for which we have coincidence of all three dimension $\ldim=\cdim=\udim$.

\begin{proposition}
Let $(G, \bim)$ be a level-transitive self-similar group, $\mathsf{D}=\deg\bim$, and let $H\le G$ be an infinite self-similar subgroup of $G$ contained in the center of $G$. Then for every $p>0$ we have
\[\mathsf{D}^{1/p}\eta_\infty(H)=\eta_p(H)\le\eta_p(G).\]

In particular, if $\eta_\infty(H)=\eta_\infty(G)$, then 
\[\eta_p(G)=\theta_p(G)=\mathsf{D}^{1/p}\eta_\infty(G)\]
 and $\cdim(G, \bim)=\udim(G, \bim)$.
\end{proposition}

\begin{proof}
The inequality $\eta_p(H)\le\eta_p(G)$ is obvious, so it is enough to prove the equality $\eta_p(H)=\mathsf{D}^{1/p}\eta_\infty(H)$.
Let $\alb\subset\bim$ be a basis such that $H$ is self-similar with respect to it. Take $n\ge 1$, choose $v_0\in\alb^n$, and choose elements $r_v, t_v\in G$ for all $v\in\alb^n$ such that $v=r_v\cdot v_0\cdot t_v$ for all $v\in\alb^n$.

Let $h\in H$. Then
\[h\cdot v=h\cdot r_v\cdot v_0\cdot t_v=r_v\cdot h(v)\cdot v_0\cdot h|_{v_0}\]
for all $v\in\alb^n$. It follows that $l(h|_v)\doteq_C l(h|_{v_0})$ for all $v\in\alb^n$, where $C$ does not depend on $h$ (but may depend on other choices). Consequently, using the triangle inequality for the $p$-norm, when $p\ge 1$, we get that the limit
\[\limsup_{l(h)\to\infty, h\in H}\frac{\left(\sum_{v\in\alb^n} l(h|_v)^p\right)^{1/p}}{l(h)}=\limsup_{l(h)\to\infty, h\in H}\frac{(\deg\bim)^{n/p}l(h|_{v_0})}{l(h)}\]
does not depend on $p$. For $p\in (0, 1)$ we get the same equality, using the triangle inequality for the sum of $p$th powers. We also get
\[\limsup_{l(h)\to\infty, h\in H}\frac{\max_{v\in\alb^n} l(h|_v)}{l(h)}=\limsup_{l(h)\to\infty, h\in H}\frac{l(h|_{v_0})}{l(h)}.\]
Taking the $n$-th root and passing to the limit as $n\to\infty$, we get $\eta_p(H)=(\deg\bim)^{1/p}\eta_\infty(H)$.
\end{proof}

\begin{corollary}
For every contracting level transitive abelian group $(G, \bim)$ we have
\[\ldim(G, \bim)=\cdim(G, \bim)=\udim(G, \bim).\]
\end{corollary}

Consider a level-transitive self-similar group $(\Z^n, \bim)$, and let $x\in\bim$. Then the stabilizer of $x$ in $\Z^n$ is a subgroup $H$ of index $\deg\bim$ in $\Z^n$, and the map $g\mapsto g|_x$ from $H$ to $\Z^n$ is a homomorphism. Let $A$ be the matrix of its linear extension to $\mathbb{Q}^n$. Then $A^{-1}$ has integer entries and $\det A=\deg\bim$. 

The group $(\Z^n, \bim)$ is contracting if and only if the spectral radius of $A$ is less than 1. In fact, the spectral radius $\rho(A)$ is equal to $\eta_\infty(\Z^n, \bim)$, see~\cite{neksid}.

It follows that \[\ldim(\Z^n, \bim)=\cdim(\Z^n, \bim)=\udim(\Z^n, \bim)=\frac{\log|\det A|}{\log\rho(A)}.\]

For example, for a solenoid of ``base $1.5$'' odometer, we get $\udim=\frac{\log 3}{\log 3-\log 2}$.

Let $H$ be a finitely generated self-similar subgroup of $G$. Let $l_H$ be the word length on $H$, and let $l_G$ be the word length on $G$. We may assume that the generating set of $H$ is contained in the generating set of $G$. Then $l_H(h)\ge l_G(h)$ for all $h\in H$.

\begin{defi}
We say that $H$ is \emph{undistorted} if there exists $C>1$ such that $l_H(h)\le Cl_G(h)$ for all $h\in H$.
\end{defi}

Since word lengths $l_H$ and $l_G$ restricted to an undistorted subgroup $H$ of $G$ are bi-Lipschitz equivalent, if we use the length $l_G$ in the definition of $\eta_p(H)$, we get the same number as for $l_H$. Consequently, we get the following.

\begin{lemma}
If $H\le G$ is an undistorted self-similar subgroup, then $\eta_p(H)\le\eta_p(G)$ for all $p\in (0, \infty]$.
\end{lemma}

\begin{theorem}
\label{th:undistorted}
Let $(G_0, \bim_0)$ be an undistorted self-similar subgroup of a contracting group $(G, \bim)$. If $\ldim(G_0, \bim_0)=\udim(G_0, \bim_0)$
and $\eta_\infty(G_0, \bim_0)=\eta_\infty(G, \bim)$, then 
  $\ldim(G, \bim)=\udim(G, \bim)$.
\end{theorem}

\begin{proof}
For $\mathsf{D}=|\bim/G|$ and every $p>1$, we have
\[\mathsf{D}^{1/p}\eta_\infty(G, \bim)\ge \eta_p(G, \bim)\ge\eta_p(G_0, \bim_0)=\mathsf{D}^{1/p}\eta_\infty(G_0, \bim_0)=\mathsf{D}^{1/p}\eta_\infty(G, \bim),\]
so we have equalities, which implies that $\udim(G, \bim)=\ldim(G, \bim)$.
\end{proof}

Let $(G_i, \bim_i)$, for $i=1, 2$, be contracting groups. Their direct product is the self-similar group $(G_1\times G_2, \bim_1\times\bim_2)$ defined in the natural way.

Then $G_1\times\{1\}$ is an undistorted self-similar subgroup of $G_1\times G_2$. 

\begin{proposition}
Let $(G_i, \bim_i)$, for $i=1, 2$, be contracting groups such that $\ldim(G_1, \bim_1)=\udim(G_1, \bim_1)$ and $\eta_\infty(G_1, \bim_1)\ge\eta_\infty(G_2, \bim_2)$. Denote $\mathsf{D}_i=|\bim_i/G_i|$. Then \[\ldim(G_1\times G_2, \bim_1\times\bim_2)=\udim(G_1\times G_2, \bim_1\times\bim_2)=\frac{\log(\mathsf{D}_1\mathsf{D}_2)}{-\log\eta_\infty(G_1, \bim_1)}.\]
\end{proposition}

\begin{proof}
It follows from the definitions that 
\[\eta_\infty(G_1\times G_2, \bim_1\times\bim_2)=\max(\eta_\infty(G_1, \bim_1), \eta_\infty(G_2, \bim_2)).\]
Consequently, if $\eta_\infty(G_1, \bim_1)\ge\eta_\infty(G_2, \bim_2)$, then we can apply Theorem~\ref{th:undistorted}.
\end{proof}

\providecommand{\bysame}{\leavevmode\hbox to3em{\hrulefill}\thinspace}
\providecommand{\MR}{\relax\ifhmode\unskip\space\fi MR }
\providecommand{\MRhref}[2]{%
  \href{http://www.ams.org/mathscinet-getitem?mr=#1}{#2}
}
\providecommand{\href}[2]{#2}

\end{document}